\documentclass[12pt, reqno]{NumPDEsArticle}

\usepackage{fullpage}
\usepackage{hyperref}
\usepackage{amsmath,amssymb,amsthm}
\usepackage{todonotes}
\usepackage{nameref}
\usepackage{biblatex}
\usepackage{booktabs}
\usepackage{bm}
\usepackage{moreverb}
\usepackage{mathtools, mathrsfs}
\usepackage[shortlabels]{enumitem}
\usepackage{tikz, diagbox}
\usepackage{tkz-euclide}
\usepackage[caption=false]{subfig}
\usetikzlibrary{shapes.geometric, calc, arrows, angles, matrix, colorbrewer}
\usepackage{xcolor}
\usepackage{nicefrac}
\usepackage{microtype}
\definecolor{TUblue}{RGB}{0,102,153}
\definecolor{TUGreen}{RGB}{0, 126, 113}
\definecolor{TUYellow}{RGB}{225, 137, 34}
\definecolor{TUred}{RGB}{252,25,25}

\definecolor{best}{HTML}{008000}
\definecolor{row}{RGB}{0,102,153}
\definecolor{col}{HTML}{ffff00}
\definecolor{darkgreen}{rgb}{0.01, 0.75, 0.24}
\definecolor{codegray}{rgb}{0.5,0.5,0.5}
\definecolor{pyBlue}{HTML}{1f77b4}
\definecolor{pyRed}{HTML}{d62728}
\definecolor{pyGreen}{HTML}{2ca02c}
\definecolor{pyOrange}{HTML}{ff7f0e}
\definecolor{pyPurple}{HTML}{9467bd}
\definecolor{pyYellow}{HTML}{bcbd22}
\definecolor{pyGrey}{HTML}{7f7f7f}
\usepackage{color, colortbl}
\usepackage{orcidlink}

\tikzset{reference/.style={thick,dashed}}

\usepackage{pgfplots}
\usepackage{pgfplotstable}
\pgfplotsset{%
	compat=newest,%
	every axis/.style={scale only axis},%
	grid style={densely dotted, semithick},%
}
\pgfdeclarelayer{bg}    
\pgfsetlayers{bg,main}  

\pgfplotsset{first/.style={
		x filter/.code={
			\ifnum\coordindex<#1\else\fi
		}
}}

\newcommand{\logLogSlope}[6][]{
	\pgfplotsextra{
		\pgfkeysgetvalue{/pgfplots/xmin}{\xmin}
		\pgfkeysgetvalue{/pgfplots/xmax}{\xmax}
		\pgfkeysgetvalue{/pgfplots/ymin}{\ymin}
		\pgfkeysgetvalue{/pgfplots/ymax}{\ymax}

		\pgfmathsetmacro{\xArel}{#2}
		\pgfmathsetmacro{\yArel}{#3}
		\pgfmathsetmacro{\xBrel}{#2+#4}

		\pgfmathsetmacro{\lnxA}{\xmin*(1-\xArel)+\xmax*\xArel} 
		\pgfmathsetmacro{\lnyA}{\ymin*(1-\yArel)+\ymax*\yArel} 
		\pgfmathsetmacro{\lnxB}{\xmin*(1-\xBrel)+\xmax*\xBrel} 
		\pgfmathsetmacro{\lnyB}{\lnyA+#5*(\lnxB-\lnxA)}
		\pgfmathsetmacro{\yBrel}{\lnyB-\ymin)/(\ymax-\ymin)}

		\coordinate (A) at (rel axis cs:\xArel,\yArel);
		\coordinate (B) at (rel axis cs:\xBrel,\yBrel);

		\draw[#1] (A)-- node[above=#6,anchor=center,sloped] {\small $r = #5$} (B);
	}
}

\usepackage[basicdelimiters]{NumPDEsMacros}
\newcommand{\fvec}{\boldsymbol{f}}
\newcommand{\bvec}{\boldsymbol{b}}
\newcommand{\gvec}{\boldsymbol{g}}
\newcommand{\Amat}{\boldsymbol{A}}
\newcommand{\Eta}{\mathrm{H}}
\newcommand{\Zeta}{\mathrm{Z}}

\newcommand{\exact}{\star}
\newcommand{\elll}{{\underline{\ell}}}
\newcommand{\mm}{{\underline{m}}}
\newcommand{\nn}{{\underline{n}}}
\newcommand{\kk}{{\underline{k}}}
\newcommand{\mmu}{{\underline{\mu}}}
\newcommand{\nnu}{{\underline{\nu}}}
\newcommand{\jj}{{\underline{j}}}

\newcommand\vvvert{|\mkern-1.5mu|\mkern-1.5mu|}

\newcommand{\lamalg}{\const{\lambda}{alg}}
\newcommand{\lamsym}{\const{\lambda}{sym}}
\newcommand{\qalg}{\const{q}{alg}}
\newcommand{\qsym}{q_{\mathrm{sym}}^{\exact}}
\newcommand{\qsymm}{\const{q}{sym}}

\newcommand{\thetamark}{\const{\theta}{mark}}

\newcommand{\Cmesh}{\const{C}{mesh}}

\DeclareFieldFormat{titlecase}{#1}

\hypersetup{
	colorlinks,
	citecolor = ForestGreen,
	linkcolor = BrickRed,
	urlcolor  = NavyBlue
}

\renewcommand{\d}[1]{\,\mathrm{d}#1}
\let\div\relax
\DeclareMathOperator{\div}{div}

\title{Optimal complexity of goal-oriented adaptive FEM for nonsymmetric linear
elliptic PDEs}

\author{Philipp Bringmann\orcidlink{0000-0002-4546-5165}}
\author{Maximilian Brunner\orcidlink{0000-0003-0636-1491}}
\author{Dirk Praetorius\orcidlink{0000-0002-1977-9830}}
\author{Julian Streitberger\orcidlink{0000-0003-1189-0611}}
\address{TU Wien, Institute of Analysis and Scientific Computing, Wiedner Hauptstr. 8-10/E101/4, 1040 Vienna, Austria}
\email{philipp.bringmann@asc.tuwien.ac.at}
\email{maximilian.brunner@asc.tuwien.ac.at}
\email{dirk.praetorius@asc.tuwien.ac.at}
\email{julian.streitberger@asc.tuwien.ac.at \quad \texttt{(corresponding author)}}



\thanks{This research was funded by the Austrian Science Fund (FWF) projects 
	\href{https://www.fwf.ac.at/en/research-radar/10.55776/F65}{10.55776/F65} (SFB F65 ``Taming 
	complexity in PDE systems''), 
	\href{https://www.fwf.ac.at/en/research-radar/10.55776/I6802}{10.55776/I6802} (international 
	project I6802 ``Functional error estimates for PDEs on unbounded domains''), and 
	\href{https://www.fwf.ac.at/en/research-radar/10.55776/P33216}{10.55776/P33216} (standalone 
	project P33216 ``Computational nonlinear PDEs''). Additionally, Maximilian Brunner
		and Julian Streitberger are supported by the Vienna School of Mathematics.}

\begin{document}

\maketitle

\begin{abstract}
	We analyze a goal-oriented adaptive
	algorithm that aims to efficiently
	compute the quantity of interest \(G(u^\star)\) with a
	linear goal functional \(G\) and the solution
	\(u^\star\) to a general second-order nonsymmetric
	linear elliptic partial differential equation. The
	current state of the analysis of iterative algebraic
	solvers for nonsymmetric systems lacks the contraction
	property in the norms that are prescribed by the
	functional ana\-lytic setting. This seemingly prevents
	their application in the optimality analysis
	of goal-oriented adaptivity. As a remedy, this paper
	proposes a
	goal-oriented adaptive iteratively symmetrized finite
	element method (GOAISFEM). It employs a nested loop
	with
	a contractive symmetrization procedure, e.g., the
	Zarantonello iteration, and a contractive algebraic
	solver, e.g., an optimal multigrid solver.
	The various
	iterative procedures require well-designed stopping
	criteria such that the adaptive algorithm can
	effectively steer the local mesh refinement and the
	computation of the
	inexact discrete approximations. The main results
	consist of full linear convergence of the proposed
	adaptive algorithm and the proof of optimal
	convergence rates with respect to both degrees of
	freedom and total computational cost (i.e., optimal
	complexity). Numerical experiments confirm the
	theoretical results and investigate the selection of the
	parameters.
\end{abstract}


\section{Introduction}

Adaptive finite element methods (AFEMs) are a cornerstone in the numerical
solution of partial differential equations (PDEs). The abundant
literature
emphasizes significant progress and manifests a matured
understanding
of the
topic; see,
e.g.,~\cite{d1996, mns2000, bdd2004, stevenson2007, ckns2008,
	ks2011,cn2012,ffp2014,axioms}
for linear elliptic PDEs.

The variational formulation of a
\emph{nonsymmetric} second-order linear elliptic PDE
with bilinear form \(b(\cdot, \cdot)\)
and right-hand side functional \(F\)
on the Sobolev space \(\XX  \coloneqq H^1_0(\Omega)\)
seeks a weak solution \(u^\star\)
to
\begin{equation}\label{eq:intro:weak_formulation}
	b(u^\star, v) = F(v)
	\quad
	\text{for all}
	\quad
	v \in \XX.
\end{equation}
While standard AFEM aims at an
efficient approximation of the solution
\(u^\star \in \XX\),
goal-oriented AFEM (GOAFEM) strives only
to approximate a \emph{quantity of interest}
\(G(u^\star)\); see~\cite{br2001, br2003, eehj1995, gs2002} for
early prominent contributions.
However,
to accurately approximate $G(u^\star)$ for a
continuous linear \emph{goal functional}
\(G \colon \XX \to \R\),
following the generic approach  \(G(u_H) \approx G(u^\star)\)
leads to convergence rates
determined by the error of the approximation
\(u_H \approx u^\star\)
to the \emph{primal problem}~\eqref{eq:intro:weak_formulation}.
Instead, GOAFEM adopts a duality technique by
additionally approximating \(z_H \approx
z^\star \in \XX\)
solving the \emph{dual problem}
\begin{equation}\label{eq:intro:dual_problem}
	b(v, z^\star)
	=
	G(v)
	\quad \text{for all } v \in \XX.
\end{equation}
Following~\cite{gs2002},
a discrete approximation
\(G_H(u_H, z_H) \approx G(u^\star)\)
enables the control of the error for
any \(u_H, z_H \in \XX\) by
\begin{equation}\label{eq:intro:goal_error_identity}
	| G(u^\star) - G_H(u_H, z_H) |
	\le
	| b(u^\star - u_H, z^\star - z_H) |
	\le
	L \,
	\vvvert u^\star - u_H \vvvert\,
	\vvvert z^\star - z_H \vvvert,
\end{equation}
where $L>0$ is the continuity constant
of \(b(\cdot, \cdot)\) with respect to the energy norm~\(\vvvert \,
	\cdot \, \vvvert\); see Section~\ref{section:preliminaries} for details.
As seen in~\eqref{eq:intro:goal_error_identity}, this approach allows
to add the convergence rates
of the primal and dual problem.
Moreover, it is not necessary
--
and may even lead to unnecessary computational expense
--
to compute approximations
\(u_H \approx u^\star\)
and \(z_H \approx z^\star\)
across the entire domain with the same accuracy.
Instead, a
careful marking of elements for refinement
enables a considerable reduction of the
computational costs and makes GOAFEM
highly relevant in both practical applications
and mathematical research.

First rigorous convergence results of GOAFEM
are found in~\cite{ms2009, bet2011,
	fghpf2016, fpz2016, hp2016},
recent contributions
in this context
include~\cite{bip2021, bbimp2022} and for a
dual weighted-residual approach see,
e.g., \cite{elw2019, elw2020, dbr2021}.
The works
\cite{ms2009,fghpf2016,fpz2016,bip2021,bbimp2022}
focus on optimal convergence rates
with respect to the degrees of freedom.
However,
the cumulative nature of adaptivity calls
for optimal convergence rates with respect to
the total computational effort, i.e.,
the overall computational time.
Coined as \emph{optimal complexity}
initially for wavelet-based
discretizations~\cite{cdd2001, cdd2003},
this notion was later adopted for
AFEM with contributions including,
e.g.,
\cite{stevenson2007, cg2012,ghps2021, bhimps2023}.
In the setting of GOAFEM,
optimal complexity was established first in \cite{ms2009} for
the Poisson problem and sufficiently small adaptivity parameters, and extended
to a general second-order \emph{symmetric} linear elliptic PDE
with
uniformly contractive algebraic solver
in~\cite{bgip2023}.
Since uniform contraction with respect to the
PDE-related energy norm
for nonsymmetric algebraic solvers
such as GMRES is still open,
as a remedy,
the proof of the Lax--Milgram lemma
motivates the application
of an iterative symmetrization
\cite{bhimps2023}.
This results in a sequence of symmetric
algebraic systems that allow the application
of optimal algebraic solvers,
e.g.,~\cite{wz2017,cnx2012,imps2022}.
Figure~\ref{fig:algorithm_details} illustrates
the nested structure of the resulting
goal-oriented adaptive iteratively symmetrized finite
element method (GOAISFEM).
The detailed Algorithm~\ref{algorithm:afem}
is presented in Section~\ref{section:algorithm} below.
Table~\ref{fig:indices} displays the notation of the associated indices and
quasi-error
quantities, which are equivalent to the total error.

\begin{figure}
	\resizebox{0.75\textwidth}{!}{
		\begin{tikzpicture}[scale=1]
	\tikzstyle{afemarrow} = [very thick, color=lightgray, -stealth]
	\tikzstyle{dummynode} = [draw=none]
	\newcommand{\intermediate}[2]{($(#1)!0.5!(#2)$)}
	\pgfmathsetmacro{\pad}{0.1pt}
	\pgfmathsetmacro{\ydist}{0pt}
	\pgfdeclarelayer{background}
	\pgfsetlayers{background,main}
	\tikzstyle{module}[lightgray]=[draw,color=#1,fill=#1!50,very 
	thick,text=#1!0!black,align=center,rounded corners,minimum 
	width=15em,font=\bfseries]
	\tikzstyle{adaptivity}[lightgray]=[draw,color=#1,fill=#1!50,very 
	thick,text=#1!0!black,align=center,rounded corners,minimum 
	width=40em]
	\tikzstyle{primal}[lightgray]=[draw,very 
	thick,text=#1!0!black,align=center,rounded corners,minimum 
	width=22em]
	\tikzstyle{class}=[align=left,minimum width=14em,text 
	width=18em,font=\footnotesize\strut]
	\tikzstyle{connection}[lightgray]=[color=#1,very 
	thick,-{Diamond[length=0.7em,width=0.4em]}]
	\tikzstyle{surround}[lightgray]=[draw,color=#1,fill=#1!20,rounded corners,very thick]
	
	%
	\node[adaptivity=TUYellow!70!white] at (0,0) (M3) 
	{\textsc{Goal-oriented Adaptivity} 
		(\(\boldsymbol{\ell}\))};

	\node[adaptivity=codegray!70!white, below = of M3.south west, anchor=west,
	yshift=0.25cm] (M4) 
	{{\textsc{Solve \& Estimate}}};
	
	\node[below = of M4.mid, anchor=mid, align = center,
	yshift=0.3cm, xshift=-11em] (M5) 
	{{\bfseries {primal problem}}};
	
	\node[below= of M4.mid, align=center, anchor = mid, yshift=0.3cm, xshift=9em] 
	(M6) 
	{{\bfseries {dual problem}}};
	
	\node[below= of M4.mid, align=center, anchor = mid, yshift=0.3cm] 
	(M8) 
	{{\bfseries {(in parallel)}}};
	
	\node[module=TUblue!70!white,below = of M5.mid, anchor=mid] (M2) 
	{symmetrize (\(\boldsymbol{m}\))};
		
	\node[module=TUGreen!70!white, below=of M2.west, anchor=west, xshift=2em] (M1) 
	{solve (\(\boldsymbol{n}\))};
	\node[class,below=\pad of M1, xshift=4em] (C13) {computable 
		approximation \(u_{\ell}^{m, n}\) \\and
		estimator 
		\(\eta_{\ell}(u_\ell^{m,n})\)};
		
	\node[module=TUblue!70!white,below = of M6.mid, anchor=mid] (M7) 
	{symmetrize (\(\boldsymbol{\mu}\))};
	
	\node[module=TUGreen!70!white, below=of M7.west, anchor=west, xshift=2em] (M8) 
	{solve (\(\boldsymbol{\nu}\))};
	\node[class,below=\pad of M8,xshift=4em] (E13) {computable 
		approximation \(z_{\ell}^{\mu, \nu}\) \\ and estimator 
		\(\zeta_{\ell}(z_\ell^{\mu, \nu})\)};
		

	\node[adaptivity=TUred!70!white, below = of M4.south west, anchor=west, yshift=-3.5cm] (M5) 
	{\textsc{Mark}};
	\node[class,below=\pad of M5.south, xshift=0.5em] (F11) {apply Dörfler 
	marking variant from 
	\cite{fpz2016}};
	
	\coordinate (B5u) at ($(M4.south)+(0*\pad,-6*\pad)$);
	\coordinate (B5l) at ($(M5.north)+(0*\pad, 6*\pad)$);
	\draw[very thick] (B5u) -- (B5l);
	
	\node[adaptivity=darkgreen!70!white,below = of M5.south west, anchor=west, 
	yshift=-0.3cm] (M6) 
	{\textsc{Refine}};
	\node[class,below=\pad of M6,xshift=1em] (G11) {employ NVB \cite{stevenson2008}};
	%
	\coordinate[left=1.5em of M6] (D1);
	\coordinate[left=1.5em of M4] (D2);
	\coordinate[left=0.2em of M6] (D3);
	\coordinate[left=0.2em of M4] (D4);
	\draw[afemarrow, rounded corners] (D3) --(D1) -- (D2) -- (D4);
	%
	%
	\begin{pgfonlayer}{background}
		\coordinate (B3u) at ($(M3.north east)+(10*\pad,2*\pad)$);
		\coordinate (B3l) at ($(M3.south west |- G11.south)+(-10*\pad,-2*\pad)$);
		\draw[surround=TUYellow!100!white] (B3u) rectangle (B3l);
		\coordinate (B2u) at ($(M4.north east)+(0.7*\pad,1*\pad)$);
		\coordinate (B2l) at ($(M4.south west |- 
		M5.north)+(-0.7*\pad,2*\pad)$);
		\draw[surround=codegray!100!white] (B2u) rectangle (B2l);
		\coordinate (B1u) at ($(M5.north east)+(0.7*\pad,0.5*\pad)$);
		\coordinate (B1l) at ($(M5.west |- F11.south)+(-0.7*\pad,1*\pad)$);
		\draw[surround=TUred!100!white] (B1u) rectangle (B1l);
		
		\coordinate (B4u) at ($(M6.north east)+(0.7*\pad,0.5*\pad)$);
		\coordinate (B4l) at ($(M6.west |- G11.south)+(-0.7*\pad,1*\pad)$);
		\draw[surround=darkgreen!100!white] (B4u) rectangle (B4l);
	\end{pgfonlayer}
\end{tikzpicture}
	}
	\caption{Schematic overview of the GOAISFEM algorithm with
		nested symmetrization and inexact solver.\label{fig:algorithm_details}}
\end{figure}
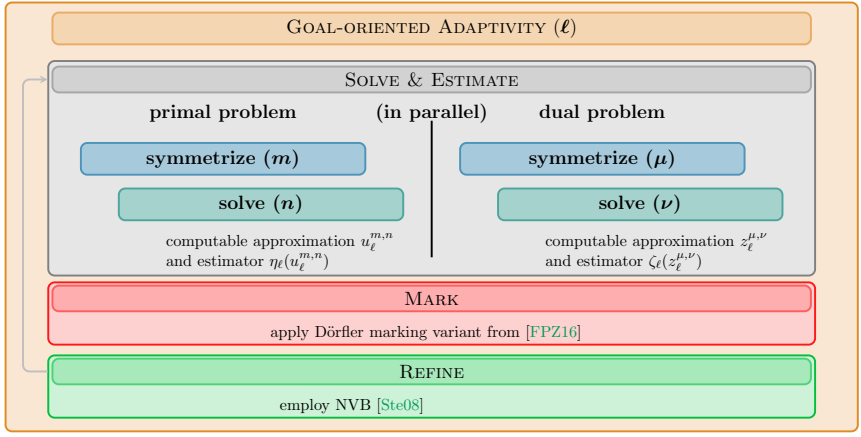

\begin{table}[htpb!]
	\centering
	\resizebox{\columnwidth}{!}{
		\begin{tabular}{cccccccccc}
			\toprule %
			\multicolumn{1}{c}{iteration}&\multicolumn{2}{c}{mesh 
			refinement}&\multicolumn{2}{c}{symmetrization}
			&\multicolumn{2}{c}{algebraic solver} &\multicolumn{1}{c}{} &\\
			\cmidrule(r){1-2} 
			\cmidrule(r){2-3}\cmidrule(l){4-5}\cmidrule(l){6-7} & running & 
			final &
			running & final & running & final & index set &
			\multicolumn{1}{c}{quasi-error} \\\hline
			primal & $\ell$  & $\underline\ell$ & \(m\) & \(\underline{m}\)& 
			\(n\) & \(\underline{n}\) &  \(\QQ^u\) in \eqref{eq:primal_index_set} & \(\Eta_{\ell}^{m, n}\) in~\eqref{eq:quasi-error-primal}
			 \\
			dual & $\ell$ & $\elll$  & $\mu$ & $\mmu$ & \(\nu\) &
			\(\nnu\) & \(\QQ^z\) in \eqref{eq:dual_index_set} & \(\Zeta_{\ell}^{\mu, \nu}\) in \eqref{eq:quasi-error-dual}
			\\
			combined & \(\ell\) & \(\elll\) & \(k\) & \(\kk = \max\set{\mm, \mmu}\) & \(j\) & 
			\(\jj = \max \set{\nn, \nnu}\) & 
			\(\QQ = \QQ^u \cup \QQ^z\)&
			\(\Eta_{\ell}^{k, j} \, \Zeta_{\ell}^{k, j}\) 
			in~\eqref{eq:quasi-error-extension}
			\\
			\bottomrule
\end{tabular}

	}
	\vspace{0.4cm}
	\caption{Iteration counters and quasi-errors for the GOAISFEM algorithm.
		We note that for the combination of the index sets, the quasi-errors are extended to the full index
		set by the last available quasi-error. We refer to
		Section~\ref{section:algorithm} for
		details on the iteration counters and index sets and to the beginning of
		Section~\ref{section:main_results} for a detailed description of the quasi-errors and their
		extension to the full index set \(\QQ\).\label{fig:indices}
	}
\end{table}

The first challenge in the analysis of the GOAISFEM
algorithm consists of the nonlinear product structure
attained by the combined quasi-error product as displayed in
Table~\ref{fig:indices}. The resulting nonlinear remainder term
significantly complicates the proof compared to treating
only the primal problem as in \cite{bhimps2023} and requires the
application of a novel proof strategy from \cite{fps2023} that only
utilizes summability of the
remainder, denoted as
\textsl{tail-summability} throughout. The second
challenge arises from the combination of the primal and dual
marking leading to a merged marked set. Thereby, either only
the primal or only the dual estimator is guaranteed to
satisfy the estimator reduction property. Since the
estimator belongs to the quasi-error, this also leads to a
failure of contraction for one of the two involved
quasi-errors. While \cite{bgip2023} solves this issue in the
symmetric case, the additional symmetrization loop results
in a more involved situation at hand.
Adapting the novel
approach of the tail-summability criterion from \cite{fps2023}
enables the proof of full linear convergence and
optimal complexity for the nonlinear
quasi-error product in this paper.
The analysis employs the
generalized quasi-orthogonality from \cite{feischl2022}
to remedy the lack of a Pythagorean identity
for nonsymmetric problems.

Our main result asserts full linear convergence of
the quasi-error product
\(\Eta_{\ell}^{k, j} \, \Zeta_{\ell}^{k, j}\) with respect to the total step counter \(| \cdot, \cdot, \cdot |\)
(measuring the total solver steps in the index set).
Therein, we allow for an arbitrary symmetrization stopping
parameter \(\lamsym\) and only require a
small algebraic solver
parameter \(\lamalg\) such that the product \(\lamsym \, \lamalg\) is
sufficiently small.
More precisely, Theorem~\ref{lem:full_linear_convergence}
states that there exist constants \(\Clin
> 0\) and \(0 < \qlin < 1\) such that,
for all \((\ell, k, j), (\ell', k', j') \in \QQ\)
with \(| \ell', k', j'| \le |\ell, k, j|\),
\begin{equation*}
	\Eta_\ell^{k,j} \, \Zeta_\ell^{k, j}
	\le
	\Clin \, \qlin^{|\ell, k, j| - |\ell^\prime, k^\prime, j^\prime|} \,
	\Eta_{\ell^\prime}^{k^\prime,j^\prime} \,
	\Zeta_{\ell^\prime}^{k^\prime,j^\prime}.
\end{equation*}
Note that, unlike \cite{bhimps2023}, where
full linear convergence is guaranteed only for sufficiently large \(\ell \ge
\ell_0\), the current result is stronger in the sense that the result holds for
\(\ell_0 = 0\)
owing to a generalized
quasi-orthogonality from~\cite{feischl2022}.
An immediate consequence of full linear convergence and the
geometric series in Corollary~\ref{cor:rate_complexity} states that the rates
with respect to the degrees of
freedom coincide with the rates with respect to the cumulative computational
work (i.e., computational time),
i.e., for all \(r > 0\), there holds
\begin{equation*}
	\thinmuskip = -1mu
	\sup \limits_{
		\substack{
			(\ell, k, j) \in \QQ
		}
	}
	\bigl(\# \TT_{\ell}\bigr)^r \Eta_\ell^{k,j} \, \Zeta_\ell^{k,j}
	\le
	\sup \limits_{
		\substack{(\ell, k, j) \in \QQ
		}
	}
	\Bigl(
	\sum \limits_{
		\substack{
			(\ell^\prime, k^\prime, j^\prime) \in \QQ
			\\
			| \ell^\prime, k^\prime, j^\prime | \le | \ell, k, j |
		}
	}
	\# \TT_{\ell^\prime}
	\Bigr)^r
	\Eta_\ell^{k,j} \Zeta_\ell^{k,j}
	\le
	C_{\rm cost} \, \sup \limits_{
		\substack{
			(\ell, k, j) \in \QQ
		}
	}
	\bigl(\# \TT_{\ell}\bigr)^r \Eta_\ell^{k,j} \, \Zeta_\ell^{k,j}
\end{equation*}
along the sequence of meshes \(\TT_{\ell}\) generated by the GOAISFEM algorithm.
The second main result of Theorem~\ref{th:optimal_complexity}
proves that, for sufficiently small adaptivity parameters and any achievable
rates
\(s, t > 0\) of the primal resp.\ dual problem (stated in terms
of nonlinear approximation classes),
the algorithm guarantees optimal complexity, i.e.,
\begin{equation*}
	\sup_{
		\substack{
			(\ell,k, j) \in \QQ
		}
	}
	\Bigl(
	\sum_{\substack{
			(\ell^\prime, k^\prime, j^\prime) \in \QQ
			\\
			|\ell^\prime, k^\prime, j^\prime | \le | \ell, k, j
			|
		}
	}
	\#\TT_{\ell^\prime}
	\Bigr)^{s+t} \,
	\Eta_\ell^{k,j} \, \Zeta_\ell^{k, j}
	\le
	\Copt \,
	\max\{
	\| u^\star \|_{\A_s} \, \| z^\star \|_{\A_t}, \,
	\Eta_{0}^{0,0} \, \Zeta_{0}^{0,0}
	\}.
\end{equation*}
This means the convergence of the algorithm
attains the optimal rate \(s+t\)
with respect to the overall
computational work, where \(\| u^\star \|_{\A_s}
< \infty\) means that \(u^\star\) can be approximated at rate \(s\) (along a
sequence of
unavailable optimal meshes) and likewise for \(z^\star\).

The remaining parts of the paper are organized as follows.
The
preliminary Section~\ref{section:preliminaries} introduces
the model problem, the assumptions on the solvers, and the axioms
of adaptivity from~\cite{axioms},
including the general quasi-orthogonality
from \cite{feischl2022}.
Following the algorithm in
Section~\ref{section:algorithm} and its contraction properties in Section~\ref{section:aposteriori},
Section~\ref{section:main_results} presents full linear convergence as the first main result of this
paper. This allows to prove optimal complexity in Section~\ref{section:optimal_complexity} as the second main result, which is
underlined by
the numerical
experiments in Section~\ref{section:numerics} including a thorough
investigation of the adaptivity parameters. The paper concludes with a summary in Section~\ref{section:conclusion}.


\section{Setting}\label{section:preliminaries}
In this section, we introduce the problem and explain the key
components needed to design the
adaptive algorithm in Section~\ref{section:algorithm}.
%
\subsection{Continuous model problem}
Let \(\Omega \subset \R^d\) with \(d \ge 1\) be a polygonal Lipschitz domain.
Given right-hand sides
\(f \in L^2(\Omega)\) and \(\fvec \in [L^2(\Omega)]^d\), we consider a
general second-order linear elliptic PDE
\begin{equation}\label{eq:model_problem}
	-\div(\Amat \nabla u^\star) +  \bvec \cdot \nabla u^\star + c \, u^\star
	=
	f - \div(\fvec)
	\quad \text{in } \Omega
	\quad \text{subject to} \quad
	u^\star = 0 \quad \text{on } \partial \Omega,
\end{equation}
with a pointwise symmetric and positive definite diffusion matrix
\(\Amat \in \bigl[L^\infty(\Omega)\bigr]^{d \times d}_{\textup{sym}}\), a
convection coefficient \(\bvec \in \bigl[L^{\infty}(\Omega)\bigr]^d\), and a
reaction
coefficient \(c \in L^\infty(\Omega)\).
For well-definedness of the \textsl{a posteriori} error estimator
in Section~\ref{section:estimator} below, we additionally
require that
\(\Amat\vert_T \in \bigl[W^{1,
			\infty}(T)\bigr]^{d \times d}_{\textup{sym}}\) and \(\fvec\vert_T \in
\bigl[H^1(T)\bigr]^d\) for
all \(T \in \TT_{0}\), where
\(\TT_0\) is an initial triangulation that subdivides
\(\Omega\) into compact simplices.
Let \(\langle\, \cdot \,, \, \cdot \,\rangle\) denote the \(L^2(\Omega)\)-scalar
product.
With the principal part \(a(u,v) \coloneqq \langle \Amat
\nabla u, \nabla v \rangle\), the variational formulation of \eqref{eq:model_problem}
seeks a solution \(u^\star \in \XX \coloneqq H^1_0(\Omega)\) to
the so-called \emph{primal problem}
\begin{equation}\label{eq:weak_formulation}
	b(u^\star, v) \coloneqq a(u^\star, v) + \langle \bvec \cdot \nabla u^\star
	+ c \, u^\star, v \rangle
	=
	\langle f, v \rangle + \langle \fvec, \nabla v \rangle \eqqcolon F(v)
	\quad \text{for all } v \in \XX.
\end{equation}
We suppose that the bilinear form \(b(\cdot, \cdot)\)
from~\eqref{eq:weak_formulation} is continuous
and elliptic with respect to the norm
\(
\| \cdot \|_{\XX}
\)
on \(\XX\), i.e., there exist constants \(L^\prime,
\alpha^\prime > 0\) such that
\begin{equation}\label{eq:b_cont_elliptic}
	b(u,v)
	\le
	L^\prime \, \| u \|_{\XX} \| v \|_{\XX}
	\quad \text{and} \quad
	b(v,v)
	\ge
	\alpha^\prime \, \| v \|_{\XX}^2
	\quad \text{for all } u,v \in \XX.
\end{equation}
Then, the Lax--Milgram lemma proves existence and uniqueness of the solution
\(u^\star\) to~\eqref{eq:weak_formulation}.
An elementary compactness argument shows
that~\eqref{eq:b_cont_elliptic} implies
ellipticity of the principal part \(a(\, \cdot \, , \, \cdot \,)\) and thus
\(a(\, \cdot\,,\, \cdot\,)\) is a scalar
product on \(\XX\) with induced energy norm
\(
a(\, \cdot \,, \, \cdot \,)^{1/2}
\eqqcolon
\vvvert \, \cdot \, \vvvert
\simeq
\| \cdot \|_{\XX}
\), cf.~\cite[Remark~3]{bhp2017}.
Therefore, \(b(\,\cdot\,,\, \cdot\,)\) is also continuous and
elliptic with respect to \(\vvvert \, \cdot \, \vvvert\), i.e.,
there
exist constants \(L, \alpha > 0\) such that
\begin{equation}\label{eq:b_cont_elliptic_enorm}
	b(u,v)
	\le
	L \,  \vvvert u \vvvert \, \vvvert v \vvvert
	\quad \text{and} \quad
	b(v,v)
	\ge
	\alpha \, \vvvert v \vvvert^2
	\quad \text{for all } u,v \in \XX.
\end{equation}
In the present paper, we suppose that
the \emph{quantity of interest} \(G\) is linear
and reads for
given data \(g \in L^2(\Omega)\) and \(\gvec \in
\bigl[L^2(\Omega)\bigr]^d\),
\begin{equation*}
	G(v)
	\coloneqq
	\int \limits_\Omega
	\bigl(g \, v + \gvec \cdot \nabla v\bigr) \d{x}.
\end{equation*}
In order to guarantee well-definedness of the error estimator
in~Section~\ref{section:estimator} below, we
suppose \(\gvec \vert_T \in \bigl[H^1(T)\bigr]^d\) for all
initial simplices \(T
\in \TT_{0}\).
In view of the continuity and coercivity of
\(b(\, \cdot \,, \, \cdot \,)\), the Lax--Milgram lemma yields existence and
uniqueness of the solution
\(
z^\star \in \XX
\)
of the so-called \emph{dual problem}: Find \(z^\star \in \XX\) such that
\begin{equation}\label{eq:dual_problem}
	b(v, z^\star)
	=
	G(v)
	\quad \text{for all } v \in \XX.
\end{equation}

\subsection{Finite element discretization and discrete goal}\label{subsec:fem}
For a polynomial degree \(p \in \N\) and a conforming
simplicial
triangulation
\(\TT_H\) of $\Omega$, the discrete ansatz space
reads
\begin{equation}\label{eq:fem_space}
	\XX_H
	\coloneqq
	\{
	v_{H} \in \XX \colon
	\forall \, T \in \TT_H, \
	v_{H}|_T \
	\text{is a polynomial of total degree}
	\le p
	\}.
\end{equation}
Since \(\XX_H \subset \XX \) is conforming, the Lax--Milgram
lemma ensures the existence and uniqueness of primal and dual
discrete solutions \(u_H^\star\), \(z_H^\star \in
\XX_H\)
satisfying
\begin{equation}\label{eq:discrete_formulation}
	b(u_H^\star, v_H)
	=
	F(v_H)
	\quad \text{and} \quad
	b(v_H, z_H^\star)
	=
	G(v_H)
	\quad \text{for all } v_H \in \XX_H.
\end{equation}
It is well-known that conforming FEMs are quasi-optimal, i.e., there
hold Céa-type estimates with constant
\(\Ccea =L / \alpha\)
\begin{equation}\label{eq:cea}
	\vvvert u^\star - u_H^\star \vvvert
	\le
	\Ccea \,\min_{v_H \in \XX_H} \vvvert u^\star -v_H \vvvert
	\quad \text { and } \quad
	\vvvert z^\star - z_H^\star \vvvert
	\le
	\Ccea \,\min_{v_H \in \XX_H}
	\vvvert z^\star -v_H \vvvert.
\end{equation}
For arbitrary approximations
\(
u_H, z_H,
\in \XX_H
\)
the linearity of the quantity of interest \(G\)
as well as the primal and the dual problem
\eqref{eq:intro:weak_formulation}
and \eqref{eq:intro:dual_problem}
show that
\begin{align*}
	G(u^\star) - G(u_H)
	=
	G(u^\star - u_H)
	 & \eqreff*{eq:intro:dual_problem}=
	b(u^\star - u_H, z^\star)
	\\
	 & \eqreff*{eq:intro:weak_formulation}=
	b(u^\star - u_H, z^\star - z_H)
	+
	\bigl[F(z_H) - b(u_H, z_H)\bigr].
\end{align*}
The definition of the discrete goal quantity by
\(
G_H(u_H, z_H)
\coloneqq
G(u_H)
+
\bigl[F(z_H) - b(u_H, z_H)\bigr]
\)
allows to control the goal error by continuity of \(b(\cdot, \cdot)\)
\begin{equation}\label{eq:goal_error_identity}
	| G(u^\star) - G_H(u_H, z_H) |
	\le
	| b(u^\star - u_H, z^\star - z_H) |
	\le
	L \, \vvvert u^\star - u_H \vvvert \vvvert z^\star - z_H \vvvert.
\end{equation}
We emphasize that \eqref{eq:goal_error_identity} holds for any
\(
u_H, z_H
\)
and, in particular, for those stemming from an iterative solution
step. Moreover, if \(u_H = u_H^\star\), then \(G(u_H, z_H) =
G(u_H^\star)\) as expected.

\subsection{Zarantonello iteration}
The discrete formulations~\eqref{eq:discrete_formulation} lead to
positive definite, but
\emph{nonsymmetric} linear systems of equations. To reduce the formulation to
symmetric and positive definite (SPD) problems, we follow previous
own work \cite{bhimps2023} for the primal problem and employ the Zarantonello
iteration~\cite{zarantonello1960}. Typically, the
latter is used in the up-to-date proof of the Lax--Milgram lemma and also defines a
linearization
scheme for the treatment of a certain class of nonlinear
elliptic
PDEs (see, e.g., \cite{cw2017, ghps2018, hpsv2021,fps2023}). In
its core, it is a fixed-point method, thus
also applicable in
the nonsymmetric setting at hand. For a
damping parameter
\(\delta > 0\) and given \(u_H, z_H \in \XX_H\),
the Zarantonello iterations \(\Phi_{H}^u,
\Phi_{H}^z \colon (0,
\infty) \times \XX_H \to \XX_H\) compute
the unique solutions \(\Phi_H^{u}(\delta;
u_H)\),
\(\Phi_H^{z}(\delta; z_H) \in \XX_H\)
to the symmetric variational formulations
\begin{subequations}\label{eq:Zarantonello-iterations}
	\begin{alignat}{2}
		a(\Phi_H^{u}(\delta; u_H), v_H)
		 & =
		a(u_H, v_H)
		+
		\delta \, \bigl[F(v_H) - b(u_H, v_H)\bigr]
		\quad
		 &   & \text{for all} \quad v_H \in \XX_H,
		\\
		a(v_H, \Phi_H^{z}(\delta; z_H))
		 & =
		a(v_H, z_H)
		+
		\delta \, \bigl[G(v_H) - b( v_H, z_H)\bigr] \quad
		 &   & \text{for all} \quad v_H \in \XX_H.
	\end{alignat}
\end{subequations}
The Riesz--Fischer theorem (and also the Lax--Milgram lemma)
guarantees existence and uniqueness of
$\Phi_H^{u}(\delta; u_H)$, $\Phi_H^{z}(\delta; z_H) \in
	\XX_H$, i.e.,
the Zarantonello operators
\(\Phi_H^{u}(\delta; \cdot)\) and
\(\Phi_H^{z}(\delta;
\cdot)\) are well-defined.
In particular, the exact discrete solutions $u_{H}^\star
	= \Phi_{H}^{u}(\delta; u_{H}^\star)$ and
$z_{H}^\star = \Phi_{H}^{z}(\delta; z_{H}^\star)$ are the
unique fixed points for
all $\delta > 0$.
Moreover, for a sufficiently small damping parameter \(\delta\),
i.e.,
\(
0 < \delta < \delta^\star \coloneqq 2 \alpha/L^2,
\)
the Banach fixed-point theorem \cite[Section 25.4]{zeidler1990} guarantees that
\(\Phi_H^{u}(\delta, \cdot)\) and \(\Phi_H^{z}(\delta, \cdot)\) are
contractive with constant \(0 <
\qsym \coloneqq \bigl[1- \delta \, (2 \alpha - \delta L^2)\bigr]^{1/2} <
1\), i.e., for all functions
\(v_H, w_H \in \XX_H\),
it holds that
\begin{align}\label{eq:Zarantonello-contraction}
	\max
	\bigl\{
	\vvvert \Phi_H^{u}(\delta; v_H) - \Phi_H^{u}(\delta; w_H) \vvvert,
	\vvvert\Phi_H^{z}(\delta; v_H) - \Phi_H^{z}(\delta;
	w_H) \vvvert
	\bigr\}
	\le
	\qsym \, \vvvert v_H - w_H \vvvert.
\end{align}
The optimal value \(\delta_{\mathrm{opt}} = \alpha / L^2\)
yields the
minimal contraction value \(\qsym = 1 - \alpha^2 / L^2\).
\subsection{Algebraic solver}
A canonical candidate for
solving~\eqref{eq:discrete_formulation} directly
is a generalized minimal
residual method~\cite{saad2003, ss1986} with optimal preconditioner for the
symmetric part. While this guarantees uniform
contraction of the algebraic residuals in a \emph{discrete vector norm},
the link between the algebraic
residuals and the functional setting
is still open \cite{bhimps2023}.
Instead, after a symmetrization with the Zarantonello iteration, it
remains to solve the SPD systems~\eqref{eq:Zarantonello-iterations}.
Since large SPD problems are still computationally expensive
and the exact solution cannot be computed in linear
computational complexity,
we employ an
iterative algebraic solver whose iteration is expressed by the
operator
\(
\Psi_H \colon \XX^\prime \times \XX_H \to \XX_H.
\)
More precisely, given a bounded linear functional \(\psi \in
\XX^\prime\) and an approximation
\(
w_H \in \XX_H
\)
of the exact solution \(w_H^\star \in \XX_H\) to
\(
a(w_H^\star, v_H)
=
\psi(v_H)
\)
for all \(v_H \in \XX_H\),
the algebraic solver returns an improved approximation \(\Psi_H(\psi;
w_H) \in \XX_H\) in the sense
that there exists \(0 < \qalg < 1\) independent of \(\psi\) and \(\XX_H\) such that
\begin{equation}\label{eq:algebra_contraction}
	\vvvert w_H^\star - \Psi_H(\psi; w_H) \vvvert
	\le
	\qalg \, \vvvert w_H^\star - w_H \vvvert
	\quad
	\text{for all}
	\quad w_H \in \XX_H.
\end{equation}
To simplify notation, we shall
identify \(\psi\) with its Riesz representative
\(w_H^\star \in \XX_H\) and write
\(
\Psi_H(w_H^\star; \cdot)
\)
instead of $\Psi_H(\psi; \cdot)$, even though
$w_H^\star$ is
unknown in
practice and will only be approximated by an optimal algebraic
solver,
e.g., \cite{cnx2012, wz2017, imps2022}.
In the following, we use the \(hp\)-robust multigrid method from \cite{imps2022} with localized lowest-order smoothing on
intermediate levels and patchwise higher-order smoothing on the finest mesh
as an innermost algebraic solver loop.

\subsection{Mesh refinement}
The mesh
refinement employs newest-vertex bisection (NVB). We refer to
\cite{stevenson2008} for NVB with admissible initial triangulation
$\TT_0$ and \(d \ge 2\), to
\cite{AFFKP13, kpp2013} for NVB with general \(\TT_0\) for \(d
\in \{1,2\}\), and to the
recent work \cite{dgs2023} for NVB with general $\TT_0$ in any dimension
\(d \ge 2\). For each
triangulation \(\TT_H\) and marked elements
\(\MM_H \subseteq \TT_H\), let
\(
\TT_h \coloneqq \texttt{refine}(\TT_H, \MM_H)
\)
be the coarsest conforming refinement of \(\TT_H\) such that
at least all \(T \in \MM_H\)
have been refined, i.e.,
\(\MM_H  \subseteq \TT_H \setminus \TT_h\). We write
\(
\TT_h \in \T(\TT_H)
\)
if \(\TT_h\) can be obtained from \(\TT_H\) by finitely many steps of
NVB, and
\(\TT_h \in \T_N(\TT_H)\) if \(\TT_h \in
\T(\TT_H)\) with
\(
\# \TT_h - \# \TT_H \le N
\)
with the number of additional elements
\(
N \in \N_0
\).
To simplify notation, we write \(\T \coloneqq \T(\TT_0)\)
and \(\T_N \coloneqq
\T_N(\TT_{0})\). We note that the nestedness of meshes
\(\TT_h \in \T(\TT_H)\) implies nestedness of the corresponding
finite element spaces
\(
\XX_H \subseteq \XX_h \subset \XX
\)
from \eqref{eq:fem_space}.

\subsection{\textsl{A posteriori} error estimation}
\label{section:estimator}
For a triangle \(T \in \TT_H \in \T\) and
\(v_H \in
\XX_H\),  let \(\boldsymbol{n}\) denote the outer unit normal
vector and \(\jump{\, \cdot \,}\) the jump
along
inner edges of \(\TT_{H}\).
We define the refinement indicators
\(
\eta_H(T; v_H) \ge 0
\)
and \(\zeta_H(T; v_H) \ge 0\) for the primal and dual
problem from~\eqref{eq:discrete_formulation}, respectively, by
\begin{subequations}\label{eq:estimators}
	\begin{equation}
		\begin{aligned}
			\eta_{H}(T; v_H)^2
			 & \coloneqq
			| T |^{2/d} \, \Vert
			-\div(\Amat \nabla v_H - \fvec) + \bvec \cdot \nabla
			v_H + c\, v_H - f
			\Vert_{L^2(T)}^2
			\\
			 & \qquad
			+ | T |^{1/d} \, \Vert
			\jump{\bigl(\Amat \nabla v_H - \fvec\bigr)
				\cdot
				\boldsymbol{n}}
			\Vert_{L^2(\partial T \cap \Omega)}^2,
			\\
			\zeta_{H}(T; v_H)^2
			 & \coloneqq
			| T |^{2/d} \, \Vert
			-\div(\Amat \nabla v_H - \gvec) - \bvec \cdot \nabla
			v_H +
			\bigl(c - \div(\bvec)\bigr)\, v_H - g
			\Vert_{L^2(T)}^2
			\\
			 & \qquad
			+ | T |^{1/d} \, \Vert
			\jump{\bigl(\Amat \nabla v_H - \gvec\bigr) \cdot
				\boldsymbol{n}}
			\Vert_{L^2(\partial T \cap \Omega)}^2.
		\end{aligned}
	\end{equation}
	For any subset \(\UU_H \subseteq \TT_H\), we abbreviate
	\begin{equation}\label{eq:estimator_2}
		\eta_H(\UU_H; v_H)^2
		\coloneqq
		\sum \limits_{T \in \UU_H} \eta_{H}(T; v_H)^2
		\quad \text{and} \quad
		\zeta_H(\UU_H; v_H)^2
		\coloneqq
		\sum \limits_{T \in \UU_H} \zeta_{H}(T; v_H)^2
	\end{equation}
	as well as
	\(
	\eta_H(v_H)
	\coloneqq
	\eta_H(\TT_H; v_H)
	\)
	and
	\(
	\zeta_H(v_H)
	\coloneqq
	\zeta_H(\TT_H; v_H)
	\)
	for all \(v_H \in \XX_H\).
\end{subequations}
For details on residual-based error estimators, we refer to~\cite{ao2000, verfuerth1994}.
Throughout the paper, the index of the estimators refer to the underlying mesh,
e.g., \(\eta_h\) and \(\zeta_h\) on the refinement \(\TT_h \in \T(\TT_H)\)
or \(\eta_\ell\) and \(\zeta_\ell\) on a sequence of meshes \(\TT_\ell\) with \(\ell \in \N_0\).
It is well-known that $\eta_H,
	\zeta_H$ satisfy the
following \emph{axioms of adaptivity}.
\begin{lemma}[{\cite[Section~6.1]{axioms}}]
	The error estimators \(\eta_{H}, \zeta_{H}\) from~\eqref{eq:estimators}
	satisfy the following properties with constants $\Cstab, \Crel, \Cdrel,
		\Cmon> 0$ and $0 < \qred < 1$ for any
	triangulation $\TT_H \in \T$ and any conforming refinement $\TT_h
		\in\T(\TT_H)$
	with the corresponding Galerkin solutions $u_H^\star, z_H^\star
		\in \XX_H$, $u_h^\star, z_h^\star \in
		\XX_h$ to~\eqref{eq:discrete_formulation},
	any subset \(\UU_H \subseteq \TT_H \cap \TT_h\),
	and arbitrary $v_H \in \XX_H$, $v_h \in \XX_h$.
	\begin{enumerate}[font=\upshape, label=\textbf{\textrm{(A\arabic*)}}, ref=A\arabic*]
		\item
		      \label{axiom:stability}
		      \textbf{stability:}
		      \(
		      \vert \eta_h (\UU_H; v_h) -
		      \eta_H(\UU_H; v_H) \vert
		      +
		      \vert \zeta_h (\UU_H; v_h) -
		      \zeta_H(\UU_H; v_H) \vert
		      \le
		      \Cstab \, \vvvert v_h - v_H \vvvert.
		      \)
		\item \textbf{reduction:}
		      \label{axiom:reduction}
		      \(\eta_h(\TT_h \setminus \TT_H; v_H)
		      \le
		      \qred \, \eta_H(\TT_H \setminus \TT_h; v_H)\)
		      and
		      \(\zeta_h(\TT_h \setminus \TT_H; v_H)
		      \le
		      \qred \zeta_H(\TT_H \setminus \TT_h; v_H)
		      \).

		\item
		      \label{axiom:reliability}
		      \textbf{reliability:}
		      \(
		      \vvvert u^\star - u_H^\star \vvvert
		      \le
		      \Crel \, \eta_H(u_H^\star)
		      \)
		      and
		      \(\vvvert z^\star - z_H^\star \vvvert
		      \le
		      \Crel \, \zeta_H(z_H^\star)\).
		      \renewcommand{\theenumi}{A3$^{+}$}
		\item[\textbf{(A3$^{\boldsymbol{+}}$)}]\refstepcounter{enumi}
		      \label{axiom:discrete_reliability}
		      \textbf{discrete reliability:}
		      \(\vvvert u_h^\star - u_H^\star \vvvert
		      \le
		      \Cdrel \, \eta_H(\TT_H \backslash \TT_h, u_H^\star)\)
		      and
		      \(\vvvert z_h^\star - z_H^\star \vvvert
		      \le
		      \Cdrel \, \zeta_H(\TT_H \backslash \TT_h, z_H^\star)\)
		      .
		      \renewcommand{\theenumi}{QM}
		\item[\textbf{(QM)}]\refstepcounter{enumi}
		      \label{axiom:qm}
		      \textbf{quasi-monotonicity:}
		      $\eta_h(u_h^\star) \le \Cmon \,
			      \eta_H(u_H^\star)$ and
		      \(\zeta_h(z_h^\star) \le \Cmon \,
		      \zeta_H(z_H^\star)\).
	\end{enumerate}
	The constant $\Crel$ depends only on the uniform
	\(\gamma\)-shape
	regularity of
	all
	$\TT_H \in \T$ and on
	the space dimension $d$, while $\Cstab$ and
	\(\Cdrel\) additionally depend on the polynomial degree
	$p$.
	For NVB, reduction~\eqref{axiom:reduction} holds with
	$\qred \coloneqq 2^{-1/(2d)}$.
	Moreover,
	the constant in
	quasi-monotonicity~\eqref{axiom:qm}
	satisfies
	$\Cmon \le
		\min \{ 1 + \Cstab(1+\Ccea)\Crel \, , \, 1 + \Cstab \, \Cdrel\}$.
	\qed
\end{lemma}
Reliability~\eqref{axiom:reliability} and stability~\eqref{axiom:stability}
verify
\begin{align*}
	\vvvert u^\star - u_H \vvvert
	 & \le
	\max \{\Crel, 1 + \Cstab \, \Crel\} \, \bigl[ \eta_H(u_H) +
		\vvvert u_H^\star
		- u_H \vvvert\bigr],
	\\
	\vvvert z^\star - z_H \vvvert
	 & \le
	\max \{\Crel,  1 + \Cstab \, \Crel\} \, \bigl[ \zeta_H(z_H) +
		\vvvert z_H^\star - z_H \vvvert \bigr].
\end{align*}
In combination with the
estimate~\eqref{eq:goal_error_identity}, we finally conclude
for
\(
C_{\mathrm{goal}} \coloneqq L \max \{\Crel,  1 + \Cstab \, \Crel\}^2
\)
the reliable
\emph{goal-error estimate}
\begin{equation}\label{eq:goalError:Estimate}
	| G(u^\star) - G_H(u_H, z_H) |
	\le
	C_{\mathrm{goal}} \, \bigl[ \eta_H(u_H) +
		\vvvert u_H^\star
		- u_H \vvvert \bigr]
	\,
	\bigl[ \zeta_H(z_H) + \vvvert z_H^\star - z_H \vvvert
		\bigr],
\end{equation}
which provides the core estimate of the proposed adaptive algorithm
in Section~\ref{section:algorithm} below.

The ellipticity of \(b(\cdot, \cdot)\)
from~\eqref{eq:b_cont_elliptic_enorm} ensures
\(\inf\)-\(\sup\) stability
of the elliptic problem at hand.
Recall from~\cite{feischl2022}
that \(\inf\)-\(\sup\) stability implies
the generalized quasi-orthogonality, which
will be an important tool in the subsequent analysis.
\begin{proposition}[validity of quasi-orthogonality
	{\cite[Equation~(8)]{feischl2022}}]
	For any sequence $\XX_\ell \subseteq \XX_{\ell+1} \subset \XX$ of
	nested discrete subspaces with {\(\ell \ge 0\)},
	there holds
	\begin{enumerate}[font=\upshape, label=\textbf{\textrm{(A\arabic*)}},
			ref=A\arabic*]
		\setcounter{enumi}{3}
		\item \label{axiom:orthogonality}
		      \textbf{quasi-orthogonality:} There exist constants \(C_{\mathrm{orth}} > 0\) and \(0< \delta <
		      1\)
		      such that the corresponding Galerkin
		      solutions $u_\ell^\star, z_\ell^\star \in \XX_\ell$
		      to~\eqref{eq:discrete_formulation} satisfy,
		      for all \(\ell, M \in \N_0\),
		      \begin{subequations}\label{eq:quasi_orthogonality}
			      \begin{align}
				      \sum_{\ell' = \ell}^{\ell + M} \,
				      \vvvert u^\star_{\ell'+1} - u^\star_{\ell'} \vvvert^2
				       & \le
				      C_{\mathrm{orth}} \, (M + 1)^{1-\delta} \,
				      \vvvert u^\star - u^\star_\ell \vvvert^2,
				      \\
				      \sum_{\ell' = \ell}^{\ell + M} \,
				      \vvvert z^\star_{\ell'+1} - z^\star_{\ell'} \vvvert^2
				       & \le
				      C_{\mathrm{orth}} \, (M+1)^{1-\delta} \,
				      \vvvert z^\star - z^\star_\ell \vvvert^2.
			      \end{align}
		      \end{subequations}
	\end{enumerate}
	The constants $C_{\mathrm{orth}}$ and \(\delta\) depend only on the dimension $d$, the
	elliptic bilinear form $b(\, \cdot \,, \, \cdot \,)$, and the chosen norm
	$\vvvert \, \cdot \, \vvvert$, but are independent of the spaces
	$\XX_\ell$.
	\qed
\end{proposition}


\section{Adaptive algorithm} \label{section:algorithm}
In this section, we introduce our goal-oriented adaptive
iteratively symmetrized algorithm. It utilizes specific stopping
indices
denoted by an underline, e.g., \(\elll, \mm[\ell], \nn[\ell,k] \in
\N_0\). For an overview, see
Table~\ref{fig:indices} above.
However, we may omit the dependence whenever it is apparent from the
context, such
as in the abbreviation \(\nn \coloneqq \nn[\ell, m]\) for \(u_\ell^{m, \nn}\).
\begin{algorithm}[GOAISFEM]
	~\label{algorithm:afem}
	{\bfseries Input:}
	Initial mesh $ \TT_0$, polynomial degree $p \in \N$,
	marking parameters $0 < \theta \le 1$,
	$C_{\rm mark} \ge 1$, solver parameters $\lamsym > 0$, \(\lamalg > 0\),
	Zarantonello damping parameter \(\delta > 0\), and
	initial guesses
	\(
	u_0^{0,0}
	=
	u_0^{0, \nn},
	\)
	\(
	z_0^{0,0}
	=
	z_0^{0, \underline{\nu}}
	\in \XX_0.
	\)

	\noindent
	{\bfseries Adaptive loop:} For all $\ell = 0, 1, 2, \dots$, repeat the following steps~\ref{alg:solve_primal}--\ref{alg:refine}:
	\begin{enumerate}[label = (\Roman*),
			ref = {\normalfont(\Roman*)},
			font = \upshape]
		\item\label{alg:solve_primal} {\tt SOLVE \& ESTIMATE (PRIMAL).} For all $m = 1, 2,
			      3,
			      \dots$, repeat~\ref{alg:primal_a}--\ref{alg:primal_c}:
		      \begin{enumerate}[label = (\alph*),
				      ref = {\normalfont(\alph*)},
				      font = \upshape]
			      \item\label{alg:primal_a} Set $u_\ell^{m,0} \coloneqq u_\ell^{m-1, \nn}$
			            and define for theoretical reasons
			            $u_\ell^{m,\star} \coloneqq \Phi_\ell^{u}(\delta;u_\ell^{m-1,\nn})$.
			      \item For all $n=1,2,3, \dots$, repeat the following steps~\ref{alg:primal_i}--\ref{alg:primal_ii}:
			            \begin{enumerate}[label = (\roman*),
					            ref = {\normalfont(\roman*)},
					            font = \upshape]
				            \item\label{alg:primal_i}
				                  Compute \(u_\ell^{m, n} \coloneqq \Psi_\ell(u_\ell^{m,\star}; u_\ell^{m, n-1})\)
				                  and corresponding refinement indicators
				                  $\eta_\ell(T; u_\ell^{m, n})$ for all
				                  \(T \in \TT_\ell\).
				            \item\label{alg:primal_ii}
				                  Terminate $n$-loop
				                  and define $\nn[\ell, m] \coloneqq n$ if
				                  \begin{equation}\label{eq:n_stopping_criterion}
					                  \vvvert u_\ell^{m, n} - u_\ell^{m, n-1} \vvvert
					                  \le
					                  \lamalg \,
					                  \bigl[
						                  \lamsym \, \eta_{\ell}(u_\ell^{m, n}) + \vvvert u_\ell^{m, n} - u_\ell^{m, 0} \vvvert
						                  \bigr].
				                  \end{equation}
			            \end{enumerate}
			      \item\label{alg:primal_c}
			            Terminate $m$-loop and define \(\mm[\ell]
			            \coloneqq m\) if
			            \begin{equation}\label{eq:m_stopping_criterion}
				            \vvvert u_\ell^{m, \nn} - u_\ell^{m, 0} \vvvert
				            \le
				            \lamsym \, \eta_\ell(u_\ell^{m, \nn}).
			            \end{equation}
		      \end{enumerate}

		\item\label{alg:solve_dual} {\tt SOLVE \& ESTIMATE (DUAL).} For all \(\mu = 1, 2, 3, \ldots,
		      \) repeat~\ref{alg:dual_a}--\ref{alg:dual_c}:
		      \begin{enumerate}[label = (\alph*),
				      ref = {\normalfont(\alph*)},
				      font = \upshape]
			      \item\label{alg:dual_a} Set $z_\ell^{\mu,0} \coloneqq z_\ell^{\mu-1,
					            \underline{\nu}}$ and define for theoretical reasons
			            $z_\ell^{\mu,\star}
				            \coloneqq \Phi_\ell^{z}(\delta; z_\ell^{\mu-1,\underline{\nu}})$.
			      \item For all $\nu=1,2,3, \dots$, repeat the following steps \ref{alg_dual_i}--\ref{alg:dual_ii}:
			            \begin{enumerate}[label = (\roman*),
					            ref = {\normalfont(\roman*)},
					            font = \upshape]
				            \item\label{alg_dual_i} Compute
				                  \(
				                  z_\ell^{\mu, \nu} \coloneqq \Psi_\ell(z_\ell^{\mu, \star}; z_\ell^{\mu, \nu-1})
				                  \)
				                  and corresponding refinement indicators
				                  $\zeta_\ell(T; z_\ell^{\mu, \nu})$ for all $T \in
					                  \TT_\ell$.
				            \item\label{alg:dual_ii} Terminate $\nu$-loop
				                  and define $\underline{\nu}[\ell, \mu] \coloneqq
					                  \nu$ if
				                  \begin{equation}\label{eq:nu_stopping_criterion}
					                  \vvvert z_\ell^{\mu, \nu} - z_\ell^{\mu,\nu-1} \vvvert
					                  \le
					                  \lamalg \,
					                  \bigl[
						                  \lamsym \, \zeta_{\ell}(z_\ell^{\mu, \nu}) + \vvvert z_\ell^{\mu, \nu} - z_\ell^{\mu, 0} \vvvert
						                  \bigr].
				                  \end{equation}
			            \end{enumerate}
			      \item\label{alg:dual_c} Terminate $\mu$-loop and define
			            \(\underline{\mu}[\ell] \coloneqq \mu\) if
			            \begin{equation}\label{eq:mu_stopping_criterion}
				            \vvvert z_\ell^{\mu, \underline{\nu}} - z_\ell^{\mu, 0} \vvvert
				            \le
				            \lamsym \, \zeta_\ell(z_\ell^{\mu, \underline{\nu}}).
			            \end{equation}
		      \end{enumerate}

		\item\label{alg:mark}  {\tt MARK.} Determine sets
		      \begin{alignat*}{2}
			      \overline{\MM}_\ell^{u} & \in \M_\ell^u[\theta, u_\ell^{\mm,
				      \nn}]
			                              &                                            & \coloneqq
			      \{
				      \UU_\ell \subseteq \TT_\ell \colon
				      \theta \, \eta_\ell(u_\ell^{\mm,
					      \nn})^2
				      \le
				      \eta_\ell(\UU_\ell, u_\ell^{\mm, \nn})^2
			      \},
			      \\
			      \overline{\MM}_\ell^{z}
			                              & \in \M_\ell^z[\theta, z_\ell^{\mmu, \nnu}]
			                              &                                            & \coloneqq
			      \{
				      \UU_\ell \subseteq \TT_\ell \colon
				      \theta \, \zeta_\ell(z_\ell^{\underline{\mu},
					      \underline{\nu}})^2
				      \le
				      \zeta_\ell(\UU_\ell, z_\ell^{\underline{\mu},
					      \underline{\nu}})^2
			      \}
		      \end{alignat*}
		      satisfying the following Dörfler criterion \cite{d1996} with
		      quasi-minimal cardinality
		      \begin{equation}\label{eq:marking_criterion}
			      \# \overline{\MM}_\ell^{u}
			      \le
			      \Cmark \min_{\UU_\ell^\star \in \M_\ell^u[\theta, u_\ell^{\mm,\nn}]} \UU_\ell^\star
			      \quad
			      \text{and}
			      \quad
			      \# \overline{\MM}_\ell^{z}
			      \le
			      \Cmark \min_{\UU_\ell^\star \in \M_\ell^z[\theta, z_\ell^{\mmu, \nnu}]} \UU_\ell^\star.
		      \end{equation}
		      As in \cite{fpz2016}, define the set of marked elements
		      \(
		      \MM_\ell \coloneqq \MM_\ell^{u} \cup \MM_\ell^{z},
		      \)
		      where \(\MM_\ell^{u} \subseteq \overline{\MM}_\ell^{u}\) and
		      \(
		      \MM_\ell^{z} \subseteq \overline{\MM}_\ell^{z}
		      \)
		      satisfy
		      \(
		      \#\MM_\ell^{u}
		      =
		      \# \MM_\ell^{z} = \min \{\# \overline{\MM}_\ell^{u}, \#\overline{\MM}_\ell^{z}\}
		      \).

		\item\label{alg:refine} {\tt REFINE.} Generate the new mesh
		      \(
		      \TT_{\ell+1} \coloneqq  {\tt refine} (\MM_\ell, \TT_\ell)
		      \)
		      by NVB and define
		      \(
		      u_{\ell+1}^{0,0}
		      \coloneqq
		      u_{\ell+1}^{0,\nn}
		      \coloneqq
		      u_{\ell+1}^{0, \star}
		      \coloneqq
		      u_\ell^{\mm,\nn}
		      \)
		      and
		      \(
		      z_{\ell+1}^{0,0}
		      \coloneqq
		      z_{\ell+1}^{0,\underline{\nu}}
		      \coloneqq
		      z_{\ell+1}^{0, \star}
		      \coloneqq
		      z_\ell^{\underline{\mu}, \underline{\nu}}
		      \)
		      (nested iteration).
	\end{enumerate}

	\noindent
	{\bfseries Output:}  Sequences of successively refined triangulations
	$\TT_\ell$, successive discrete
	approximations $u_\ell^{m, n}$, \(z_\ell^{\mu, \nu}\), and corresponding error estimators
	$\eta_\ell(u_\ell^{m, n})$, $\zeta(z_\ell^{\mu, \nu})$.
\end{algorithm}
\begin{remark}
	\textup{(i)}\:
	Although the primal loop~\ref{alg:solve_primal} and dual
	loop~\ref{alg:solve_dual}
	in Algorithm~\ref{algorithm:afem} are displayed
	sequentially, they
	are
	independent of each other. Therefore, a practical implementation
	will
	realize these iterations simultaneously since the
	system
	matrix is the same
	(thanks to the symmetrization step).

	\textup{(ii)}\:
	In order to investigate the asymptotic behavior,
	it is reasonable to analyze
	Algorithm~\ref{algorithm:afem} in the present
	formulation with infinitely many steps. We note that a practical
	implementation will
	terminate with \(\elll \coloneqq \ell\)
	provided that the estimator product is smaller than a user-specified
	tolerance.
\end{remark}

For the analysis of Algorithm~\ref{algorithm:afem}, we define the
index set
\(\QQ \coloneqq \QQ^{u} \cup \QQ^{z}\) with
\begin{subequations}
	\begin{align}
		\QQ^{u} & \coloneqq \{(\ell, m, n) \in \N_0^3 \colon u_\ell^{m, n} \text{
				is used in Algorithm~\ref{algorithm:afem}}\}, \label{eq:primal_index_set}
		\\
		\QQ^{z} & \coloneqq \{(\ell, \mu, \nu) \in \N_0^3 \colon z_\ell^{\mu, \nu}
			\text{ is  used in Algorithm~\ref{algorithm:afem}}\}. \label{eq:dual_index_set}
	\end{align}
\end{subequations}
Furthermore, we require the following final indices and notice that these are
consistent with those defined in Algorithm~\ref{algorithm:afem}:
\begin{subequations}
	\begin{align}
		\elll        & \coloneqq \sup \{\ell \in \N_0 \colon (\ell, 0, 0) \in \QQ^{u}
			\text{ or } (\ell, 0, 0) \in \QQ^{z}\} \in \N_0 \cup \{\infty\},
		\\
		\mm[\ell]    & \coloneqq \sup \{m \in \N \colon (\ell, m, 0) \in
			\QQ^{u}\},
		\hphantom{\underline{\nu}[\ell, \mu]}\!
		\underline{\mu}[\ell] \coloneqq \sup \{\mu \in \N \colon (\ell, \mu,
			0) 	\in \QQ^{z}\},
		\\
		\nn[\ell, m] & \coloneqq \sup \{n \in \N \colon (\ell, m, n) \in
			\QQ^{u}\},
		\hphantom{\underline{\mu}[\ell]}
		\underline{\nu}[\ell, \mu] \coloneqq \sup \{\nu \in \N \colon (\ell,
			\mu, \nu) \in  \QQ^{z}\}.
	\end{align}
\end{subequations}
In addition, we set \(\kk[\ell] \coloneqq \max \{\mm[\ell], \underline{\mu}[\ell]\}\) as
well
as \(\jj[\ell, k] \coloneqq \max \{\nn[\ell, k], \underline{\nu}[\ell, k]\}\).

Finally, we introduce the total step counter $\vert \cdot, \cdot,
	\cdot \vert$ defined for all $(\ell, k, j) \in \QQ$ by
\begin{equation*}
	|\ell, k, j|
	= \sum_{\ell' = 0}^{\ell-1} \sum_{k' = 0}^{\kk[\ell']}
	\sum_{j'=0}^{\jj[\ell', k']} 1
	+ \sum_{k' = 0}^{k-1} \sum_{j'=0}^{\jj[\ell, k']}
	1
	+ \sum_{j'=0}^{j-1} 1.
\end{equation*}
This definition indeed provides a lexicographic ordering on
\(\QQ\), if the solver steps \ref{algorithm:afem}\ref{alg:solve_primal} for \(u_\ell^{m,
	n}\) and \ref{algorithm:afem}\ref{alg:solve_dual} for
\(z_\ell^{\mu, \nu}\) are done in \emph{parallel}. We note that one solver step
of an
optimal geometric multigrid method on graded meshes can be
performed in \(\OO(\# \TT_{\ell})\)
operations; see, e.g., \cite{wz2017, imps2022}. For given \(u_\ell^{m, n},
z_\ell^{\mu,\nu} \in \XX_\ell\), the simultaneous computation of the refinement
indicators \(\eta_{\ell}(T, u_\ell^{m, n})\) and
\(\zeta_{\ell}(T, z_\ell^{\mu, \nu})\)
requires \(\OO(\# \TT_{\ell})\) operations, hence the steps
\ref{algorithm:afem}\ref{alg:solve_primal}--\ref{alg:solve_dual} require \(\OO(\# \TT_{\ell})\)
operations as well. Furthermore, Dörfler marking can be
performed in \(\OO(\#
\TT_{\ell})\) operations; see, e.g., \cite{stevenson2007, pp2020}.
Therefore, the total work to compute \(u_\ell^{m, n}\) and
\(z_\ell^{\mu,\nu}\) is (up to a constant) given by
\begin{equation}\label{eq:cost}
	\mathtt{cost}(\ell, k, j)
	\coloneqq
	\! \! \! \! \! \!
	\sum_{
		\substack{(\ell', m', n') \in \QQ^u
			\\
			|\ell', m', n'| \le |\ell,k,j|
		}
	}
	\! \! \! \! \! \# \TT_{\ell^\prime} \
	+
	\! \! \! \sum_{
		\substack{
			(\ell', \mu', \nu') \in \QQ^z
			\\
			|\ell', \mu', \nu'| \le |\ell,k,j|
		}
	}
	\! \! \! \! \! \# \TT_{\ell^\prime}
	\simeq
	\sum_{
		\substack{
			(\ell', k', j') \in \QQ
			\\
			|\ell', k', j'| \le |\ell,k,j|
		}
	}
	\! \! \! \# \TT_{\ell^\prime}.
\end{equation}

Since \(\# \QQ = \infty\), we have either \(\elll =
\infty\), or
\(\kk[\elll] = \infty\), or \(\jj[\elll, \kk] = \infty\).
A further observation about Algorithm~\ref{algorithm:afem} is
that the nested algebraic
solver loop within the Zarantonello loop is guaranteed to
terminate, and the latter case
\(\jj[\elll, \kk] = \infty\) is therefore excluded.
\begin{lemma}[finite termination of algebraic solver
			{\cite[Lemma~3.2]{bhimps2023}}]
	\label{lemma:termination-j}
	Independently of the algorithmic parameters \(\delta\), $\theta$,
	$\lamsym$, and $\lamalg$, the innermost \(n\)- and \(\nu\)-loops
	of
	Algorithm~\ref{algorithm:afem} always terminate.
	In particular,
	$\jj[\ell,k] < \infty$ for all $(\ell,k,0) \in \QQ$. \qed
\end{lemma}

\section{\textsl{A~posteriori} error analysis}\label{section:aposteriori}

Algorithm~\ref{algorithm:afem} does not provide
the exact algebraic solutions \(u_\ell^{m, \star}\) and
\(z_\ell^{\mu, \star}\)
to~\eqref{eq:Zarantonello-iterations} but instead uses
an inexact algebraic solver.
However,
the following result from \cite{bhimps2023}
applies to the primal and the dual problem
alike and
shows that these inexact Zarantonello
iterations remain
contractions
except for the final iterate on each mesh
(see also \cite{bhimps2023b} for an extended
version).
\begin{lemma}[contraction of inexact Zarantonello iteration
			{\cite[Lemma~5.1]{bhimps2023}}]
	\label{lem:inexact_Zarantonello_contraction}
	Choose any damping parameter
	\(0 < \delta < \delta^\star = 2 \alpha / L^2\)
	to ensure the
	contraction~\eqref{eq:Zarantonello-contraction} of
	the Zarantonello iteration and
	\begin{equation}\label{eq:qsymm}
		0
		<
		\lamalg^\star
		<
		\frac{(1-\qsym) (1-\qalg)}{4 \qalg}
		\quad\text{such that}\quad
		0
		<
		\qsymm
		\coloneqq
		\frac{\qsym + 2 \, \frac{\qalg}{1-\qalg} \, \lamalg^\star}{1-2 \,
			\frac{\qalg}{1-\qalg} \, \lamalg^\star}
		<
		1.
	\end{equation}
	Then, for arbitrary \(\lamsym > 0\) and any \(0 < \lamalg \le
	\lamalg^\star\),
	we have for all  \((\ell, m, \nn) \in \QQ^{u}\)
	with \(1 \le m < \mm[\ell]\) and all \((\ell, \mu, \nnu) \in \QQ^{z}\) with
	\(1 \le \mu < \mmu[\ell]\) that
	\begin{equation}\label{eq:inexact_Zarantonello_contraction}
		\vvvert u_\ell^\star - u_\ell^{m, \nn} \vvvert \le \qsymm \, \vvvert u_\ell^\star -
		u_\ell^{m-1, \nn} \vvvert
		\quad
		\text{and}
		\quad
		\vvvert z_\ell^\star - z_\ell^{\mu, \nnu} \vvvert
		\le
		\qsymm \, \vvvert z_\ell^\star -
		z_\ell^{\mu-1, \nnu} \vvvert.
	\end{equation}
	Moreover, for \(m = \mm[\ell]\) resp. \(\mu = \mmu[\ell]\), it holds that
	\begin{equation}\label{eq2:inexact_Zarantonello_contraction}
		\begin{aligned}
			\vvvert u_\ell^\star - u_\ell^{\mm, \nn} \vvvert
			 & \le
			\qsym \, \vvvert u_\ell^\star - u_\ell^{\mm-1, \nn} \vvvert
			+
			\frac{2 \, \qalg}{1-\qalg} \, \lamalg \, \lamsym \,
			\eta_{\ell}(u_\ell^{\mm, \nn}),
			\\
			\vvvert z_\ell^\star - z_\ell^{\mmu, \nnu} \vvvert
			 & \le
			\qsym \, \vvvert z_\ell^\star - z_\ell^{\mmu - 1, \nnu} \vvvert
			+
			\frac{2 \, \qalg}{1-\qalg} \, \lamalg \, \lamsym \,
			\zeta_{\ell}(z_\ell^{\mmu, \nnu}). \qed
		\end{aligned}
	\end{equation}
\end{lemma}
The subsequent lemma gathers \textsl{a~posteriori} error estimates
following directly from the corresponding contraction of the symmetrization,
algebraic solver, and the inexact Zarantonello iteration.
Further details of the
elementary proof are omitted.
\begin{lemma}[stability and \textsl{a~posteriori} error control]
	For all \((\ell, m,
	0) \in \QQ^u\),
	contraction~\eqref{eq:Zarantonello-contraction} shows
	\begin{equation}\label{eq:aposteriori_Zarantonello}
		\frac{1-\qsym}{\qsym} \, \vvvert u_\ell^{\star} - u_\ell^{m, \star} \vvvert
		\le
		\vvvert u_\ell^{m, \star} - u_\ell^{m-1,
			\nn} \vvvert
		\le
		(1+\qsym) \, \vvvert u_\ell^{\star} -
		u_\ell^{m-1, \nn} \vvvert.
	\end{equation}
	Analogously, for all \((\ell, m, n) \in \QQ^u\)
	the contraction~\eqref{eq:algebra_contraction} ensures
	\begin{equation}\label{eq:aposteriori_algebra}
		\frac{1-\qalg}{\qalg} \, \vvvert u_\ell^{m, \star} - u_\ell^{m, n} \vvvert
		\le
		\vvvert u_\ell^{m, n} - u_\ell^{m, n-1} \vvvert
		\le
		(1+\qalg) \, \vvvert u_\ell^{m, \star} - u_\ell^{m, n-1} \vvvert.
	\end{equation}
	For all
	\((\ell, m, \nn) \in \QQ^{u}\) with \(m < \mm[\ell]\), the
	contraction~\eqref{eq:inexact_Zarantonello_contraction} leads to
	\begin{equation}\label{eq:aposteriori_inexactZarantonello}
		\frac{1-\qsymm}{\qsymm} \, \vvvert u_\ell^{\star} - u_\ell^{m, \nn} \vvvert
		\le
		\vvvert u_\ell^{m, \nn} - u_\ell^{m-1, \nn} \vvvert
		\le
		(1+\qsymm) \, \vvvert u_\ell^{\star} - u_\ell^{m-1, \nn} \vvvert.
	\end{equation}
	The analogous estimates are also valid for the dual variable. \qed
\end{lemma}
Finally, the following lemma shows that in the case of finitely many mesh-refinement steps,
the Zarantonello iteration does not terminate and one of the two exact continuous solutions is already the discrete solution to~\eqref{eq:discrete_formulation}.
\begin{lemma}
	[case of finite mesh-refinement steps]
	\label{cor:earlystop}
	Suppose that the inexact Zarantonello iteration satisfies
	contraction~\eqref{eq:inexact_Zarantonello_contraction}
	and that \(\eta\) and \(\zeta\) satisfy
	\eqref{axiom:stability}--\eqref{axiom:reliability}.
	If \(\elll < \infty\), then \(\kk[\elll] =
	\infty\) and
	\(\eta_{\elll}(u_{\elll}^{\star}) = 0\) (so that \(u^\star =
	u_{\elll}^\star\)) or \(\zeta_{\elll}(z_{\elll}^{\star}) = 0\)
	(so
	that \(z^{\star} = z_{\elll}^{\star}\)).
\end{lemma}
\begin{proof}
	By Lemma~\ref{lemma:termination-j}, we have \(\jj[\ell, k] < \infty\). If \(\elll < \infty\),
	then \(\kk[\elll] = \infty\) and, hence,
	\begin{equation}\label{eq:failed_k_stopping}
		\eta_{\elll}(u_{\elll}^{m, \nn})
		\eqreff{eq:m_stopping_criterion}<
		\lamsym^{-1} \, \vvvert u_{\elll}^{m, \nn} -
		u_{\elll}^{m-1, \nn} \vvvert
		\quad \text{for all \(m\in \N\)}
	\end{equation}
	or
	\begin{equation}\label{eq2:failed_k_stopping}
		\zeta_{\elll}(z_{\elll}^{\mu, \nnu})
		\eqreff{eq:mu_stopping_criterion}<
		\lamsym^{-1} \, \vvvert z_{\elll}^{\mu, \nnu} -
		z_{\elll}^{\mu-1, \nnu} \vvvert
		\quad \text{for all \(\mu \in \N\).}
	\end{equation}
	If \eqref{eq:failed_k_stopping} holds,
	then
	the inexact
	Zarantonello iterates \(u_{\elll}^{m, \nn}\) are convergent
	with limit
	\(u_{\elll}^\star\) and we
	obtain by stability \eqref{axiom:stability} that
	\begin{align*}
		\eta_{\elll}(u_{\elll}^\star)
		 & \eqreff{axiom:stability}\le
		\eta_{\elll}(u_{\elll}^{m, \nn}) + \Cstab \, \vvvert u_{\elll}^\star
		- u_{\elll}^{m, \nn} \vvvert
		\eqreff{eq:failed_k_stopping}\lesssim
		\vvvert u_{\elll}^{m, \nn} - u_{\elll}^{m-1, \nn} \vvvert
		\xrightarrow{m \to \infty}
		0.
	\end{align*}
	This proves that \(\eta_{\elll}(u_{\elll}^\star) = 0\),
	and we infer from reliability \eqref{axiom:reliability}
	that \(u_{\elll}^\star = u^\star\). The same arguments
	apply to
	\(z_{\elll}^\star\) in the case of \eqref{eq2:failed_k_stopping}.
\end{proof}
Due to the contraction of the inexact Zarantonello iteration~\eqref{eq:inexact_Zarantonello_contraction}, we have
the following a~posteriori error estimates for the final iterates.
\begin{lemma}
	[stability of final iterates]
	\label{lem:auxiliary_estimates}
	Suppose that the inexact Zarantonello iteration satisfies
	\eqref{eq:inexact_Zarantonello_contraction}.
	Then, for all \((\ell+1, \mm, \nn) \in \QQ^u\) and \((\ell+1,
	\mmu, \nnu) \in \QQ^{z}\), there holds
	\begin{alignat}{2}
		\vvvert u_{\ell+1}^\star -
		u_{\ell+1}^{\mm-1,\nn} \vvvert
		 & \le
		\vvvert u_{\ell+1}^\star - u_\ell^{\mm,\nn} \vvvert,
		\quad
		\vvvert z_{\ell+1}^\star -
		z_{\ell+1}^{\mmu-1,\nnu} \vvvert
		 &     & \le
		\vvvert z_{\ell+1}^\star - z_\ell^{\mmu,\nnu} \vvvert,
		\label{eq2:stability}
		\\
		\vvvert u_{\ell+1}^{\mm,\nn} - u_\ell^{\mm,\nn} \vvvert
		 & \le
		4 \, \vvvert u_{\ell+1}^\star - u_\ell^{\mm,\nn} \vvvert,
		\quad
		\vvvert z_{\ell+1}^{\mmu,\nnu} - z_\ell^{\mmu,\nnu} \vvvert
		 &     & \le
		4 \, \vvvert z_{\ell+1}^\star - z_\ell^{\mmu,\nnu} \vvvert,
		\label{eq3:stability}
		\\
		\vvvert u_{\ell}^{\mm,\nn} - u_{\ell}^{\mm-1,\nn} \vvvert
		 & \le
		4 \,
		\vvvert u_{\ell}^\star -	u_{\ell}^{\mm-1,\nn} \vvvert,
		\quad
		\vvvert z_{\ell}^{\mmu,\nnu} -
		z_{\ell}^{\mmu-1,\nnu} \vvvert
		 &     & \le
		4 \,
		\vvvert z_{\ell}^\star - z_{\ell}^{\mmu-1,\nnu} \vvvert.
		\label{eq4:stability}
	\end{alignat}
\end{lemma}

\begin{proof}
	For $(\ell+1, \mm, \nn) \in
		\QQ^u$, nested iteration
	$u_{\ell+1}^{0,\nn} = u_\ell^{\mm, \nn}$ together
	with the
	contraction of the inexact Zarantonello
	iteration~\eqref{eq:inexact_Zarantonello_contraction}
	and $\mm[\ell+1] \ge 1$ prove~\eqref{eq2:stability} by
	\begin{equation*}
		\vvvert u_{\ell+1}^\star - u_{\ell+1}^{\mm-1, \nn} \vvvert
		\eqreff*{eq:inexact_Zarantonello_contraction}
		\le
		\qsymm^{\mm[\ell+1]-1} \, \vvvert u_{\ell+1}^\star -
		u_{\ell+1}^{0,\nn} \vvvert
		\le
		\vvvert u_{\ell+1}^\star - u_\ell^{\mm, \nn} \vvvert.
	\end{equation*}
	Let $(\ell, \mm, \nn) \in \QQ^u$.
	Contraction of the algebraic
	solver~\eqref{eq:algebra_contraction}, the fact
	$\nn[\ell,\mm] \ge 1$,
	and nested iteration $u_\ell^{\mm,0} =
		u_\ell^{\mm-1,\nn}$ show that
	\begin{equation}\label{eq:step1:1*}
		\vvvert u_\ell^{\mm, \star} - u_\ell^{\mm, \nn} \vvvert
		\eqreff{eq:algebra_contraction}\le
		\qalg^{\nn[\ell, \mm]} \,
		\vvvert u_\ell^{\mm,\star} - u_\ell^{\mm,0} \vvvert
		\le
		\qalg \, \vvvert u_\ell^{\mm,\star} -
		u_\ell^{\mm-1,\nn} \vvvert.
	\end{equation}
	This and with the contraction of the exact Zarantonello
	iteration~\eqref{eq:Zarantonello-contraction}
	result in
	\begin{equation}\label{eq1:stability}
		\begin{aligned}
			\vvvert u_\ell^\star - u_\ell^{\mm, \nn} \vvvert
			 & \le
			\vvvert u_\ell^\star - u_\ell^{\mm,\star} \vvvert +
			\vvvert u_\ell^{\mm,\star} - u_\ell^{\mm, \nn} \vvvert
			\\&
			\eqreff*{eq:step1:1*}\le
			(1+\qalg) \, \vvvert u_\ell^\star - u_\ell^{\mm,\star} \vvvert +
			\qalg \,
			\vvvert u_\ell^\star - u_\ell^{\mm-1,\nn} \vvvert
			\\&
			\eqreff*{eq:Zarantonello-contraction}\le
			\bigl[ (1+\qalg) \qsym + \qalg \bigr] \, \vvvert u_\ell^\star -
			u_\ell^{\mm-1,\nn} \vvvert
			\le
			3 \, \vvvert u_\ell^\star - u_\ell^{\mm-1,\nn} \vvvert.
		\end{aligned}
	\end{equation}
	Consequently,
	the combination of~\eqref{eq1:stability}
	and~\eqref{eq2:stability} validates
	\eqref{eq3:stability} via
	\begin{align*}
		\vvvert u_{\ell+1}^{\mm, \nn} - u_\ell^{\mm, \nn} \vvvert
		 & \le
		\vvvert u_{\ell+1}^\star - u_{\ell+1}^{\mm, \nn} \vvvert
		+ \vvvert u_{\ell+1}^\star - u_\ell^{\mm, \nn} \vvvert
		\\&
		\eqreff*{eq1:stability}
		\le
		3 \, \vvvert u_{\ell+1}^\star - u_{\ell+1}^{\mm-1, \nn} \vvvert +
		\vvvert u_{\ell+1}^\star - u_\ell^{\mm, \nn} \vvvert
		\eqreff{eq2:stability}
		\le
		4 \, \vvvert u_{\ell+1}^\star - u_\ell^{\mm, \nn} \vvvert.
	\end{align*}
	The estimate~\eqref{eq1:stability} also implies \eqref{eq4:stability},
	because
	\begin{equation*}
		\vvvert u_{\ell}^{\mm, \nn} - u_{\ell}^{\mm-1,
				\nn} \vvvert
		\le
		\vvvert u_{\ell}^\star - u_{\ell}^{\mm, \nn} \vvvert +
		\vvvert u_{\ell}^\star -  u_{\ell}^{\mm-1, \nn} \vvvert
		\eqreff{eq1:stability}\le
		4 \, \vvvert u_{\ell}^\star -  u_{\ell}^{\mm-1, \nn} \vvvert.
	\end{equation*}
	The same arguments prove the estimates for the dual variable
	and conclude the proof.
\end{proof}
The subsequent lemma states the estimator reduction for only one of the
two error estimators. This poses a significant challenge in the
proof of full
linear convergence due to the required contraction of the nonlinear
quasi-error product in Lemma~\ref{lem:contraction_remainder} below.
\begin{lemma}
	[estimator reduction and stability]
	\label{lemma:estimator_reduction}
	Define the constant
	\(0 < q(\theta)
	\coloneqq
	\bigl[1 - (1-\qred^{2} ) \, \theta\bigr]^{1/2}
	< 1\)
	and suppose that the estimators \(\eta\) and \(\zeta\)
	satisfy
	\eqref{axiom:stability}--\eqref{axiom:reduction}.
	If the primal error estimator satisfies the
	Dörfler criterion, i.e., \(\MM_\ell^u =
	\overline{\MM}_\ell^u \subseteq \MM_\ell\) in
	Algorithm~\ref{algorithm:afem}\ref{alg:mark},
	then
	\begin{equation}\label{eq1:estimator_reduction}
		\begin{aligned}
			\eta_{\ell+1}(u_{\ell+1}^{\mm, \nn})
			      & \le
			q(\theta) \, \eta_\ell(u_\ell^{\mm, \nn})
			+
			4 \, \Cstab \, \vvvert u_{\ell+1}^\star - u_\ell^{\mm,
				\nn} \vvvert
			\quad &     & \text{for all} \quad
			(\ell+1, \mm, \nn) \in \QQ^u,
			\\
			\zeta_{\ell+1}(z_{\ell+1}^{\mmu, \nnu})
			      & \le
			\zeta_\ell(z_\ell^{\mmu, \nnu})
			+
			4 \, \Cstab \, \vvvert z_{\ell+1}^\star - z_\ell^{\mmu,
				\nnu} \vvvert
			\quad &     & \text{for all} \quad
			(\ell+1, \mmu, \nnu) \in \QQ^z.
		\end{aligned}
	\end{equation}
	If the dual error estimator satisfies the Dörfler
	criterion,
	i.e., \(\MM_\ell^z =
	\overline{\MM}_\ell^z \subseteq \MM_\ell\) in
	Algorithm~\ref{algorithm:afem}\ref{alg:mark}, then
	\begin{equation}\label{eq2:estimator_reduction}
		\begin{aligned}
			\eta_{\ell+1}(u_{\ell+1}^{\mm, \nn})
			      & \le
			\eta_\ell(u_\ell^{\mm, \nn})
			+
			4 \, \Cstab \, \vvvert u_{\ell+1}^\star - u_\ell^{\mm,
				\nn} \vvvert
			\quad &     & \text{for all} \quad
			(\ell+1, \mm, \nn) \in \QQ^u,
			\\
			\zeta_{\ell+1}(z_{\ell+1}^{\mmu, \nnu})
			      & \le
			q(\theta) \, \zeta_\ell(z_\ell^{\mmu, \nnu})
			+
			4 \, \Cstab \, \vvvert z_{\ell+1}^\star - z_\ell^{\mmu,
				\nnu} \vvvert
			\quad &     & \text{for all} \quad
			(\ell+1, \mmu, \nnu) \in \QQ^z.
		\end{aligned}
	\end{equation}
\end{lemma}
\begin{proof}
	For $(\ell+1,0, 0) \in \QQ^{u}$, stability~\eqref{axiom:stability}
	and reduction~\eqref{axiom:reduction}
	yield that
	\begin{equation} \label{eq5:combined-contraction}
		\begin{aligned}
			\eta_{\ell+1}(u_\ell^{\mm, \nn})^2
			 & =
			\eta_{\ell+1}(\TT_{\ell+1} \cap \TT_\ell ; u_\ell^{\mm,
				\nn})^2
			+
			\eta_{\ell+1}(\TT_{\ell+1} \backslash \TT_\ell ;
			u_\ell^{\mm, \nn})^2
			\\&
			\le
			\,
			\eta_\ell(\TT_{\ell+1} \cap \TT_\ell ; u_\ell^{\mm,
				\nn})^2
			+
			\qred^{2} \, \eta_\ell(\TT_\ell \backslash \TT_{\ell+1} ;
			u_\ell^{\mm, \nn})^2
			\\&
			=
			\,
			\eta_\ell(u_\ell^{\mm, \nn})^2
			-
			(1 - \qred^{2} ) \, \eta_\ell(\TT_\ell \backslash \TT_{\ell+1} ;
			u_\ell^{\mm, \nn})^2.
		\end{aligned}
	\end{equation}
	The D\"orfler marking in
	Algorithm~\ref{algorithm:afem}\ref{alg:mark} for the primal
	error estimator \(\eta\) and \(\MM_{\ell} \subseteq \TT_{\ell}
	\setminus \TT_{\ell+1}\)
	prove the contraction in~\eqref{eq1:estimator_reduction}
	\begin{equation}\label{eq3:estimator_reduction}
		\eta_{\ell+1}(u_\ell^{\mm, \nn})^2
		\le
		\eta_\ell(u_\ell^{\mm, \nn})^2
		-
		(1 - \qred^{2} ) \, \eta_\ell(\MM_\ell ; u_\ell^{\mm,
			\nn})^2
		\le
		q(\theta)^{2} \, \eta_\ell(u_\ell^{\mm,
			\nn})^2.
	\end{equation}
	For $(\ell+1, \mm, \nn) \in \QQ^u$, this
	and \eqref{eq3:stability} lead to
	\begin{align*}
		\begin{split}
			\eta_{\ell+1}(u_{\ell+1}^{\mm, \nn})
			 & \eqreff*{axiom:stability}
			\le
			\eta_{\ell+1}(u_\ell^{\mm, \nn})
			+
			\Cstab \, \vvvert u_{\ell+1}^{\mm, \nn} -
				u_\ell^{\mm, \nn} \vvvert
			\\
			 & \eqreff*{eq3:estimator_reduction}
			\le
			q(\theta) \, \eta_\ell(u_\ell^{\mm, \nn})
			+
			\Cstab \, \vvvert u_{\ell+1}^{\mm, \nn} -
				u_\ell^{\mm, \nn} \vvvert
			\\&
			\eqreff*{eq3:stability}
			\le
			q(\theta) \, \eta_\ell(u_\ell^{\mm, \nn})
			+
			4 \, \Cstab \, \vvvert u_{\ell+1}^\star - u_\ell^{\mm,
					\nn} \vvvert.
		\end{split}
	\end{align*}
	For \((\ell+1,
	\mmu, \nnu) \in \QQ^z\), we argue analogously to
	\eqref{eq5:combined-contraction} in order to obtain
	that
	\(\zeta_{\ell+1}(z_{\ell}^{\mmu, \nnu}) \le
	\zeta_{\ell}(z_{\ell}^{\mmu, \nnu})\).
	Together with \eqref{eq3:stability}, it follows that
	\begin{equation*}
		\zeta_{\ell+1}(z_{\ell+1}^{\mmu, \nnu})
		\eqreff{axiom:stability}\le
		\zeta_{\ell+1}(z_\ell^{\mmu, \nnu})
		+
		\Cstab \, \vvvert z_{\ell+1}^{\mmu, \nnu} -
			z_\ell^{\mmu, \nnu} \vvvert
		\eqreff{eq3:stability}\le
		\zeta_\ell(z_\ell^{\mmu, \nnu})
		+
		4 \, \Cstab \, \vvvert z_{\ell+1}^\star - z_\ell^{\mmu,
				\nnu} \vvvert.
	\end{equation*}
	The proof holds verbatim in the case of Dörfler marking for the dual error
	estimator, albeit with reversed roles. This concludes the proof.
\end{proof}

%

\section{Full linear convergence}\label{section:main_results}
This section presents full linear convergence of
Algorithm~\ref{algorithm:afem} as the first main result of this work.
Recall the goal-error estimate from~\eqref{eq:goalError:Estimate}
motivating the product structure of the respective primal and
dual error
components. Thus, we define the quasi-errors
\begin{subequations}\label{eq:quasi-error}
	\begin{align}
		\Eta_\ell^{m,n}
		      & \coloneqq
		\vvvert u_\ell^\star - u_\ell^{m,n} \vvvert
		+ \vvvert u_\ell^{m, \star} - u_\ell^{m,n} \vvvert
		+ \eta_{\ell}(u_\ell^{m,n})
		\quad &           & \text{for all} \quad (\ell, m, n) \in \QQ^u,
		\label{eq:quasi-error-primal}
		\\
		\Zeta_\ell^{\mu,\nu}
		      & \coloneqq
		\vvvert z_\ell^\star - z_\ell^{\mu,\nu} \vvvert
		+ \vvvert z_\ell^{\mu, \star} - z_\ell^{\mu,\nu} \vvvert
		+ \zeta_{\ell}(z_\ell^{\mu,\nu})
		      &           & \text{for all} \quad (\ell, \mu, \nu) \in \QQ^z.
		\label{eq:quasi-error-dual}
	\end{align}
\end{subequations}
The quasi-errors naturally extend to the full index set
\((\ell, k,
j) \in \QQ\) by
\begin{equation}\label{eq:quasi-error-extension}
	\begin{aligned}
		\Eta_\ell^{k, j}
		 & \coloneqq
		\begin{cases}
			\Eta_{\ell}^{k, \nn}
			\  & \text{if} \
			(\ell, k, 0) \in \QQ^u \text{ but } (\ell, k, j) \notin \QQ^u,
			\\
			\Eta_{\ell}^{\mm, \nn}
			\  & \text{if} \
			(\ell, k, 0) \notin \QQ^u,
		\end{cases}
		\\
		\Zeta_\ell^{k, j}
		 & \coloneqq
		\begin{cases}
			\Zeta_{\ell}^{k, \nnu}
			\  & \text{if} \
			(\ell, k, 0) \in \QQ^z
			\text{ but } (\ell, k, j) \notin \QQ^z,
			\\
			\Zeta_{\ell}^{\mmu, \nnu}
			\  & \text{if} \
			(\ell, k, 0) \notin \QQ^z.
		\end{cases}
	\end{aligned}
\end{equation}
The following theorem asserts full linear convergence of the quasi-error
product.
\begin{theorem}[full linear convergence]\label{lem:full_linear_convergence}
	Suppose that the estimators \(\eta\) and \(\zeta\) satisfy
	\eqref{axiom:stability}--\eqref{axiom:reliability} and
	\eqref{axiom:qm} and suppose~\eqref{axiom:orthogonality}.
	Recall
	\(\lamalg^\star\) and \(\qsymm\) from
	Lemma~\ref{lem:inexact_Zarantonello_contraction}. With the constant
	\(
	q(\theta)
	\)
	from Lemma~\ref{lemma:estimator_reduction}
	and
	\(
	\overline{q} \coloneqq \max \{q(\theta)^{1/2},
	(1 +\qsym)/2\}
	< 1
	\),
	let
	\begin{equation}\label{eq:choice_lambda}
		0
		<
		\lambda^\star
		\coloneqq
		\frac{(1-\qalg) \, (\overline{q} - \qsym) \, (1 - \overline{q})}{10
			\, \qalg \, \Cstab}.
	\end{equation}
	Then, for arbitrary
	marking parameter \(0 < \theta \le 1\) and
	any solver parameters \(\lamsym > 0\) and
	\(0 < \lamalg \le \lamalg^\star\) with \(\lamsym \lamalg \le
	\lambda^\star\),
	Algorithm~\ref{algorithm:afem} guarantees full linear convergence: There exist constants
	\(\Clin \ge 1\) and \(0 < \qlin < 1\) such that the
	quasi-error product satisfies,
	for all \((\ell, k, j),
	(\ell^\prime, k^\prime, j^\prime) \in \QQ\) with
	\(|\ell^\prime, k^\prime, j^\prime| \le |\ell, k, j|\)
	\begin{equation}\label{eq:full_linear_convergence}
		\Eta_\ell^{k,j} \, \Zeta_\ell^{k, j}
		\le
		\Clin \, \qlin^{|\ell, k, j| - |\ell^\prime, k^\prime, j^\prime|} \,
		\Eta_{\ell^\prime}^{k^\prime,j^\prime} \,
		\Zeta_{\ell^\prime}^{k^\prime,j^\prime}.
	\end{equation}
	The constants \(\Clin\) and \(\qlin\) depend only on \(\Cstab\),
	\(\Crel\), \(\Cmon\),
	\(C_{\textup{orth}}\),
	\(\Ccea\), \(\theta\), $\qred$,
	$\qsymm$, $\qsym$, $\qalg$,
	$\lamsym$, and $\lamalg$.
\end{theorem}
Three lemmas are required to prove Theorem~\ref{lem:full_linear_convergence}.
The characterization of \(R\)-linear convergence
from~\cite[Lemma~5 and 10]{fps2023} is the primary tool for the
proof
of Theorem~\ref{lem:full_linear_convergence};
see \eqref{eq:summability_criterion} below.
The proof of Theorem~\ref{lem:full_linear_convergence} departs with
the
contraction of the quasi-error for the final iterates of the
inexact Zarantonello loop up to a remainder on the mesh
level
\(\ell\). To this end, we define the simplified weighted
quasi-error
\begin{equation}\label{eq:def:Eta_ell}
	\Eta_\ell
	\coloneqq
	\bigl[\vvvert u_\ell^\star - u_\ell^{\mm, \nn} \vvvert + \gamma \,
		\eta_\ell(u_\ell^{\mm, \nn})\bigr],
	\,
	\Zeta_\ell
	\coloneqq \bigl[\vvvert z_\ell^\star -
		z_\ell^{\mmu, \nnu} \vvvert + \gamma \,
		\zeta_\ell(z_\ell^{\mmu, \nnu})\bigr]
	\,
	\text{ for all \((\ell,\kk,\jj) \in \QQ\)},
\end{equation}
where \(\gamma > 0\) is a free parameter chosen
in~\eqref{eq:choice_gamma_lambda}
below.
This quasi-error quantity satisfies
contraction up to a tail-summable remainder due to estimator
reduction~\eqref{eq1:estimator_reduction}--\eqref{eq2:estimator_reduction}.
\begin{lemma}[contraction in mesh level up to tail-summable remainder]
	\label{lem:contraction_remainder}
	Under the assumptions of Theorem~\ref{lem:full_linear_convergence}, there
	exists \(0 < q < 1\) such that the quasi-error
	product \(\Eta_\ell \, \Zeta_{\ell}\)
	from~\eqref{eq:def:Eta_ell} satisfies contraction up to
	a remainder \(R_\ell \ge 0\),
	\begin{align}\label{eq:contraction_remainder}
		\Eta_{\ell+1} \, \Zeta_{\ell+1}
		\le
		q \, \Eta_\ell \, \Zeta_{\ell}  + q \, R_\ell
		\quad \text{for all } (\ell+1, \kk,\jj) \in \QQ.
	\end{align}
	The remainder \(R_\ell\) satisfies
	\begin{equation}\label{eq:summability_remainder}
		R_{\ell+M} \lesssim \Eta_{\ell} \, \Zeta_{\ell}
		\, \text{ and } \,
		\sum_{\ell'=\ell}^{\ell + M} R_{\ell'}^2
		\lesssim
		(M + 1)^{1-\delta} \,
		\Eta_{\ell}^{2} \, \Zeta_{\ell}^{2}
		\,
		\text{ for all
			\(\ell, M \in \N_0\) with \(\ell + M < \elll\)}.
	\end{equation}
\end{lemma}
\begin{proof}
	The proof consists of four steps.

	\textbf{Step~1 (choice of constants).}
	Recall the constants
	\(0 < q(\theta) < 1\) from Lemma~\ref{lemma:estimator_reduction} and
	\(\lambda^\star > 0\) and \(0 < \overline{q} < 1\) defined in the statement of
	Theorem~\ref{lem:full_linear_convergence} and
	define the constants
	\begin{equation*}
		C(\gamma, \lambda)
		\coloneqq
		1 + \frac{2 \, \qalg}{1 - \qalg} \, \frac{\lambda}{\gamma} >
		1
		\text{ and }
		0 < \! q_{\mathrm{ctr}}
		\coloneqq \!
		\max
		\bigl\{
		\qsym + 4 \Cstab \, C(\gamma, \lambda) \, \gamma, \,
		q(\theta) C(\gamma, \lambda)
		\bigr\}.
	\end{equation*}
	Elementary calculations show that
	the choice of
	\begin{equation}\label{eq:choice_gamma_lambda}
		\gamma
		\coloneqq
		\frac{\overline{q} \, (\overline{q}-\qsym)}{4 \, \Cstab}
		< 1
	\end{equation}
	ensures
	\(
	\qsym \, C(\gamma, \lambda) + 4 \, \Cstab \, \gamma \, C(\gamma, \lambda)^2 < 1
	\)
	as well as, for all \(0 < \lambda < \lambda^\star\),
	\begin{equation}\label{eq:estimate_C}
		C(\gamma, \lambda) = 1 + \frac{2 \, \qalg}{1-\qalg}
		\frac{\lambda}{\gamma} < 1 + \frac{1-\overline{q}}{\overline{q}} =
		\frac{1}{\overline{q}} \le \frac{1}{q(\theta)^{1/2}}.
	\end{equation}
	Consequently,
	we have
	\(q(\theta) \, C(\gamma, \lambda)^2 < 1\)
	and thus
	\(0 < q_{\mathrm{ctr}}^\prime \coloneqq C(\gamma, \lambda) \, q_{\mathrm{ctr}} < 1\)
	and \(q_{\mathrm{ctr}} < 1\).

	\textbf{Step~2 (contraction of \(\boldsymbol{\Eta_\ell}\) and
		\(\boldsymbol{\Zeta_\ell}\)).}
	Abbreviate \(\lambda \coloneqq
	\lamalg \,
	\lamsym\). Recall that marking in
	Algorithm~\ref{algorithm:afem}\ref{alg:mark} ensures that the
	estimate~\eqref{eq1:estimator_reduction}
	or~\eqref{eq2:estimator_reduction}
	hold. If \eqref{eq1:estimator_reduction} is satisfied,
	the quasi-contraction of the inexact
	Zarantonello
	iteration~\eqref{eq2:inexact_Zarantonello_contraction} for the final
	iterate, the stability estimate~\eqref{eq2:stability}, and the
	estimator
	reduction~\eqref{eq1:estimator_reduction} lead,
	for all \((\ell+1, \kk, \jj) \in \QQ^{u}\), to
	\begin{equation}\label{eq:contraction-primal}
		\begin{aligned}
			\Eta_{\ell+1}
			 & \eqreff*{eq2:inexact_Zarantonello_contraction}
			\le
			\qsym \, \vvvert u_{\ell+1}^\star - u_{\ell+1}^{\mm-1, \nn} \vvvert
			+
			C(\gamma, \lambda) \, \gamma \,
			\eta_{\ell+1}(u_{\ell+1}^{\mm, \nn})
			\\&
			\eqreff*{eq2:stability}
			\le
			\qsym \, \vvvert u_{\ell+1}^\star - u_\ell^{\mm, \nn} \vvvert
			+
			C(\gamma, \lambda) \, \gamma \,
			\eta_{\ell+1}(u_{\ell+1}^{\mm, \nn})
			\\&
			\eqreff*{eq1:estimator_reduction}
			\le
			\bigl( \qsym + 4 \, \Cstab \, C(\gamma, \lambda) \, \gamma \bigr)
			\,
			\vvvert u_{\ell+1}^\star - u_\ell^{\mm, \nn} \vvvert
			+
			q(\theta) \, C(\gamma, \lambda) \, \gamma \,
			\eta_\ell(u_\ell^{\mm, \nn})
			\\&
			\le
			q_{\mathrm{ctr}} \,
			\bigl[ \vvvert u_{\ell+1}^\star - u_\ell^{\mm, \nn} \vvvert +
				\gamma \, \eta_\ell(u_\ell^{\mm, \nn}) \bigr].
		\end{aligned}
	\end{equation}
	The same arguments
	yield, for all \((\ell+1, \mmu, \nnu) \in \QQ^z\),
	\begin{equation}\label{eq:step3:5}
		\begin{aligned}
			\Zeta_{\ell+1}
			 & \eqreff{eq2:inexact_Zarantonello_contraction}
			\le
			\qsym \, \vvvert z_{\ell+1}^\star - z_{\ell+1}^{\mmu-1,
					\nnu} \vvvert
			+
			C(\gamma, \lambda) \, \gamma \,
			\zeta_{\ell+1}(z_{\ell+1}^{\mmu, \nnu})
			\\&
			\eqreff*{eq2:stability}
			\le
			\qsym \vvvert z_{\ell+1}^\star - z_\ell^{\mmu, \nnu} \vvvert
			+
			C(\gamma, \lambda) \gamma \,
			\zeta_{\ell+1}(z_{\ell+1}^{\mmu, \nnu})
			\\&
			\eqreff{eq1:estimator_reduction}\le
			C(\gamma, \lambda) \,
			\bigl[q_{\textup{ctr}} \, \vvvert z_{\ell+1}^\star
				-
				z_\ell^{\mmu, \nnu} \vvvert +
				\gamma \,
				\zeta_\ell(z_\ell^{\mmu, \nnu}) \bigr].
		\end{aligned}
	\end{equation}
	For
	\(
	0
	<
	q_{\mathrm{ctr}}
	<
	q_{\mathrm{ctr}}^\prime
	=
	C(\gamma, \lambda) \, q_{\mathrm{ctr}}
	< 1
	\),
	the product of~\eqref{eq:contraction-primal} and
	\eqref{eq:step3:5}
	reads
	\begin{equation}\label{eq:step3:6}
		\medmuskip = 4mu
		\begin{aligned}
			\Eta_{\ell+1} \, \Zeta_{\ell+1}
			 & \le
			C(\gamma, \lambda) \, q_{\mathrm{ctr}} \,
			\bigl[ \vvvert u_{\ell+1}^\star - u_\ell^{\mm, \nn} \vvvert +
				\gamma \, \eta_\ell(u_\ell^{\mm, \nn}) \bigr]
			\bigl[\vvvert z_{\ell+1}^\star -
				z_\ell^{\mmu, \nnu} \vvvert + \gamma \, \zeta_\ell(z_\ell^{\mmu,
					\nnu}) \bigr]
			\\
			 & =
			q_{\mathrm{ctr}}^\prime \bigl[ \vvvert u_{\ell+1}^\star -
				u_\ell^{\mm, \nn} \vvvert + \gamma \,
				\eta_\ell(u_\ell^{\mm, \nn}) \bigr] \bigl[
				\vvvert z_{\ell+1}^\star - z_\ell^{\mmu, \nnu} \vvvert + \gamma
				\,
				\zeta_\ell(z_\ell^{\mmu, \nnu}) \bigr].
		\end{aligned}
	\end{equation}
	If \eqref{eq2:estimator_reduction} is satisfied, we obtain
	the same estimate with reversed roles in the derivation.

	\textbf{Step~3 (quasi-monotonicity of
		\(\boldsymbol{\Eta_{\ell}}\) and
		\(\boldsymbol{\Zeta_{\ell}}\))}.
	The Céa
	estimate~\eqref{eq:cea}, nestedness of
	the discrete spaces,
	reliability~\eqref{axiom:reliability},
	quasi-monotonicity~\eqref{axiom:qm}, stability~\eqref{axiom:stability},
	and the definition~\eqref{eq:def:Eta_ell} prove, for all
	\(\ell \le
	\ell' \le
	\ell^{''} \le \elll\) with \((\ell, \mm, \nn) \in \QQ^u\) and \((\ell, \mmu,
	\nnu) \in \QQ^z\), that
	\begin{subequations}\label{eq:estimator_quasiError}
		\medmuskip = 1.1mu
		\begin{align}
			\vvvert u_{\ell''}^\star - u_{\ell'}^\star \vvvert
			 & \eqreff*{eq:cea}\lesssim
			\vvvert u^\star - u_{\ell'}^\star \vvvert
			\eqreff{axiom:reliability}
			\lesssim
			\eta_{\ell'}(u_{\ell'}^\star)
			\eqreff{axiom:qm}
			\lesssim
			\eta_{\ell}(u_{\ell}^\star)
			\eqreff{axiom:stability}
			\lesssim
			\eta_\ell(u_\ell^{\mm, \nn}) +
			\vvvert u_\ell^\star - u_\ell^{\mm, \nn} \vvvert
			\eqreff{eq:def:Eta_ell}
			\simeq
			\Eta_\ell, \label{eq:estimator_EtaEll_primal}
			\\
			\vvvert z_{\ell''}^\star - z_{\ell'}^\star \vvvert
			 & \eqreff*{eq:cea}\lesssim
			\vvvert z^\star - z_{\ell'}^\star \vvvert
			\eqreff{axiom:reliability}
			\lesssim
			\zeta_{\ell'}(z_{\ell'}^\star)
			\eqreff{axiom:qm}
			\lesssim
			\zeta_{\ell}(z_{\ell}^\star)
			\eqreff{axiom:stability}
			\lesssim
			\zeta_\ell(z_\ell^{\mmu, \nnu}) +
			\vvvert z_\ell^\star - z_\ell^{\mmu, \nnu} \vvvert
			\eqreff{eq:def:Eta_ell}
			\simeq
			\Zeta_\ell, \label{eq:estimator_ZetaEll_dual}
		\end{align}
	\end{subequations}
	where the hidden constants depend only on \(\gamma^{-1}\),
	\(\Ccea\), \(\Cstab\), \(\Crel\), and \(\Cmon\).
	
		Similarly to~\eqref{eq:contraction-primal}, the inexact 
		Zarantonello contraction~\eqref{eq2:inexact_Zarantonello_contraction}, 
		stability~\eqref{axiom:stability}, and the stability 
		estimate~\eqref{eq3:stability} yield for \(\ell < \ell' < \elll\) and \(\lambda = \lamsym \, \lamalg\),
		\begin{equation}\label{eq:contraction-primal-error}
			\begin{aligned}
				\vvvert u_{\ell'}^\star - u_{\ell'}^{\mm, \nn} \vvvert
				\, &\eqreff*{eq2:inexact_Zarantonello_contraction}
				\le
				\,
				\qsym \, \vvvert u_{\ell'}^\star - u_{\ell'}^{\mm-1, \nn} \vvvert
				+
				2 \frac{\qalg}{1-\qalg} \, \lambda \, \eta_{\ell'}(u_{\ell'}^{\mm, \nn})
				\\
				& \eqreff*{eq:inexact_Zarantonello_contraction}\le
				\,
				\qsym \, \qsymm^{\mm[\ell'] - 1} \, \vvvert u_{\ell'}^\star - u_{\ell' - 1}^{\mm, \nn} \vvvert 
				+ 2 \frac{\qalg}{1-\qalg} \, \lambda \, \eta_{\ell'}(u_{\ell'}^{\mm, \nn})
				\\
				 &\stackrel{\mathclap{\eqref{axiom:stability}, \, \eqref{eq5:combined-contraction}}}{\le} \quad \,
				\,
				\qsym \, \vvvert u_{\ell'}^\star - u_{\ell' - 1}^{\mm, \nn} \vvvert
				+
				2 \frac{\qalg}{1-\qalg} \, \lambda \,  \eta_{\ell' - 1}(u_{\ell' - 1}^{\mm, \nn})
				\\
				&\quad + 2 \Cstab \frac{\qalg}{1-\qalg} \, \lambda \,  \vvvert u_{\ell'}^{\mm, \nn} - u_{\ell'-1}^{\mm, \nn} \vvvert
				\\
				 & \eqreff*{eq3:stability}
				\le
				\,
				\Bigl( \qsym + 8 \Cstab \, \frac{\qalg}{1-\qalg} \, \lambda \Bigr) \, \vvvert u_{\ell'}^\star - u_{\ell' - 1}^{\mm, \nn} \vvvert
				+
				2 \frac{\qalg}{1-\qalg} \, \lambda \,  \eta_{\ell'-1}(u_{\ell'-1}^{\mm, \nn})
				\\
				 & \eqreff*{axiom:stability}
				\le
				\,
				\Bigl( \qsym + 10 \Cstab \, \frac{\qalg}{1-\qalg} \, \lambda \Bigr) \, \vvvert u_{\ell' - 1}^\star - u_{\ell' - 1}^{\mm, \nn} \vvvert
				\\
				 &\quad+
				\Bigl( \qsym + 8 \Cstab \, \frac{\qalg}{1-\qalg} \, \lambda \Bigr)
				\vvvert u_{\ell'}^\star - u_{\ell' - 1}^\star \vvvert
				+
				2 \frac{\qalg}{1-\qalg} \, \lambda \,  \eta_{\ell' - 1}(u_{\ell' - 1}^\star)
				\\
				&\eqreff*{eq:estimator_EtaEll_primal}{\le}\,
				\Bigl( \qsym + 10 \Cstab \, \frac{\qalg}{1-\qalg} \, \lambda \Bigr) \, \vvvert u_{\ell' - 1}^\star - u_{\ell' - 1}^{\mm, \nn} \vvvert
				\\
				&\quad +
				\Bigl( \qsym + 2 \, \frac{\qalg}{1-\qalg} \, \lambda \, \Cmon  \bigl[ 
					4 \, \Cstab \, ( 1 + \Ccea) \, \Crel  + 1
				\bigr]\Bigr) \,
				\eta_{\ell}(u_{\ell}^\star).
			\end{aligned}
		\end{equation}
		The choice of \(\lambda \le \lambda^\star\) with \(\lambda^\star\) from~\eqref{eq:choice_lambda} ensures
		\begin{equation}
			0 < q \coloneqq  \qsym + 10 \, \Cstab \, \frac{\qalg}{1-\qalg} \, \lambda < 1.
		\end{equation}
		With \(C 
		\coloneqq \qsym + 2 \, \frac{\qalg}{1-\qalg} \, \lambda \, \Cmon  \bigl[ 
					4 \, \Cstab \, ( 1 + \Ccea) \, \Crel + 1 \bigr]\), a successive application of~\eqref{eq:contraction-primal-error} and the geometric series shows
		\begin{equation}\label{eq:quasimon_primal_error}
			\begin{aligned}
				\vvvert u_{\ell+M}^\star - u_{\ell+M}^{\mm, \nn} \vvvert
				 & \eqreff*{eq:contraction-primal-error}\le
				q \, \vvvert u_{\ell+M-1}^\star - u_{\ell+M-1}^{\mm, \nn} \vvvert 
				+ C \, \eta_{\ell}(u_{\ell}^\star)
				\\
				 & \eqreff*{eq:contraction-primal-error}\le
				q^{M} \, \vvvert u_{\ell}^\star - u_{\ell}^{\mm, \nn} \vvvert  + 
				\Bigl(C \, \sum_{j=0}^{M-1} q^{j} \Bigr) \, \eta_{\ell}(u_{\ell}^{\star})
				\\
				 &\lesssim
				 \vvvert u_{\ell}^\star - u_{\ell}^{\mm, \nn} \vvvert
				+ \eta_{\ell}(u_{\ell}^{\star})
				 \eqreff{eq:estimator_EtaEll_primal}\lesssim \
				\Eta_{\ell}
				\,
				\text{ for all \(\ell, M \in \N_0\) with \(\ell+M < \elll\)}.
			\end{aligned}
		\end{equation}
		Hence, we have quasi-monotonicity of the quasi-error
		\begin{subequations}\label{eq:quasimon_error}
		\begin{equation}\label{eq:quasimon_primal}
			\begin{aligned}
				\Eta_{\ell+M}
				 & \simeq
				\vvvert u_{\ell+M}^\star - u_{\ell+M}^{\mm, \nn} \vvvert
				+
				\eta_{\ell+M}(u_{\ell+M}^{\mm, \nn})
				\eqreff{axiom:stability}
				\simeq
				\vvvert u_{\ell+M}^\star - u_{\ell+M}^{\mm, \nn} \vvvert
				+
				\eta_{\ell+M}(u_{\ell+M}^{\star})
				\\
				 & \eqreff*{eq:estimator_EtaEll_primal}
				\lesssim \
				\vvvert u_{\ell+M}^\star - u_{\ell+M}^{\mm, \nn} \vvvert + \Eta_{\ell}
				\eqreff{eq:quasimon_primal_error}
				\lesssim \Eta_{\ell}
				\text{ for all \(\ell, M \in \N_0\) with \(\ell+M < \elll\)}.
			\end{aligned}
		\end{equation}
	The same argument proves
	\begin{equation}\label{eq:quasimon_dual}
		\Zeta_{\ell+M} \lesssim \Zeta_{\ell}
		\quad
		\text{for all \(\ell, M \in \N_0\) with \(\ell+M < \elll\)}.
	\end{equation}
	\end{subequations}

	\textbf{Step~4 (contraction of \(\boldsymbol{\Eta_{\ell} \,
			\Zeta_{\ell}}\)
		up to tail-summable remainder).}
	Define
	\begin{align*}
		R_\ell & \coloneqq \vvvert u_{\ell+1}^\star - u_\ell^{\star} \vvvert \,
		\bigl[\vvvert z_\ell^\star - z_\ell^{\mmu, \nnu} \vvvert +
			\vvvert z_{\ell+1}^\star - z_\ell^{\star} \vvvert +
			\gamma \, \zeta_\ell(z_\ell^{\mmu, \nnu})\bigr]
		\\
		       & \quad+ \vvvert z_{\ell+1}^\star - z_\ell^{\star} \vvvert
		\bigl[\vvvert u_\ell^\star - u_\ell^{\mm, \nn} \vvvert +
			\vvvert u_{\ell+1}^\star - u_\ell^{\star} \vvvert + \gamma \,
			\eta_\ell(u_\ell^{\mm, \nn})\bigr].
	\end{align*}
	The contraction~\eqref{eq:step3:6} proves
	the
	quasi-contraction~\eqref{eq:contraction_remainder} via
	\begin{align*}
		\begin{split}
			\Eta_{\ell+1} \, \Zeta_{\ell+1}
			 & \eqreff*{eq:step3:6}
			\le
			q_{\mathrm{ctr}}^\prime
			\bigl[
				\vvvert u_{\ell+1}^\star - u_\ell^{\mm, \nn} \vvvert + \gamma
				\, \eta_\ell(u_\ell^{\mm, \nn})
				\bigr]
			\bigl[
				\vvvert z_{\ell+1}^\star - z_\ell^{\mmu, \nnu} \vvvert + \gamma
				\, \zeta_\ell(z_\ell^{\mmu, \nnu})
				\bigr]
			\\
			 & \le
			q_{\mathrm{ctr}}^\prime \,
			\bigl[
				\vvvert u_{\ell}^\star - u_\ell^{\mm, \nn} \vvvert
				+ \vvvert u_{\ell+1}^\star - u_{\ell}^\star \vvvert
				+ \gamma \, \eta_\ell(u_\ell^{\mm, \nn})
				\bigr]
			\\
			 & \quad \qquad \times
			\bigl[
				\vvvert z_{\ell}^\star - z_\ell^{\mmu, \nnu} \vvvert
				+ \vvvert z_{\ell+1}^\star - z_{\ell}^\star \vvvert
				+\gamma \, \zeta_\ell(z_\ell^{\mmu, \nnu})
				\bigr]
			\\
			 & \le
			q_{\mathrm{ctr}}^\prime \, \Eta_\ell \, \Zeta_{\ell}
			+
			q_{\mathrm{ctr}}^\prime \, R_\ell.
		\end{split}
	\end{align*}
	The
	remainder term \(R_\ell\) can be estimated
	by~\eqref{eq:estimator_quasiError} and the Young inequality
	to show
	\begin{equation}\label{eq:remainder}
		R_\ell^2
		\eqreff{eq:estimator_quasiError}\lesssim
		\bigl(\vvvert u_{\ell+1}^\star - u_\ell^{\star} \vvvert \, \Zeta_{\ell} +
		\vvvert z_{\ell+1}^\star - z_\ell^{\star} \vvvert \, \Eta_{\ell}
		\bigr)^{2}
		\lesssim
		\vvvert u_{\ell+1}^\star - u_\ell^{\star} \vvvert^2 \, \Zeta_{\ell}^2 +
		\vvvert z_{\ell+1}^\star - z_\ell^{\star} \vvvert^2 \, \Eta_{\ell}^2.
	\end{equation}
	Thus, the
	quasi-monotonicity~\eqref{eq:quasimon_error}
	verifies
	\begin{equation*}
		R_{\ell+M} \lesssim \Eta_{\ell+M} \, \Zeta_{\ell+M}
		\eqreff{eq:quasimon_error}\lesssim
		\Eta_{\ell} \, \Zeta_\ell
		\quad
		\text{for all \(\ell, M \in \N\) with \(\ell+M < \elll\).}
	\end{equation*}
	Quasi-orthogonality~\eqref{axiom:orthogonality},
	reliability~\eqref{axiom:reliability},
	and the estimates~\eqref{eq:estimator_quasiError} imply,
	for
	all
	\(\ell, M \in \N_0\) with \(\ell + M < \elll\),
	\begin{equation}\label{eq1:summability_remainder}
		\medmuskip = 1mu
		\begin{aligned}
			\sum_{\ell^\prime=\ell}^{\ell+M}
			\vvvert u_{\ell'+1}^\star - u_{\ell'}^\star \vvvert^2
			 & \eqreff*{axiom:orthogonality}
			\lesssim
			(M+1)^{1-\delta} \,
			\vvvert u^\star - u_{\ell}^\star \vvvert^2
			\eqreff{axiom:reliability}
			\lesssim
			(M+1)^{1-\delta} \, \eta_{\ell}(u_{\ell}^\star)^2
			\eqreff{eq:estimator_EtaEll_primal}
			\lesssim
			(M+1)^{1-\delta} \, \Eta_{\ell}^2,
			\\
			\sum_{\ell^\prime=\ell}^{\ell+M}
			\vvvert z_{\ell'+1}^\star - z_{\ell'}^\star \vvvert^2
			 & \eqreff*{axiom:orthogonality}
			\lesssim
			(M+1)^{1-\delta} \,
			\vvvert z^\star - z_{\ell}^\star \vvvert^2
			\eqreff{axiom:reliability}
			\lesssim
			(M+1)^{1-\delta} \, \zeta_{\ell}(z_{\ell}^\star)^2
			\eqreff{eq:estimator_ZetaEll_dual}
			\lesssim
			(M+1)^{1-\delta} \, \Zeta_{\ell}^2.
		\end{aligned}
	\end{equation}
	Using~\eqref{eq:remainder}, the
	quasi-monotonicity~\eqref{eq:quasimon_error},
	and~\eqref{eq1:summability_remainder}, we conclude the proof of
	\eqref{eq:summability_remainder}, for all \(\ell,
	M \in \N_0\) with \(\ell+M < \elll\),
	\begin{align*}
		\sum_{\ell^\prime=\ell}^{\ell+M}
		R_{\ell^\prime}^2
		 & \eqreff*{eq:remainder}
		\lesssim
		\
		\sum_{\ell^\prime=\ell}^{\ell+M}
		\vvvert u_{\ell'+1}^\star - u_{\ell'}^\star \vvvert^2 \,
		\Zeta_{\ell^\prime}^2
		+
		\sum_{\ell^\prime=\ell}^{\ell+M}
		\vvvert z_{\ell'+1}^\star - z_{\ell'}^\star \vvvert^2 \,
		\Eta_{\ell^\prime}^2
		\\
		 & \eqreff*{eq:quasimon_error}
		\lesssim
		\
		\Zeta_{\ell}^2
		\sum_{\ell^\prime=\ell}^{\ell+M}
		\vvvert u_{\ell'+1}^\star - u_{\ell'}^\star \vvvert^2
		+
		\Eta_{\ell}^2
		\sum_{\ell^\prime=\ell}^{\ell+M}
		\vvvert z_{\ell'+1}^\star - z_{\ell'}^\star \vvvert^2
		\eqreff{eq1:summability_remainder}
		\lesssim
		(M+1)^{1-\delta} \, \Eta_{\ell}^{2} \, \Zeta_{\ell}^{2}.	\qedhere
	\end{align*}
\end{proof}

The tail-summability in \(\ell\) provides the basis for the proof of tail-summability on
the mesh level \(\ell\) together with the Zarantonello symmetrization index
\(k\) for the final iterates of the algebraic solver.
The main ingredients in the proof of tail-summability in
\((\ell,k)\) are
Lemma~\ref{lem:contraction_remainder} and the following
quasi-contraction in the symmetrization index \(k\).
\begin{lemma}[quasi-contraction of inexact Zarantonello symmetrization]
	\label{lem:quasi_contraction_Eta_ell^k}
	There holds
	\begin{align}\label{eq:quasi_contraction_Eta_ell^k}
		\Eta_\ell^{k', \jj} \, \Zeta_\ell^{k', \jj}
		 & \lesssim
		\qsymm^{k' - k} \,
		\Eta_\ell^{k, \jj} \,
		\Zeta_{\ell}^{k, \jj}
		\quad
		 &          & \text{for all \((\ell,k', \jj) \in \QQ\) with
			\(0 \le k \le k' \le \kk[\ell]\)},
		\\
		\label{eq:Eta_ell^0}
		\Eta_\ell^{0, \jj} \, \Zeta_\ell^{0, \jj}
		 & \lesssim
		\Eta_{\ell-1} \, \Zeta_{\ell-1}
		\quad
		 &          & \text{for all \((\ell, 0, 0) \in \QQ\) with \(\ell \ge 1\)}.
	\end{align}
\end{lemma}
\begin{proof}
	First, we note that the
	\textsl{a~posteriori} error
	control~\eqref{eq:aposteriori_algebra}
	and the
	stopping criteria of the
	algebraic solver~\eqref{eq:n_stopping_criterion} and of the
	symmetrization~\eqref{eq:m_stopping_criterion} lead, for \((\ell, \mm, \nn)
	\in \QQ^{u}\), to
	\begin{equation*}
		\vvvert u_\ell^{\mm, \star} - u_\ell^{\mm, \nn} \vvvert
		\eqreff{eq:aposteriori_algebra}
		\lesssim
		\vvvert u_\ell^{\mm, \nn} - u_\ell^{\mm, \nn-1} \vvvert
		\eqreff{eq:n_stopping_criterion}
		\lesssim
		\eta_{\ell}(u_\ell^{\mm, \nn})
		+ \vvvert u_\ell^{\mm, \nn} - u_\ell^{\mm, 0} \vvvert
		\eqreff{eq:m_stopping_criterion}
		\lesssim
		\eta_\ell(u_\ell^{\mm, \nn})
		\lesssim
		\Eta_\ell.
	\end{equation*}
	Since the two notions of quasi-errors
	\(\Eta_\ell\) and \(\Eta_{\ell}^{\kk,\jj}\) only differ
	by the middle term \(\vvvert u_\ell^{\mm, \star} - u_\ell^{\mm, \nn} \vvvert\)
	and the fixed constant factor \(0 < \gamma < 1\), this and
	the analogous estimate for the dual variable show
	\begin{equation}\label{eq:step7:x}
		\Eta_\ell
		\le
		\Eta_\ell^{\kk, \jj}
		\lesssim \Eta_\ell
		\quad
		\text{and}
		\quad
		\Zeta_{\ell}
		\le
		\Zeta_\ell^{\kk, \jj}
		\lesssim
		\Zeta_{\ell}
		\quad
		\text{for all \((\ell, \kk, \jj) \in \QQ\)}.
	\end{equation}
	For $0 \le k < k' < \mm[\ell] < \kk[\ell]$
	(i.e., the primal iteration stops earlier than the dual iteration),
	the validity of the stopping criterion~\eqref{eq:n_stopping_criterion}
	for the algebraic solver and
	the failure of criterion~\eqref{eq:m_stopping_criterion} for
	the inexact Zarantonello symmetrization
	prove that
	\begin{equation}\label{eq1a:step7}
		\begin{aligned}
			\Eta_\ell^{k', \nn}
			 & \eqreff*{eq:aposteriori_algebra}\lesssim
			\vvvert u_\ell^\star - u_\ell^{k', \nn} \vvvert +
			\vvvert u_\ell^{k', \nn} - u_\ell^{k', \nn-1} \vvvert +
			\eta_\ell(u_\ell^{k', \nn})
			\\
			 & \eqreff*{eq:n_stopping_criterion}\lesssim
			\vvvert u_\ell^\star - u_\ell^{k'-1, \nn} \vvvert +
			\vvvert u_\ell^{k', \nn} - u_\ell^{k'-1, \nn} \vvvert +
			\eta_\ell(u_\ell^{k', \nn})
			\\
			 & \eqreff*{eq:m_stopping_criterion}\lesssim
			\vvvert u_\ell^\star - u_\ell^{k',\nn} \vvvert +
			\vvvert u_\ell^{k',\nn} - u_\ell^{k'-1,\nn} \vvvert
			\\
			 & \eqreff*{eq:aposteriori_inexactZarantonello}\le
			\vvvert u_\ell^\star - u_\ell^{k'-1,\nn} \vvvert
			\eqreff*{eq:inexact_Zarantonello_contraction}
			\lesssim
			\qsymm^{k'-k} \, \vvvert u_\ell^\star -
			u_\ell^{k,\nn} \vvvert
			\lesssim
			\qsymm^{k'-k} \, \Eta_\ell^{k, \nn}.
		\end{aligned}
	\end{equation}
	Moreover, for $0 \le k < k' = \mm[\ell]$, stability~\eqref{axiom:stability}
	and the estimate~\eqref{eq4:stability} verify
	\begin{align*}\label{eq1:step7}
		\Eta_\ell^{\mm, \nn}
		\, & \eqreff*{eq:step7:x}
		\simeq \,
		\vvvert u_\ell^\star - u_\ell^{\mm, \nn} \vvvert + \eta_\ell(u_\ell^{\mm,
			\nn})
		\eqreff*{axiom:stability}
		\lesssim
		\vvvert u_\ell^\star - u_\ell^{\mm,\nn} \vvvert
		+
		\vvvert u_\ell^{\mm, \nn} - u_\ell^{\mm-1, \nn} \vvvert
		+
		\eta_\ell(u_\ell^{\mm-1, \nn})
		\\&
		\lesssim
		\Eta_\ell^{\mm-1, \nn} + \vvvert u_\ell^{\mm, \nn} - u_\ell^{\mm-1,
			\nn} \vvvert
		\eqreff{eq4:stability}\lesssim \, \Eta_\ell^{\mm-1, \nn}
		\eqreff{eq1a:step7}\lesssim \qsymm^{\mm[\ell] - 1 - k} \,
		\Eta_{\ell}^{k, \nn}
		\simeq
		\qsymm^{\mm[\ell]-k} \, \Eta_\ell^{k, \nn}.
	\end{align*}
	For \(0 \le k \le \mm[\ell] < k' \le \kk[\ell]\),
	it follows
	\(
	\Eta_\ell^{k', \nn} = \Eta_{\ell}^{\mm, \nn}
	\lesssim
	\qsymm^{\mm[\ell] - k} \, \Eta_\ell^{k, \nn}.
	\)
	Finally, for \(\mm[\ell] \le k < k' \le \kk[\ell]\), we have
	\(\Eta_{\ell}^{k', \nn} =
	\Eta_{\ell}^{\mm[\ell], \nn} = \Eta_{\ell}^{k, \nn}\).
	Notice that the same argumentation holds for
	the dual quasi-error \(\Zeta_{\ell}^{k, \nnu}\) in
	the remaining cases with \(\mmu[\ell] < \kk[\ell]\)
	(i.e., the dual iteration stops earlier than the primal iteration).

	Since \(\kk[\ell] = \mm[\ell]\) or \(\kk[\ell] = \mu[\ell]\) by
	definition, we obtain, for all \((\ell, k', \jj) \in \QQ\) with
	\(0 \le k \le k' \le \kk[\ell]\),
	\begin{equation*}
		\Eta_{\ell}^{k', \jj} \lesssim \qsymm^{k' - k} \, \Eta_{\ell}^{k, \jj}
		\quad \text{if \(\kk[\ell] = \mm[\ell]\)}
		\quad
		\text{or}
		\quad
		\Zeta_{\ell}^{k', \jj} \lesssim \qsymm^{k' - k} \, \Zeta_{\ell}^{k, \jj}
		\quad
		\text{if \(\kk[\ell] = \mmu[\ell]\).}
	\end{equation*}
	Furthermore, there holds \(\Eta_{\ell}^{k', \jj} \lesssim
	\Eta_{\ell}^{k, \jj}\) and
	\(\Zeta_{\ell}^{k', \jj} \lesssim \Zeta_{\ell}^{k, \jj}\) in any
	case. This yields~\eqref{eq:quasi_contraction_Eta_ell^k} via
	\begin{equation*}
		\Eta_\ell^{k', \jj} \, \Zeta_\ell^{k', \jj}
		\lesssim \qsymm^{k' - k} \, \Eta_\ell^{k, \jj} \,
		\Zeta_{\ell}^{k, \jj}
		\quad \text{for all } (\ell, k', \jj) \in \QQ \text{ with
		} 0 \le k \le k' \le \kk[\ell],
	\end{equation*}
	where the hidden constant depends only on $\Cstab$, $\lamsym$, and $\qsymm$.

	Nested iteration $u_{\ell-1}^{\mm, \nn} =
		u_\ell^{0,\nn}$ and $z_{\ell-1}^{\mmu, \nnu} = z_\ell^{0,\nnu}$
	and the estimates~\eqref{eq:estimator_quasiError} yield, for
	all \((\ell,0,0)
	\in \QQ\) with \(\ell > 0\),
	\begin{align*}
		\Eta_\ell^{0, \jj}
		 & \eqreff{eq:step7:x}\simeq
		\vvvert u_\ell^\star - u_{\ell-1}^{\mm, \nn} \vvvert +
		\eta_\ell(u_{\ell-1}^{\mm, \nn})
		\le \vvvert u_\ell^\star - u_{\ell-1}^\star \vvvert + \Eta_{\ell-1}^{\kk,
			\jj}
		\eqreff{eq:estimator_quasiError}
		\lesssim
		\Eta_{\ell-1} + \Eta_{\ell-1}^{\kk, \jj}
		\eqreff{eq:step7:x}\simeq
		\Eta_{\ell-1},
		\\
		\Zeta_\ell^{0, \jj}
		 & \eqreff{eq:step7:x}\simeq
		\vvvert z_\ell^\star - z_{\ell-1}^{\mmu, \nnu} \vvvert +
		\zeta_\ell(z_{\ell-1}^{\mmu, \nnu})
		\le \vvvert z_\ell^\star - z_{\ell-1}^\star \vvvert +
		\Zeta_{\ell-1}^{\kk,
			\jj}
		\eqreff{eq:estimator_quasiError}
		\lesssim
		\Zeta_{\ell-1} + \Zeta_{\ell-1}^{\kk, \jj}
		\eqreff{eq:step7:x}\simeq
		\Zeta_{\ell-1}.
	\end{align*}
	A multiplication of the two previous estimates
	proves~\eqref{eq:Eta_ell^0}.
\end{proof}

Finally, the quasi-contraction in \((\ell, k)\) from
Lemma~\ref{lem:quasi_contraction_Eta_ell^k} together with a
quasi-contraction in
the algebraic solver index \(j\) leads to
tail-summability in \((\ell, k, j)\).
\begin{lemma}[quasi-contraction and stability by algebraic solver]
	\label{lem:quasi_contraction_Eta_ell^kj}
	There holds
	\begin{equation}\label{eq:quasi_contraction_Eta_ell^kj}
		\Eta_\ell^{k,j'} \, \Zeta_\ell^{k,j'}
		\lesssim
		\qalg^{j' - j} \,
		\Eta_\ell^{k,j} \, \Zeta_\ell^{k,j}
		\quad
		\text{for all \((\ell,k,j') \in \QQ\) with
			\(0 \le j \le j' \le \jj[\ell,k]\)}
	\end{equation}
	and, with the abbreviation \((m-1)_{+} \coloneqq \max \{m-1, 0\}\),
	\begin{equation}\label{eq1:step9}
		\Eta_\ell^{m,0} \le 3 \, \Eta_\ell^{(m-1)_{+},
			\nn}
		\,
		\text{ and }
		\,
		\Zeta_\ell^{\mu,0} \le 3 \, \Zeta_\ell^{(\mu-1)_{+},
			\nnu}
		\,
		\text{ for all \((\ell, m, 0) \in \QQ^{u}, (\ell, \mu, 0)
			\in
			\QQ^z\)}.
	\end{equation}
\end{lemma}

\begin{proof}
	We recall that $u_\ell^{0,0} = u_\ell^{0, \nn} = u_\ell^{0,\star}$
	by definition
	and, hence, $\Eta_\ell^{0,0} = \Eta_\ell^{0, \nn} = \Eta_\ell^{0,
			\jj}$. Nested
	iteration
	$u_\ell^{m,0} = u_\ell^{m-1,\nn}$ implies that
	\begin{align*}
		\vvvert u_\ell^{m,\star} - u_\ell^{m, 0} \vvvert
		\eqreff{eq:aposteriori_Zarantonello}\le
		(\qsym + 1) \, \vvvert u_\ell^\star - u_\ell^{m-1,\nn} \vvvert
		\le 2 \, \Eta_\ell^{m-1, \jj}
		\quad
		\text{for all \((\ell, m, 0) \in \QQ^u\)}.
	\end{align*}
	Therewith, we derive~\eqref{eq1:step9}.

	The combination of \textsl{a~posteriori} error
	control~\eqref{eq:aposteriori_Zarantonello} for the exact
	Zarantonello iteration, for
	the algebraic solver~\eqref{eq:aposteriori_algebra}, and the
	failure of the stopping criterion~\eqref{eq:n_stopping_criterion} in
	Algorithm~\ref{algorithm:afem}\href{alg:primal_ii}{(\textrm{I.b.ii})} for the algebraic
	solver proves, for $0 \le j < j' < \nn[\ell, m] < \jj[\ell, m]$,
	\begin{equation}\label{eq2c:step9}
		\begin{aligned}
			\Eta_\ell^{m, j'}
			 & \le \
			\vvvert u_\ell^\star - u_\ell^{m, \star} \vvvert
			+
			2 \, \vvvert u_\ell^{m, \star} - u_\ell^{m, j'} \vvvert
			+
			\eta_\ell(u_\ell^{m, j'})
			\\&
			\eqreff*{eq:aposteriori_Zarantonello}
			\le
			\frac{\qsym}{1-\qsym} \, \vvvert u_\ell^{m, j'} - u_\ell^{m - 1,
				\jj} \vvvert
			+
			\Big(2 + \frac{\qsym}{1-\qsym} \Big) \, \vvvert u_\ell^{m, \star} -
			u_\ell^{m, j'} \vvvert
			+
			\eta_\ell(u_\ell^{m,j'})
			\\&
			\eqreff*{eq:aposteriori_algebra}\lesssim
			\vvvert u_\ell^{m, j'} - u_\ell^{m-1, \jj} \vvvert + \vvvert u_\ell^{m, j'}
			-
			u_\ell^{m, j'-1} \vvvert
			+
			\eta_\ell(u_\ell^{m, j'})
			\eqreff{eq:n_stopping_criterion}
			\lesssim \
			\vvvert u_\ell^{m, j'} - u_\ell^{{m, j'-1}} \vvvert
			\\&
			\eqreff*{eq:aposteriori_algebra}
			\lesssim
			\vvvert u_\ell^{m, \star} - u_\ell^{m, j'-1} \vvvert
			\eqreff{eq:algebra_contraction}
			\le
			\qalg^{(j'-1)-j} \, \vvvert u_\ell^{m, \star} - u_\ell^{m, j} \vvvert
			\lesssim
			\qalg^{j'-j} \, \Eta_\ell^{m, j}.
		\end{aligned}
	\end{equation}
	For $0 \le j < \nn[\ell,m] \le j' \le \jj[\ell, m]$,
	stability~\eqref{axiom:stability} and contraction of the algebraic
	solver~\eqref{eq:algebra_contraction} verify that
	\begin{align*}\label{eq2d:step9}
		\begin{split}
			\Eta_\ell^{m, j'} = \Eta_\ell^{m, \nn}
			 & \eqreff*{eq:algebra_contraction}\le
			\vvvert u_\ell^\star - u_\ell^{m, \nn-1} \vvvert + \vvvert u_\ell^{m,
				\nn} - u_\ell^{m, \nn-1} \vvvert + \qalg \, \vvvert u_\ell^{m, \star} -
			u_\ell^{m, \nn-1} \vvvert
			+ \eta_\ell(u_\ell^{m, \nn})
			\\
			 & \eqreff*{axiom:stability}
			\le
			\,
			\Eta_\ell^{m, \nn-1} + (2 + \Cstab) \, \vvvert u_\ell^{m, \nn} -
			u_\ell^{m, \nn-1} \vvvert
			\\&
			\eqreff*{eq:aposteriori_algebra}
			\lesssim \
			\Eta_\ell^{m, \nn-1} + \vvvert u_\ell^{m, \star} - u_\ell^{m,
				\nn-1} \vvvert
			\lesssim \
			\Eta_\ell^{m, \nn-1}
			\eqreff{eq2c:step9}
			\lesssim
			\qalg^{\nn[\ell]-j} \, \Eta_\ell^{m, j}.
		\end{split}
	\end{align*}
	For \(\nn[\ell, m] \le j < j' \le \jj[\ell, m]\), it holds that
	\(\Eta_\ell^{m, j} = \Eta_\ell^{m, \nn}
	= \Eta_\ell^{m, j'}\).
	Since \(\jj[\ell, k] = \nn[\ell, k]\) or \(\jj[\ell,
		k] = \nnu[\ell, k]\),
	we have, for all \((\ell, k, j') \in \QQ\) with \(0 \le j \le j' \le \jj[\ell,
		k]\),
	\begin{align*}
		\Eta_{\ell}^{k, j} \lesssim \qalg^{j - j'} \, \Eta_{\ell}^{k, j'}
		\quad \text{if \(\jj[\ell, k] = \nn[\ell, k]\)}
		\quad
		\text{or}
		\quad
		\Zeta_{\ell}^{k, j} \lesssim \qalg^{j - j'} \, \Zeta_{\ell}^{k, j'}
		\quad \text{if \(\jj[\ell, k] = \nnu[\ell, k]\)}.
	\end{align*}
	Furthermore, we have \(\Eta_{\ell}^{k, j} \lesssim
	\Eta_{\ell}^{k, j'}\) and
	\(\Zeta_{\ell}^{k, j} \lesssim \Zeta_{\ell}^{k, j'}\) in any
	case.
	Hence, we obtain
	\begin{equation*}\label{eq3:step10}
		\Eta_\ell^{k,j} \, \Zeta_\ell^{k,j}
		\lesssim
		\qalg^{j - j'} \,
		\Eta_\ell^{k,j'} \, \Zeta_\ell^{k,j'}
		\quad
		\text{
			for all \((\ell,k,j) \in \QQ\) with \(0 \le j' \le j
			\le
			\jj[\ell,k]\)},
	\end{equation*}
	where the hidden constant depends only on $\qsym$, $\lamsym$, $\qalg$,
	$\lamalg$, and $\Cstab$.
\end{proof}

Ultimately, synthesizing the preceding lemmas
yields tail-summability of the quasi-error product
and thus leads to the following proof
of
Theorem~\ref{lem:full_linear_convergence}.

\begin{proof}[\textbf{Proof of Theorem~\ref{lem:full_linear_convergence}}]
	The proof consists of four steps.

		{\bf Step~1 (tail-summability in mesh level \(\boldsymbol{\ell}\)).}
	We apply the tail-summability criterion from~\cite[Lemma~5]{fps2023} to
	the
	sequences \(a_\ell \coloneqq \Eta_{\ell} \, \Zeta_{\ell}\) and \(b_\ell
	\coloneqq q_{\mathrm{ctr}}^\prime \, R_\ell\). Therein, it is shown that
	\(R\)-linear convergence is equivalent to
	tail-summability and that, for tail-summability,
	it is sufficient to guarantee
	\begin{equation}\label{eq:summability_criterion}
		a_{\ell+1} \le q a_\ell + b_\ell,
		\quad
		b_{\ell+M} \le C_1 \, a_\ell,
		\,\, \text{and} \,\,
		\sum_{\ell' = \ell}^{\ell+M} b_\ell^2 \le C_2 \,
		(M+1)^{1-\delta} \, a_\ell^2
		\,\,\, \text{for all } \ell, M \in \N_0.
	\end{equation}
	Indeed,
	contraction up to a remainder
	from~\eqref{eq:contraction_remainder}, the estimate of the
	remainder from \eqref{eq:summability_remainder}, and
	the quasi-monotonicity of
	\(\Eta_{\ell}\) and \(\Zeta_{\ell}\) from
	\eqref{eq:quasimon_error}
	validate the assumptions of
	the tail-summability criterion~\eqref{eq:summability_criterion} and
	lead to tail-summability
	\begin{equation}\label{eq:summability_ell}
		\sum_{\ell' = \ell+1}^{\elll-1} \Eta_{\ell'}  \, \Zeta_{\ell'}
		\lesssim
		\Eta_{\ell} \, \Zeta_{\ell}
		\quad \text{for all }
		(\ell, \kk,\jj) \in \QQ.
	\end{equation}

	\textbf{Step~2 (tail-summability in \(\boldsymbol{(\ell, k)}\)).}
	For $(\ell,k,\jj) \in \QQ$, the estimates
	\eqref{eq:quasi_contraction_Eta_ell^k}--\eqref{eq:Eta_ell^0} and the
	geometric series prove
	tail-summability
	\begin{equation}\label{eq:summability_k}
		\begin{aligned}
			\sum_{
				\substack{
					(\ell',k',\jj) \in \QQ
			\\ |\ell',k',\jj|
					>
					|\ell,k,\jj|
				}
			}
			\Eta_\ell^{k', \jj} \,  \Zeta_\ell^{k', \jj}
			 & =
			\sum_{k' = k+1}^{\kk[\ell]} \Eta_{\ell}^{k', \jj} \,
			\Zeta_{\ell}^{k',
				\jj}
			+
			\sum_{\ell' = \ell+1}^{\elll} \sum_{k'=0}^{\kk[\ell']}
			\Eta_{\ell'}^{k', \jj}\,
			\Zeta_{\ell'}^{k', \jj}
			\\
			 & \eqreff{eq:quasi_contraction_Eta_ell^k}
			\lesssim
			\Eta_\ell^{k, \jj} \, \Zeta_\ell^{k, \jj}
			+ \sum_{\ell' = \ell+1}^{\elll} \Eta_{\ell'}^{0, \jj} \,
			\Zeta_{\ell'}^{0, \jj}
			\eqreff{eq:Eta_ell^0}
			\lesssim
			\Eta_{\ell}^{k, \jj} \, \Zeta_{\ell}^{k, \jj}
			+ \sum_{\ell' = \ell}^{\elll-1} \Eta_{\ell'} \,  \Zeta_{\ell'}
			\\
			 & \eqreff*{eq:summability_ell}
			\lesssim
			\Eta_{\ell}^{k, \jj} \, \Zeta_{\ell}^{k, \jj}
			+ \Eta_{\ell} \, \Zeta_{\ell}
			\eqreff{eq:step7:x}
			\lesssim
			\Eta_{\ell}^{k, \jj} \, \Zeta_{\ell}^{k, \jj}.
		\end{aligned}
	\end{equation}

	\textbf{Step~3 (tail-summability in \(\boldsymbol{(\ell, k, j)}\)).}
	Finally, for all \( (\ell, k, j) \in \QQ \), we observe that
	\begin{align*}
		\sum_{\substack{(\ell', k', j') \in \QQ  \\ |\ell', k', j'| >
				|\ell,k,j|}}  \Eta_{\ell'}^{k',j'} \,
		 & \Zeta_{\ell'}^{k',j'}
		\!= \! \!
		\sum_{j' = j + 1}^{\jj[\ell, k]} \Eta_{\ell}^{k, j'} \,
		\Zeta_{\ell}^{k, j'}
		+
		\! \! \! \sum_{k' = k + 1}^{\kk[\ell]} \sum_{j'=0}^{\jj[\ell,k']}
		\Eta_{\ell}^{k',j'} \, \Zeta_{\ell}^{k',j'}
		+
		\! \! \!  \sum_{\ell' = \ell+1}^{\elll} \sum_{k'=0}^{\kk[\ell']}
		\sum_{j'=0}^{\jj[\ell',k']}
		\Eta_{\ell'}^{k',j'} \, \Zeta_{\ell'}^{k',j'}
		\\&
		\eqreff*{eq:quasi_contraction_Eta_ell^kj}
		\lesssim
		\Eta_{\ell}^{k, j} \, \Zeta_{\ell}^{k, j}
		+ \sum_{k'=k+1}^{\kk[\ell]} \Eta_{\ell}^{k', 0} \, \Zeta_{\ell}^{k', 0}
		+ \sum_{\ell' = \ell+1}^{\elll} \sum_{k'=0}^{\kk[\ell']}
		\Eta_{\ell'}^{k', 0} \, \Zeta_{\ell'}^{k', 0}
		\\&
		\eqreff*{eq1:step9}
		\lesssim
		\Eta_{\ell}^{k, j} \, \Zeta_{\ell}^{k, j}
		+ \sum_{\substack{(\ell',k',\jj) \in \QQ \\ |\ell',k',\jj| >
				|\ell,k,\jj|}} \Eta_{\ell'}^{k', \jj} \,
		\Zeta_{\ell'}^{k', \jj}
		\eqreff{eq:summability_k}
		\lesssim
		\Eta_{\ell}^{k, j} \, \Zeta_{\ell}^{k, j}
		+
		\Eta_{\ell}^{k, \jj} \, \Zeta_{\ell}^{k, \jj}
		\eqreff{eq:quasi_contraction_Eta_ell^kj}
		\lesssim
		\Eta_{\ell}^{k, j} \, \Zeta_{\ell}^{k, j}.
	\end{align*}

	\textbf{Step~4.}
	Since the index set \(\QQ\) is linearly
	ordered with respect to
	the total step counter
	\(|\cdot, \cdot, \cdot |\), tail-summability
	in Step~3 and
	the equivalence of tail-summability and \(R\)-linear convergence
	from~\cite[Lemma~10]{fps2023} conclude
	the proof of
	\eqref{eq:full_linear_convergence} in
	Theorem~\ref{lem:full_linear_convergence}.
\end{proof}

\section{Optimal complexity of Algorithm~\ref{algorithm:afem}}\label{section:optimal_complexity}
Full linear convergence~\eqref{eq:full_linear_convergence} has a
simple but crucial consequence. Using a geometric series
argument, one can
prove that the
cumulative computational
cost up to a given level is bounded by the cost of the said
level;
see~\cite[Corollary~14]{fps2023}, where only
the primal quasi-error \(\Eta_{\ell}^{k,j}\) has to be replaced by the
quasi-error product~\(\Eta_{\ell}^{k,j} \, \Zeta_{\ell}^{k,j}\).
As a consequence, the convergence
rates with respect
to the number of degrees of freedom (defined as \(M(r)\) in
\eqref{eq:rate_complexity} below) and the rates
with respect to the overall
computational cost (cf.\ \eqref{eq:cost} and the discussion following
the
statement of Algorithm~\ref{algorithm:afem}) coincide.
\begin{corollary}[rates = complexity
			{\cite[Corollary~14]{fps2023}}]
	\label{cor:rate_complexity}
	Suppose the assumptions of
	Theorem~\ref{lem:full_linear_convergence}.
	For all \(r > 0\), the output
	\((\TT_{\ell})_{\ell \in \N_0}\) of
	Algorithm~\ref{algorithm:afem} satisfies
	\begin{equation}\label{eq:rate_complexity}
		\medmuskip = -2mu
		M(r)
		\coloneqq
		\!\!
		\sup \limits_{
			\substack{
				(\ell, k, j) \in \QQ
			}
		}
		\bigl(\# \TT_{\ell}\bigr)^r \Eta_\ell^{k,j} \, \Zeta_\ell^{k,j}
		\le \!\!
		\sup \limits_{
			\substack{(\ell, k, j) \in \QQ
			}
		}
		\Bigl( \!\!\!
		\sum \limits_{
			\substack{
				(\ell^\prime, k^\prime, j^\prime) \in \QQ
				\\
				| \ell^\prime, k^\prime, j^\prime | \le | \ell, k, j |
			}
		}
		\# \TT_{\ell^\prime}
		\Bigr)^r
		\Eta_\ell^{k,j} \Zeta_\ell^{k,j}
		\le \!
		C_{\rm cost}(r)  M(r),
	\end{equation}
	with the constant $C_{\rm cost}(r)
		\coloneqq \Clin /
		(1-\qlin^{1/r})^r > 0$. \qed
\end{corollary}
While Theorem~\ref{lem:full_linear_convergence} only concerns
\(R\)-linear convergence, a sufficiently small choice of the adaptivity
parameters \(\theta, \lamsym\), and \(\lamalg\) even guarantees the optimal
convergence rate \(r = s+t\) with respect to
computational cost, i.e., the overall computational time. Here,
we suppose that the primal
solution \(u^\star\)
to \eqref{eq:weak_formulation} can be approximated at rate \(s\) and the dual
solution \(z^\star\) to~\eqref{eq:dual_problem} can be approximated at rate
\(t\).
To formalize this idea, we introduce the notion of approximation
classes~\cite{bdd2004, stevenson2007, ckns2008, axioms}.  For
\(s, t > 0\), define
\begin{equation*}
	\|u^\star\|_{\A_s}
	\coloneqq
	\sup_{N \in \N_0}
	\Bigl(
	\bigl( N+1 \bigr)^s \min_{\TT_{\rm opt} \in \T_N } \eta_{\rm opt} (u^\star_{\rm opt})
	\Bigr),
	\quad
	\|z^\star\|_{\A_t}
	\coloneqq
	\sup_{N \in \N_0}
	\Bigl(
	\bigl( N+1 \bigr)^t \min_{\TT_{\rm opt} \in \T_N} \zeta_{\rm opt} (z^\star_{\rm opt})
	\Bigr),
\end{equation*}
where $\eta_{\rm opt}(\cdot)$ and $\zeta_{\rm opt}(\cdot)$ denote the estimator
values for the exact discrete solutions
\(u_{\mathrm{opt}}^\star\) and \(z_{\mathrm{opt}}^\star\) on
the unavailable optimal triangulations $\TT_{\rm opt} \in \T_N
	(\TT)$. We stress that \(\|u^\star\|_{\A_s}\) and
\(\|z^\star\|_{\A_t}\) can equivalently be defined by energy error plus
data oscillations~\cite{ffp2014, axioms}.
\begin{theorem}[optimal complexity]\label{th:optimal_complexity}
	Suppose that the estimators \(\eta\) and \(\zeta\) satisfy
	\eqref{axiom:stability}--\eqref{axiom:discrete_reliability}
	and \eqref{axiom:qm} and suppose
	quasi-orthogonality~\eqref{axiom:orthogonality}.
	Recall \(\lamalg^\star\) from
	Lemma~\ref{lem:inexact_Zarantonello_contraction} and \(\lambda^\star\) from
	\eqref{eq:choice_lambda} in
	Theorem~\ref{lem:full_linear_convergence}.
	Define the constants
	\begin{equation}\label{eq:def:Calg}
		\begin{aligned}
			\lamsym^\star
			             & \coloneqq
			\min \{1, \Cstab^{-1} \, C_{\textup{alg}}^{-1}\} \le 1
			\quad
			\text{with}
			\quad
			C_{\mathrm{alg}}
			\coloneqq
			\frac{1}{1-\qsym} \,
			\Bigl(
			\frac{2 \, \qalg}{1-\qalg} \, \lamalg^\star + \qsym
			\Bigr),
			\\
			\theta^\star & \coloneqq (1 + \Cstab^2 \, \Crel^2)^{-1} < 1.
		\end{aligned}
	\end{equation}
	Suppose that \(\theta\), \(\lamsym\), and \(\lamalg\) are
	sufficiently small in the sense of
	\begin{equation}
		\begin{aligned}
			0 & < \lamalg \le \lamalg^\star,
			\quad
			0 < \lamsym < \lamsym^\star,
			\quad
			\text{and }
			\quad
			\lamalg \, \lamsym < \lambda^\star,
			\\
			0
			  & <
			\thetamark
			\coloneqq
			\frac{(\theta^{1/2}+ \, \lamsym / \lamsym^\star)^{2}
			}{(1-\lamsym / \lamsym^\star)^{2}}
			<
			\theta^\star
			<
			1.
		\end{aligned}
	\end{equation}
	Then, Algorithm~\ref{algorithm:afem} guarantees, for all \(s, t > 0\), that
	\begin{equation}\label{eq:optimal_complexity}
		\sup_{
			\substack{
				(\ell,k, j) \in \QQ
			}
		}
		\Bigl(
		\sum_{\substack{
				(\ell^\prime, k^\prime, j^\prime) \in \QQ
				\\
				|\ell^\prime,k^\prime, j^\prime| \le |\ell,k, j|
			}
		}
		\#\TT_{\ell^\prime}
		\Bigr)^{s+t} \,
		\Eta_\ell^{k,j} \, \Zeta_\ell^{k, j}
		\le
		\Copt \,
		\max\{
		\|u^\star\|_{\A_s} \, \|z^\star\|_{\A_t}, \,
		\Eta_{0}^{0,0} \, \Zeta_{0}^{0,0}
		\}.
	\end{equation}
	The constant \(\Copt\) depends only on \(\Cstab\),
	\(\Crel\),
	\(\Cdrel\), \(\Cmark\), \(\Cmesh\), $\Clin$, $\qlin$, \(\# \TT_{0}\),
	and \(s+t\).
	In particular, there holds optimal complexity of
	Algorithm~\ref{algorithm:afem}.
\end{theorem}

The proof of Theorem~\ref{th:optimal_complexity} employs the
following result from \cite{bhimps2023b}
providing estimator equivalence between the (unavailable)
estimators for the exact
discrete solutions~\(u_\ell^\star, z_\ell^\star\) and the
estimators at the computed approximations
\(u_\ell^{\mm, \nn}, z_\ell^{\mmu, \nnu}\).
\begin{lemma}[estimator equivalence
			{\cite[Lemma~15]{bhimps2023b}}]
	\label{lem:estimator-equivalence}
	Recall the constants \(\lamsym^\star\), $C_{\mathrm{alg}} >
		0$ from~\eqref{eq:def:Calg} and
	$\lamalg^\star > 0$ from
	Lemma~\ref{lem:inexact_Zarantonello_contraction}.
	Then, for all \(0 < \theta \le 1\), \(0 < \lamalg \le \lamalg^\star\),
	\(
	0
	<
	\lamsym
	<
	\lamsym^\star
	\),
	it holds that
	\begin{equation}\label{eq:corrigendum:equivalence}
		\medmuskip = -2mu
		\begin{aligned}
			\bigl(1 - \lamsym/ \lamsym^\star\bigr)
			\,\eta_\ell(u_\ell^{\mm, \nn})
			 & \le
			\eta_\ell(u_\ell^{\star})
			\le
			\bigl( 1 + \lamsym/ \lamsym^\star\bigr)
			\,\eta_\ell(u_\ell^{\mm, \nn})
			\,
			\text{ for all \((\ell, \mm, \nn) \in \QQ^u\)},
			\\
			\bigl( 1 - \lamsym / \lamsym^\star\bigr)
			\,\zeta_\ell(z_\ell^{\mmu, \nnu})
			 & \le
			\zeta_\ell(z_\ell^{\star})
			\le
			\bigl( 1 + \lamsym/ \lamsym^\star \bigr)
			\,\zeta_\ell(z_\ell^{\mmu, \nnu})
			\,
			\text{ for all \((\ell, \mmu, \nnu) \in \QQ^z\)}.
			\hspace{-0.7em}\qed
		\end{aligned}
	\end{equation}
\end{lemma}

\begin{proof}[Proof of Theorem~\ref{th:optimal_complexity}]
	By Corollary~\ref{cor:rate_complexity}, it suffices to prove that,
	for any \(s, t > 0\),
	\begin{equation}\label{eq:proof_optimal_complexity}
		\sup \limits_{\substack{(\ell, k, j) \in \QQ}}
		\bigl(\# \TT_{\ell}\bigr)^{s+t} \, \Eta_{\ell}^{k, j} \, \Zeta_{\ell}^{k, j}
		\lesssim
		\max\{
		\|u^\star\|_{\A_s} \, \|z^\star\|_{\A_t},
		\Eta_{0}^{0,0} \, \Zeta_{0}^{0,0}
		\}.
	\end{equation}
	Since the inequality becomes trivial if either
	\(\|u^\star\|_{\A_s} = \infty\) or
	\(\|z^\star\|_{\A_t} = \infty\), we may assume
	\(\|u^\star\|_{\A_s} \, \|z^\star\|_{\A_t} <
	\infty\). The proof consists of three steps.

	\medskip
	\textbf{Step 1.}
	With
	\(
	0
	<
	\thetamark
	\coloneqq
	(\theta^{1/2}+ \, \lamsym / \lamsym^\star)^{2}  \, (1-\lamsym / \lamsym^\star)^{-{2}}
	<
	\theta^\star,
	\)
	the validity of \eqref{axiom:discrete_reliability} for both estimators and
	\cite[Lemma~14]{fghpf2016}
	guarantee the existence of sets \(\RR_{\ell^\prime} \subseteq
	\TT_{\ell^\prime}\) with \(0 \le \ell^\prime < \elll\) such
	that
	\begin{subequations}
		\begin{align}
			\# \RR_{\ell^\prime}
			\lesssim
			\bigl(
			\|u^\star\|_{\A_s} \,
			\|z^\star\|_{\A_t}
			\bigr)^{1/(s+t)} \,
			\bigl[
				\eta_{\ell^\prime}(u_{\ell^\prime}^{\star}) \, \zeta_{\ell^\prime}(z_{\ell^\prime}^{\star})
				\bigr]^{-1/(s+t)}
			\label{eq:R_estimate},
			\\
			\thetamark \, \eta_{\ell^\prime}(u_{\ell^\prime}^{\star})
			\le
			\eta_{\ell^\prime}(\RR_{\ell^\prime},  u_{\ell^\prime}^\star)
			\quad
			\text{or}
			\quad
			\thetamark \, \zeta_{\ell^\prime}(z_{\ell^\prime}^{\star})
			\le
			\zeta_{\ell^\prime}(\RR_{\ell^\prime}, z_{\ell^\prime}^\star)
			\label{eq:proof_marking}.
		\end{align}
	\end{subequations}
	For \(0 \le \ell^\prime < \elll\),  the estimator
	equivalence~\eqref{eq:corrigendum:equivalence} in
	Lemma~\ref{lem:estimator-equivalence} leads to
	\begin{equation*}
		\bigl(1-\lamsym / \lamsym^\star \bigr) \,
		\eta_{\ell'}(u_{\ell'}^{\mm, \nn})
		\le
		\eta_{\ell'}(u_{\ell'}^\star)
		\quad
		\text{and}
		\quad
		\bigl(1-\lamsym / \lamsym^\star \bigr) \,
		\zeta_{\ell'}(z_{\ell'}^{\mmu, \nnu})
		\le
		\zeta_{\ell'}(z_{\ell'}^\star)
	\end{equation*}
	and consequently with \eqref{eq:R_estimate} to
	\begin{equation}\label{eq2:R_estimator_final}
		\# \RR_{\ell^\prime}
		\lesssim
		\bigl(
		\|u^\star\|_{\A_s} \, \|z^\star\|_{\A_t}
		\bigr)^{1/(s+t)} \,
		\bigl[
			\eta_{\ell^\prime}(u_{\ell^\prime}^{\mm, \nn}) \,
			\zeta_{\ell^\prime}(z_{\ell^\prime}^{\mmu, \nnu})
			\bigr]^{-1/(s+t)}.
	\end{equation}
	Note that the stopping
	criteria~\eqref{eq:m_stopping_criterion} and \eqref{eq:mu_stopping_criterion}
	lead to
	\begin{equation*}
		\Eta_{\ell'}
		\simeq
		\vvvert u_{\ell'}^\star - u_{\ell'}^{\mm, \nn} \vvvert +
		\eta_{\ell'}(u_{\ell'}^{\mm, \nn})
		\eqreff{eq:m_stopping_criterion}
		\lesssim
		\eta_{\ell'}(u_{\ell'}^{\mm, \nn})
		\quad
		\text{and}
		\quad
		\Zeta_{\ell'}
		\simeq
		\vvvert z_{\ell'}^\star - z_{\ell'}^{\mmu, \nnu} \vvvert +
		\zeta_{\ell'}(z_{\ell'}^{\mmu, \nnu})
		\eqreff{eq:mu_stopping_criterion}
		\lesssim
		\zeta_{\ell'}(z_{\ell'}^{\mmu, \nnu})
	\end{equation*}
	and with~\eqref{eq:Eta_ell^0} to
	\begin{equation}\label{eq:Delta_estimator}
		\Eta_{\ell'+1}^{0, \jj} \, \Zeta_{\ell'+1}^{0, \jj}
		\eqreff{eq:Eta_ell^0}
		\lesssim
		\Eta_{\ell'} \, \Zeta_{\ell'}
		\lesssim
		\eta_{\ell'}(u_{\ell'}^{\mm, \nn}) \,
		\zeta_{\ell'}(z_{\ell'}^{\mmu,
			\nnu}).
	\end{equation}
	Hence, the combination of \eqref{eq2:R_estimator_final} and
	\eqref{eq:Delta_estimator} reads
	\begin{equation}\label{eq:R_Delta}
		\# \RR_{\ell^\prime}
		\lesssim
		\bigl(
		\|u^\star\|_{\A_s} \, \|z^\star\|_{\A_t}
		\bigr)^{1/(s+t)} \,
		\bigl[
			\Eta_{\ell^\prime+1}^{0, \jj} \, \Zeta_{\ell^\prime+1}^{0, \jj}
			\bigr]^{-1/(s+t)}.
	\end{equation}

	\medskip
	\textbf{Step 2.}
	Recall from \cite[Theorem~8]{bgip2023} that the set
	\(\RR_{\ell^\prime}\)
	satisfies the
	Dörfler criterion from
	Algorithm~\ref{algorithm:afem}\ref{alg:mark} with the same
	parameter \(\theta\).
	The quasi-minimality of
	\(\MM_{\ell^\prime}\) implies
	\begin{align}\label{eq:dorfleropt}
		\# \MM_{\ell^\prime}
		\le
		C_{\rm mark} \, \# \RR_{\ell^\prime}
		\quad\text{for all $0 \le
				\ell' < \elll$}
	\end{align}
	with the constant $\Cmark \ge 1$ from Algorithm~\ref{algorithm:afem}.

	\textbf{Step 3.} Let $(\ell, k, j) \in \QQ$. Full
	linear convergence \eqref{eq:full_linear_convergence} from
	Theorem~\ref{lem:full_linear_convergence} yields that
	\begin{equation}\label{eq:lin_cv_sum}
		\begin{aligned}
			\sum_{
			\substack{
			(\ell^{\prime},k^{\prime},j^{\prime}) \in \QQ
			\\
			|\ell^{\prime},k^{\prime},j^{\prime}|
			\le
			|\ell,k,j|
			}
			} \hspace{-0.5cm}
			(
			\Eta_{\ell^{\prime}}^{k^{\prime}, j^{\prime}} \,
			\Zeta_{\ell^{\prime}}^{k^{\prime}, j^{\prime}}
			)^{-1/(s+t)}
			&\eqreff*{eq:full_linear_convergence}
			\lesssim
			(\Eta_{\ell}^{k, j} \, \Zeta_{\ell}^{k, j})^{-1/(s+t)}
			\hspace{-0.5cm}
			\sum_{
			\substack{
			(\ell^{\prime},k^{\prime},j^{\prime}) \in \QQ
			\\
			|\ell^{\prime},k^{\prime},j^{\prime}|
			\le
			|\ell,k,j|
			}
			}
			\hspace{-0.5cm}
			(\qlin^{1/s})^{|\ell, k, j| - |\ell', k', j'|}
			\\
			&\lesssim
			(\Eta_{\ell}^{k, j} \, \Zeta_{\ell}^{k, j})^{-1/(s+t)}.
		\end{aligned}
	\end{equation}
	NVB refinement satisfies the mesh-closure estimate~\cite[Eqn.~(2.9)]{axioms} reading,
	\begin{align}\label{eq:meshclosure}
		\# \TT_\ell - \# \TT_0
		\le
		\Cmesh \sum_{\ell' = 0}^{\ell-1} \# \MM_{\ell'}
		\quad\text{for all \(\ell \ge
			0\)},
	\end{align}
	where $\Cmesh > 1$ depends only on $\TT_{0}$.
	Thus, for $(\ell, k, j) \in \QQ$, we have by the mesh-closure
	estimate~\eqref{eq:meshclosure}, quasi-optimality of
	D\"orfler marking \eqref{eq:dorfleropt}, and the result
	\eqref{eq:lin_cv_sum} that
	\begin{align*}
		\# \TT_{\ell} - \# \TT_{0}
		\
		 & \eqreff*{eq:meshclosure}
		\lesssim
		\
		\sum_{\ell^\prime = 0}^{\ell-1}  \# \MM_{\ell^\prime}
		\eqreff{eq:dorfleropt}
		\lesssim
		\sum_{\ell^\prime = 0}^{\ell-1}  \#
		\RR_{\ell^\prime}
		\eqreff*{eq:R_Delta}
		\lesssim
		\bigl(
		\|u^\star\|_{\A_s} \, \|z^\star\|_{\A_t}
		\bigr)^{1/(s+t)} \,
		\sum_{\ell^{\prime} = 0}^{\ell-1}
		\bigl(
		\Eta_{\ell^{\prime}+1}^{0, \jj} \, \Zeta_{\ell^{\prime}+1}^{0, \jj}
		\bigr)^{-1/(s+t)}
		\\
		 & \le
		\bigl(
		\|u^\star\|_{\A_s} \, \|z^\star\|_{\A_t}
		\bigr)^{1/(s+t)} \! \!
		\sum_{
		\substack{
		(\ell^{\prime},k^{\prime},j^{\prime}) \in \QQ
		\\
		|\ell^{\prime},k^{\prime},j^{\prime}|
		\le
		|\ell,k,j|
		}
		}
		(\Eta_{\ell'}^{k', j'} \, \Zeta_{\ell'}^{k', j'})^{-1/(s+t)}
		\\
		 & \eqreff*{eq:lin_cv_sum}
		\lesssim
		\bigl(
		\|u^\star\|_{\A_s} \, \|z^\star\|_{\A_t}
		\bigr)^{1/(s+t)}
		(\Eta_{\ell}^{k, j} \, \Zeta_{\ell}^{k, j})^{-1/(s+t)}.
	\end{align*}
	Rearranging the terms and noting that \(1 \le \# \TT_{\ell} - \# \TT_{0}\)
	implies
	$\# \TT_{\ell} - \# \TT_{0} +1 \le 2 \, (\#\TT_{\ell} - \# \TT_{0})$, we
	obtain, for \(\ell > 0\), that
	\begin{subequations}\label{eq:optimality}
		\begin{equation}\label{eq:first_optimality}
			(\# \TT_{\ell} - \# \TT_{0} + 1)^{s+t} \, \Eta_{\ell}^{k, j} \,
			\Zeta_{\ell}^{k, j}
			\lesssim
			\|u^\star\|_{\A_s} \, \|z^\star\|_{\A_t}.
		\end{equation}
		Moreover, full linear
		convergence~\eqref{eq:full_linear_convergence} proves that
		\begin{equation}\label{eq:second_optimality}
			(\# \TT_{0} - \# \TT_{0} + 1)^{s+t} \,
			\Eta_{0}^{k, j} \, \Zeta_{0}^{k, j}
			=
			\Eta_{0}^{k, j} \, \Zeta_{0}^{k, j}
			\lesssim
			\Eta_{0}^{0, 0} \, \Zeta_{0}^{0, 0}.
		\end{equation}
	\end{subequations}
	We recall from~\cite[Lemma~22]{bhp2017} that, for
	all
	$\TT_{\ell} \in \T$, it holds
	\begin{equation}\label{eq:bhp-lemma22}
		\# \TT_{\ell} - \# \TT_{0} +1
		\le
		\# \TT_{\ell}
		\le
		\# \TT_{0} \, (\# \TT_{\ell} - \# \TT_{0} +1).
	\end{equation}
	This shows, for all $(\ell, k, j) \in \QQ$,
	\begin{equation*}
		(\# \TT_{\ell})^{s+t} \, \Eta_\ell^{k,j} \, \Zeta_\ell^{k,j}
		\eqreff{eq:bhp-lemma22}
		\lesssim
		(\# \TT_{\ell} - \# \TT_{0} + 1)^{s+t} \, \Eta_\ell^{k,j} \,
		\Zeta_\ell^{k,j}
		\eqreff{eq:optimality}
		\lesssim
		\max \{
			\|u^\star\|_{\A_s} \, \|z^\star\|_{\A_t},
			\Eta_{0}^{0,0} \, \Zeta_{0}^{0,0}
		\}
	\end{equation*}
	and concludes the proof of~\eqref{eq:proof_optimal_complexity}.
\end{proof}

%

\section{Numerical examples}\label{section:numerics}
In this section, we present numerical experiments using 
the open source software package MooAFEM~\cite{mooafem}\footnote{All 
experiments presented in this paper are reproducible with 
the openly available software package 
under \url{https://www.tuwien.at/mg/asc/praetorius/software/mooafem}.}.
In the following, Step~\ref{alg:solve_primal} and \ref{alg:solve_dual} of Algorithm~\ref{algorithm:afem} employ
the optimal \(hp\)-robust local multigrid method from 
\cite{imps2022} 
as an algebraic solver. If not explicitly stated 
otherwise, we choose 
the 
parameters \(\theta = 0.5\), \(\delta = 
0.5\), \(\lamsym = \lamalg = 0.7\) in 
Algorithm~\ref{algorithm:afem} throughout the 
numerical 
experiments.

\subsection{Singularity in the goal functional} The first model problem is a 
nonsymmetric variant 
of the benchmark problem from~\cite[Section~4.1]{bgip2023} with 
a singularity 
only in the 
goal functional.
On the unit square \(\Omega = (0,1)^2 \subset \R^2\), we consider
\begin{equation}\label{eq:singularity_goal}
	- \Delta u^\star + x \cdot \nabla u^\star + 
	u^\star 
	= 
	f
	\quad \text{in } \Omega 
	\quad
	\text{subject to}
	\quad
	u^\star = 0
	\quad \text{on } \partial \Omega,
\end{equation}
where the right-hand side is chosen such that the exact solution \(u^\star\) 
reads 
\[u^\star(x) = x_1 \, x_2 \, (1-x_1) \, (1-x_2).\]
Consider \(g = 0\) and \(\gvec = \chi_K \, (1, 
0)^\top\) 
in the quantity of interest
\[
	 G(u^\star) \coloneqq \int_{K} \partial_{x_1} u^\star \d{x} = 
	 11/960
	 \quad
	 \text{with \(K 
	 	\coloneqq \mathrm{conv} \{
			(\nicefrac{1}{2}, 1), (1, 
	 	\nicefrac{1}{2}), (1,1)\}\)
		}.
\]
Figure~\ref{fig:mesh} (left) displays a 
mesh 
generated by Algorithm~\ref{algorithm:afem} and the support \(K\) of 
\(\gvec\). The error estimator captures and resolves the two 
point singularities induced by \(G\).

\subsection{Geometric singularity and strong convection} The second 
benchmark problem investigates
\(\Omega = (-1, 1)^2 \setminus \mathrm{conv}\{(0,0), (-1, 0), (-1,-1)\} \subset 
\R^2\) with 
the Dirichlet boundary \(\Gamma_D = \mathrm{conv}\{(-1,0), (0,0)\} 
\cup \, \mathrm{conv}\{(0,0), (-1,-1)\}\) and Neumann 
boundary \(\Gamma_N = \partial \Omega \setminus \Gamma_D\); see 
Figure~\ref{fig:mesh} (right) for a visualization of the geometry. 
We 
consider
\begin{equation}\label{eq:geometric_singularity}
	- \Delta u^\star + (5, 5)^\top \cdot \nabla u^\star
	= 
	1
	\ \text{in } \Omega 
	\quad \text{subject to} \quad
	u^\star = 0
	\text{ on \(\Gamma_D\) and }
	\nabla u^\star \cdot \boldsymbol{n} = 0
	\text{ on \(\Gamma_N\)}.
\end{equation}
Consider \(g = 0\) and \(\gvec = \chi_S \, (1, 
1)^\top\) in 
the 
quantity of 
interest
\[
G(u^\star) = \int_{S} \partial_{x_1} u^\star + 
\partial_{x_2} u^\star \d{x}
\quad
\text{with \(S \coloneqq (-\nicefrac{1}{2}, \nicefrac{1}{2})^2 
\cap \Omega\)}
.
\]
The exact solution \(u^\star\) is not known analytically in 
this case so that we do not have access to the 
exact goal error \(| 
G(u^\star) - G_\ell(u_\ell^{\mm, \nn}, z_\ell^{\mmu, \nnu})|\).
Figure~\ref{fig:mesh} (right) shows a mesh generated by 
Algorithm~\ref{algorithm:afem} 
as well as the configuration, i.e., the support \(S\) of 
\(\gvec\) in blue, the 
Dirichlet boundary in red solid lines, and the Neumann boundary in 
green dashed lines.
\begin{figure}[htbp!]
	\centering
	\resizebox{0.4\textwidth}{!}{
	\includegraphics{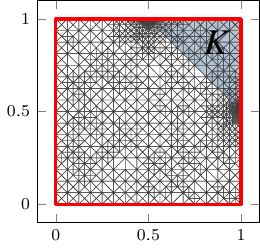}
	}
	\hfil
	\resizebox{0.4\textwidth}{!}{
	\includegraphics{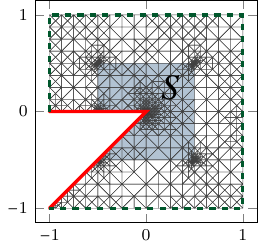}
	}
	\caption{\label{fig:mesh}  Left: Mesh \(\TT_{15}\) for the 
	problem~\eqref{eq:singularity_goal} generated by 
	Algorithm~\ref{algorithm:afem} with 
	\(\# \TT_{15} = 2315\). Right: Mesh \(\TT_{18}\) for the 
	problem~\eqref{eq:geometric_singularity} with \(\# \TT_{18} = 
	2130\), 
	where the 
	Dirichlet boundary part \(\Gamma_D\) is marked by red solid lines
	and the Neumann boundary part \(\Gamma_N\) by green dashed lines.
}
\end{figure}

\textbf{Optimality of Algorithm~\ref{algorithm:afem}.} 
Figure~\ref{fig:optimality} displays the
estimator product \(\eta_\ell(u_\ell^{\mm, \nn}) 
\, 
\zeta_\ell(z^{\underline{\mu}, \underline{\nu}})\) 
and the goal error 
\(| G(u^\star) - G_\ell(u_\ell^{\mm, \nn}, 
z_\ell^{\underline{\mu}, \underline{\nu}}) |\) 
from~\eqref{eq:goalError:Estimate} for the 
problem~\eqref{eq:singularity_goal}, due to higher-order 
approximations, we only show results prior to machine precision.
For all investigated polynomial degrees 
\(p\), the goal error and the estimator product are 
indeed 
equivalent. Algorithm~\ref{algorithm:afem} 
achieves the optimal rate 
\(-p\) with 
respect to the cumulative computational work and with respect to the 
cumulative computational time in Figure~\ref{fig:optimality} for 
problem~\eqref{eq:singularity_goal} and Figure~\ref{fig:Zshape_convergence} for 
problem~\eqref{eq:geometric_singularity}.
Figure~\ref{fig:timing} shows that the proposed algorithm  indeed achieves 
linear complexity and is substantially faster than the \textsc{Matlab} built-in 
direct solver as the slightly larger slope of the
latter indicates super-linear complexity. 
Table~\ref{tab:cost} displays the weighted costs 
\begin{equation}\label{eq:weighted_costs}
	\eta_{\ell}(u_\ell^{\underline{m}, \underline{n}}) \, 
	\zeta_{\ell}(z_\ell^{\underline{\mu}, \underline{\nu}})
	\Bigl(\sum_{
		\substack{(\ell', k', j') \in \QQ 
			\\ |\ell', k', j'| \le  		|\ell,\kk,\jj|}
		} 
            \mathtt{time}(\ell', k', j') \Bigr)^p
\end{equation}
of 
Algorithm~\ref{algorithm:afem} for polynomial degree \(p=2\) with \(\mathtt{time}(\ell', k', j')\) in seconds and 
highlights the corresponding optimal choices of the parameters. 
This justifies the 
selection of \(\theta = 0.5\) together with 
larger symmetrization 
parameter \(\lamsym = 0.7\), and algebraic solver parameter \(\lamalg = 0.7\).
The table for the second benchmark problem
    from~\eqref{eq:geometric_singularity} leads to similar results and is therefore omitted.
While the choice of the damping parameter \(0 < \delta < 2\alpha/L^2\)
in~\eqref{eq:Zarantonello-iterations} is crucial in the case of large convection
to guarantee the contraction property~\eqref{eq:Zarantonello-contraction},
the adaptivity parameters appear more robust with respect to
other coefficients in~\eqref{eq:model_problem}.

Finally, in Figure~\ref{fig:iterations}, we display the number of total solver
steps \(| \ell, \underline{m}, \underline{n}| - |\ell, 0, 0|\) resp.\ \(| \ell, \underline{\mu}, \underline{\nu}| - |\ell, 0, 0|\) on each mesh level for both
benchmark problems~\eqref{eq:singularity_goal} and~\eqref{eq:geometric_singularity}.
The plots show that the two iterations often stop after the same number of steps.
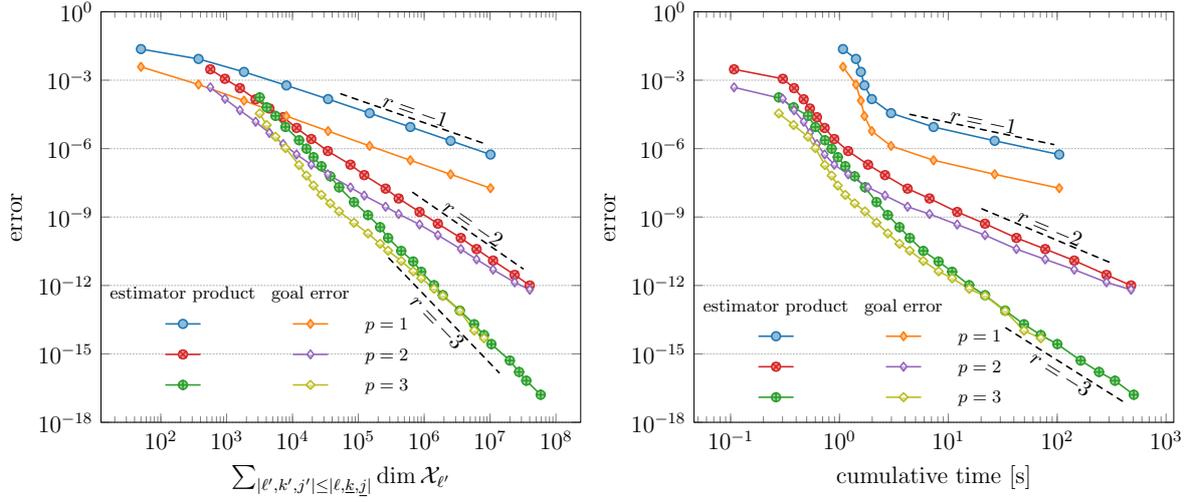
\begin{figure}
	\resizebox{\textwidth}{!}{
\pgfplotstableread[col sep = comma]{
	ndofs,  nElems, combinedEstimator,  GoalError,time_cumulative,relative_time, work
	  41,     126,      2.31414e-02,3.84768e-03,    1.07724e+00,  2.15449e-02,5.00000e+01
	  73,     242,      1.47569e-02,1.52268e-03,    1.27295e+00,  1.15723e-02,1.60000e+02
	 137,     510,      8.54214e-03,6.45762e-04,    1.41798e+00,  6.62606e-03,3.74000e+02
	 278,    1052,      4.39263e-03,2.82059e-04,    1.47635e+00,  3.15459e-03,8.42000e+02
	 558,    2130,      2.30667e-03,1.26305e-04,    1.57169e+00,  1.58436e-03,1.83400e+03
	1115,    4348,      1.16646e-03,5.71738e-05,    1.62370e+00,  7.98280e-04,3.86800e+03
	2244,    9050,      5.89691e-04,2.66693e-05,    1.69571e+00,  4.02589e-04,8.08000e+03
	4622,   18072,      2.90811e-04,1.22185e-05,    1.80404e+00,  2.03616e-04,1.69400e+04
	9174,   37713,      1.48216e-04,5.84359e-06,    1.98274e+00,  1.11390e-04,3.47400e+04
   19063,   76722,      7.18022e-05,2.69668e-06,    2.30891e+00,  6.18944e-05,7.20440e+04
   38644,  156191,      3.58197e-05,1.31114e-06,    2.97516e+00,  3.90646e-05,1.48204e+05
   78511,  313751,      1.76320e-05,6.20380e-07,    4.30442e+00,  2.77054e-05,3.03568e+05
  157451,  634849,      8.86721e-06,3.08135e-07,    7.31765e+00,  2.34087e-05,6.16172e+05
  318270, 1272065,      4.36095e-06,1.47944e-07,    1.35168e+01,  2.13480e-05,1.24933e+06
  637198, 2558227,      2.19933e-06,7.46053e-08,    2.66353e+01,  2.09770e-05,2.51907e+06
 1280829, 5106128,      1.08454e-06,3.61679e-08,    5.23958e+01,  2.05088e-05,5.07387e+06
 2555433,10221606,      5.49089e-07,1.83987e-08,    1.04429e+02,  2.04706e-05,1.01753e+07
 5114242,10221606,      2.71280e-07,8.97122e-09,    1.42428e+02,  1.39434e-05,2.03900e+07
}{\ComplexityOne}%
\pgfplotstableread[col sep = comma]{
	ndofs,  nElems,combinedEstimator,  GoalError,time_cumulative,relative_time, work
	 145,     104,      2.99996e-03,4.85034e-04,    1.07478e-01,  4.75566e-04,5.65000e+02
	 229,     176,      1.14929e-03,1.51594e-04,    3.01815e-01,  7.98452e-04,9.43000e+02
	 383,     310,      4.48363e-04,4.77986e-05,    3.82502e-01,  5.92108e-04,1.58900e+03
	 659,     446,      1.45652e-04,1.50284e-05,    4.68814e-01,  4.02070e-04,2.75500e+03
	 935,     722,      5.74064e-05,4.86716e-06,    5.34872e-01,  3.14261e-04,4.45700e+03
	1509,    1122,      2.37864e-05,1.57423e-06,    6.20398e-01,  2.24619e-04,7.21900e+03
	2327,    1984,      8.04252e-06,5.54883e-07,    7.31099e-01,  1.69001e-04,1.15450e+04
	4063,    3848,      2.62902e-06,1.97043e-07,    8.95769e-01,  1.15583e-04,1.92950e+04
	7851,    7009,      7.82524e-07,7.45983e-08,    1.20223e+00,  7.96917e-05,3.43810e+04
   14216,   11906,      1.97669e-07,1.97961e-08,    1.83029e+00,  6.62092e-05,7.58470e+04
   24117,   22924,      6.88793e-08,8.74250e-09,    2.62110e+00,  5.57468e-05,1.22865e+05
   46223,   37807,      1.75803e-08,2.82370e-09,    4.22105e+00,  4.64106e-05,2.59290e+05
   76144,   73870,      6.51940e-09,1.35305e-09,    6.71556e+00,  4.47191e-05,4.09462e+05
  148457,  127726,      1.67137e-09,4.68765e-10,    1.20571e+01,  4.10037e-05,8.50537e+05
  256403,  248864,      5.10332e-10,1.60182e-10,    2.17085e+01,  4.26487e-05,1.61405e+06
  499115,  441210,      1.21067e-10,3.94533e-11,    4.22437e+01,  4.25549e-05,3.59942e+06
  884067,  815212,      3.93167e-11,1.38685e-11,    7.77253e+01,  4.41233e-05,6.24174e+06
 1632989, 1562897,      1.23394e-11,4.93028e-12,    1.43443e+02,  4.40587e-05,1.11253e+07
 3128892, 2746491,      2.90413e-12,1.37136e-12,    2.83767e+02,  4.54361e-05,2.36161e+07
 5497874, 2746491,      1.00164e-12,6.48438e-13,    4.77237e+02,  4.34793e-05,4.00804e+07
}{\ComplexityTwo}%
\pgfplotstableread[col sep = comma]{
	ndofs,  nElems,combinedEstimator,  GoalError,time_cumulative,relative_time, work
	 313,     104,      1.76177e-04,3.46968e-05,    2.77032e-01,  5.22702e-04,3.18000e+03
	 499,     171,      6.45340e-05,1.07525e-05,    3.79052e-01,  4.31722e-04,4.05800e+03
	 808,     267,      2.71487e-05,3.31175e-06,    5.13490e-01,  3.50266e-04,5.52400e+03
	1255,     367,      8.87150e-06,1.05336e-06,    6.00886e-01,  2.61255e-04,7.82400e+03
	1714,     428,      2.33261e-06,1.89897e-07,    7.34921e-01,  2.30962e-04,1.25970e+04
	1993,     518,      9.91448e-07,6.54185e-08,    8.47784e-01,  2.27776e-04,1.63190e+04
	2404,     762,      4.19770e-07,2.39479e-08,    9.65629e-01,  2.13635e-04,2.08390e+04
	3520,    1128,      1.69137e-07,9.32184e-09,    1.13130e+00,  1.69356e-04,2.75190e+04
	5188,    1534,      6.11080e-08,3.95088e-09,    1.37915e+00,  1.38859e-04,3.74510e+04
	7048,    2618,      2.01365e-08,1.75986e-09,    1.71713e+00,  1.27007e-04,5.09710e+04
   11959,    4082,      4.51882e-09,5.59880e-10,    2.19206e+00,  9.44447e-05,8.57860e+04
   18628,    5546,      1.21680e-09,1.91101e-10,    2.78143e+00,  7.67840e-05,1.40122e+05
   25258,    7752,      3.58910e-10,6.68841e-11,    3.55788e+00,  7.21446e-05,2.14096e+05
   35230,   12322,      1.20665e-10,3.33714e-11,    4.42298e+00,  6.40269e-05,2.83176e+05
   55858,   17408,      3.26592e-11,1.17908e-11,    5.88676e+00,  5.34752e-05,4.48302e+05
   78859,   24628,      1.12361e-11,4.14505e-12,    8.13363e+00,  5.22626e-05,6.81747e+05
  111475,   37959,      3.98502e-12,2.05838e-12,    1.08345e+01,  4.91679e-05,9.02105e+05
  171559,   56639,      1.02501e-12,7.09925e-13,    1.55623e+01,  4.57516e-05,1.41233e+06
  255829,   86892,      3.72847e-13,3.45583e-13,    2.17354e+01,  4.27990e-05,1.92018e+06
  392257,  126585,      7.89995e-14,7.50702e-14,    3.32486e+01,  4.26512e-05,3.47927e+06
  571045,  174324,      1.98672e-14,1.04691e-14,    4.95826e+01,  4.36296e-05,5.75216e+06
  786175,  265274,      6.68642e-15,4.81733e-15,    7.07131e+01,  4.51700e-05,8.10038e+06
 1195927,  422902,      2.68592e-15,1.61451e-14,    1.00333e+02,  4.21024e-05,1.04835e+07
 1905679,  565556,      5.11277e-16,4.43812e-14,    1.64494e+02,  4.32778e-05,1.99857e+07
 2548006,  878609,      1.61276e-16,6.45838e-14,    2.42529e+02,  4.77044e-05,2.76117e+07
 3957715, 1294874,      6.68162e-17,1.13002e-13,    3.39430e+02,  4.29684e-05,3.55112e+07
 5831854, 1294874,      1.62050e-17,1.85827e-13,    5.05381e+02,  4.34026e-05,5.87993e+07
}{\ComplexityThree}%
\begin{tikzpicture}[scale=1]
	\begin{loglogaxis}[xlabel = {$\sum_{| \ell', k', j' | \le | \ell, \underline{k}, \underline{j} |} {\rm dim} \, \mathcal{X}_{\ell'}$}, 
			ylabel={error},
			ymin = 1e-18,
			each nth point=2,line join = round,
			ymajorgrids      = true,%
			grid style={densely dotted, semithick},
			line cap=round]
			
			\coordinate (legend) at (axis description cs:0.33,0.2);
			\addplot [pyBlue,thick,mark=*,mark options={solid, 
			fill=pyBlue!50!white, scale = 1.1}, solid] table [x={work},  
			y={combinedEstimator}]{\ComplexityOne}; \label{Figure3_left:EstimatorOne}
			\addplot [pyRed, first = 27, thick,mark=otimes*,mark 
			options={solid,fill=pyRed!50!white, scale = 1.1}, solid] table 
			[x={work}, y={combinedEstimator}]{\ComplexityTwo}; 
			\label{Figure3_left:EstimatorTwo}
			\addplot [pyGreen, first = 28, thick,mark=oplus*,mark 
			options={solid,fill=pyGreen!50!white, scale = 1.1}, solid] table 
			[x={work}, y={combinedEstimator}]{\ComplexityThree}; 
			\label{Figure3_left:EstimatorThree}
			
			\addplot [pyOrange,thick,mark=diamond*,mark options={solid, 
			fill=pyOrange!50!white, scale = 1.1}, solid] 
			table [x={work}, y={GoalError}]{\ComplexityOne}; \label{Figure3_left:quasiErrorOne}
			\addplot [pyPurple,first = 27, thick,mark=halfdiamond*,mark 
			options={solid,fill=pyPurple!50!white, scale = 1.1}, 
			solid] 
			table 
			[x={work}, y={GoalError}]{\ComplexityTwo}; \label{Figure3_left:quasiErrorTwo}
			\addplot [pyYellow,first = 22, thick,mark=halfsquare*,mark 
			options={solid,fill=pyYellow!50!white, scale = 1.1}, solid] 
			table [x={work}, y={GoalError}]{\ComplexityThree}; \label{Figure3_left:quasiErrorThree}
			\logLogSlope[reference]{0.5}{0.8}{0.3}{-1}{4pt}
			\logLogSlope[reference]{0.65}{0.56}{0.23}{-2}{4pt}
			\logLogSlope[reference]{0.6}{0.4}{0.23}{-3}{-5pt}
		\end{loglogaxis}
		\matrix [
		matrix of nodes,
		anchor=center,
		font=\scriptsize,
		line join = round,
		line cap=round
		] at (legend) {
			estimator product & goal error & \\
			\ref{Figure3_left:EstimatorOne} & \ref{Figure3_left:quasiErrorOne} &$p = 1$\\
			\ref{Figure3_left:EstimatorTwo} & \ref{Figure3_left:quasiErrorTwo} & $p = 2$\\
			\ref{Figure3_left:EstimatorThree} & \ref{Figure3_left:quasiErrorThree} & $p = 3$\\
		};
\end{tikzpicture}
		\hfil
		\pgfplotstableread[col sep = comma]{
	ndofs,  nElems, combinedEstimator,  GoalError,time_cumulative,relative_time, work
	  41,     126,      2.31414e-02,3.84768e-03,    1.07724e+00,  2.15449e-02,5.00000e+01
	  73,     242,      1.47569e-02,1.52268e-03,    1.27295e+00,  1.15723e-02,1.60000e+02
	 137,     510,      8.54214e-03,6.45762e-04,    1.41798e+00,  6.62606e-03,3.74000e+02
	 278,    1052,      4.39263e-03,2.82059e-04,    1.47635e+00,  3.15459e-03,8.42000e+02
	 558,    2130,      2.30667e-03,1.26305e-04,    1.57169e+00,  1.58436e-03,1.83400e+03
	1115,    4348,      1.16646e-03,5.71738e-05,    1.62370e+00,  7.98280e-04,3.86800e+03
	2244,    9050,      5.89691e-04,2.66693e-05,    1.69571e+00,  4.02589e-04,8.08000e+03
	4622,   18072,      2.90811e-04,1.22185e-05,    1.80404e+00,  2.03616e-04,1.69400e+04
	9174,   37713,      1.48216e-04,5.84359e-06,    1.98274e+00,  1.11390e-04,3.47400e+04
   19063,   76722,      7.18022e-05,2.69668e-06,    2.30891e+00,  6.18944e-05,7.20440e+04
   38644,  156191,      3.58197e-05,1.31114e-06,    2.97516e+00,  3.90646e-05,1.48204e+05
   78511,  313751,      1.76320e-05,6.20380e-07,    4.30442e+00,  2.77054e-05,3.03568e+05
  157451,  634849,      8.86721e-06,3.08135e-07,    7.31765e+00,  2.34087e-05,6.16172e+05
  318270, 1272065,      4.36095e-06,1.47944e-07,    1.35168e+01,  2.13480e-05,1.24933e+06
  637198, 2558227,      2.19933e-06,7.46053e-08,    2.66353e+01,  2.09770e-05,2.51907e+06
 1280829, 5106128,      1.08454e-06,3.61679e-08,    5.23958e+01,  2.05088e-05,5.07387e+06
 2555433,10221606,      5.49089e-07,1.83987e-08,    1.04429e+02,  2.04706e-05,1.01753e+07
 5114242,10221606,      2.71280e-07,8.97122e-09,    1.42428e+02,  1.39434e-05,2.03900e+07
}{\ComplexityOne}%
\pgfplotstableread[col sep = comma]{
	ndofs,  nElems,combinedEstimator,  GoalError,time_cumulative,relative_time, work
	 145,     104,      2.99996e-03,4.85034e-04,    1.07478e-01,  4.75566e-04,5.65000e+02
	 229,     176,      1.14929e-03,1.51594e-04,    3.01815e-01,  7.98452e-04,9.43000e+02
	 383,     310,      4.48363e-04,4.77986e-05,    3.82502e-01,  5.92108e-04,1.58900e+03
	 659,     446,      1.45652e-04,1.50284e-05,    4.68814e-01,  4.02070e-04,2.75500e+03
	 935,     722,      5.74064e-05,4.86716e-06,    5.34872e-01,  3.14261e-04,4.45700e+03
	1509,    1122,      2.37864e-05,1.57423e-06,    6.20398e-01,  2.24619e-04,7.21900e+03
	2327,    1984,      8.04252e-06,5.54883e-07,    7.31099e-01,  1.69001e-04,1.15450e+04
	4063,    3848,      2.62902e-06,1.97043e-07,    8.95769e-01,  1.15583e-04,1.92950e+04
	7851,    7009,      7.82524e-07,7.45983e-08,    1.20223e+00,  7.96917e-05,3.43810e+04
   14216,   11906,      1.97669e-07,1.97961e-08,    1.83029e+00,  6.62092e-05,7.58470e+04
   24117,   22924,      6.88793e-08,8.74250e-09,    2.62110e+00,  5.57468e-05,1.22865e+05
   46223,   37807,      1.75803e-08,2.82370e-09,    4.22105e+00,  4.64106e-05,2.59290e+05
   76144,   73870,      6.51940e-09,1.35305e-09,    6.71556e+00,  4.47191e-05,4.09462e+05
  148457,  127726,      1.67137e-09,4.68765e-10,    1.20571e+01,  4.10037e-05,8.50537e+05
  256403,  248864,      5.10332e-10,1.60182e-10,    2.17085e+01,  4.26487e-05,1.61405e+06
  499115,  441210,      1.21067e-10,3.94533e-11,    4.22437e+01,  4.25549e-05,3.59942e+06
  884067,  815212,      3.93167e-11,1.38685e-11,    7.77253e+01,  4.41233e-05,6.24174e+06
 1632989, 1562897,      1.23394e-11,4.93028e-12,    1.43443e+02,  4.40587e-05,1.11253e+07
 3128892, 2746491,      2.90413e-12,1.37136e-12,    2.83767e+02,  4.54361e-05,2.36161e+07
 5497874, 2746491,      1.00164e-12,6.48438e-13,    4.77237e+02,  4.34793e-05,4.00804e+07
}{\ComplexityTwo}%
\pgfplotstableread[col sep = comma]{
	ndofs,  nElems,combinedEstimator,  GoalError,time_cumulative,relative_time, work
	 313,     104,      1.76177e-04,3.46968e-05,    2.77032e-01,  5.22702e-04,3.18000e+03
	 499,     171,      6.45340e-05,1.07525e-05,    3.79052e-01,  4.31722e-04,4.05800e+03
	 808,     267,      2.71487e-05,3.31175e-06,    5.13490e-01,  3.50266e-04,5.52400e+03
	1255,     367,      8.87150e-06,1.05336e-06,    6.00886e-01,  2.61255e-04,7.82400e+03
	1714,     428,      2.33261e-06,1.89897e-07,    7.34921e-01,  2.30962e-04,1.25970e+04
	1993,     518,      9.91448e-07,6.54185e-08,    8.47784e-01,  2.27776e-04,1.63190e+04
	2404,     762,      4.19770e-07,2.39479e-08,    9.65629e-01,  2.13635e-04,2.08390e+04
	3520,    1128,      1.69137e-07,9.32184e-09,    1.13130e+00,  1.69356e-04,2.75190e+04
	5188,    1534,      6.11080e-08,3.95088e-09,    1.37915e+00,  1.38859e-04,3.74510e+04
	7048,    2618,      2.01365e-08,1.75986e-09,    1.71713e+00,  1.27007e-04,5.09710e+04
   11959,    4082,      4.51882e-09,5.59880e-10,    2.19206e+00,  9.44447e-05,8.57860e+04
   18628,    5546,      1.21680e-09,1.91101e-10,    2.78143e+00,  7.67840e-05,1.40122e+05
   25258,    7752,      3.58910e-10,6.68841e-11,    3.55788e+00,  7.21446e-05,2.14096e+05
   35230,   12322,      1.20665e-10,3.33714e-11,    4.42298e+00,  6.40269e-05,2.83176e+05
   55858,   17408,      3.26592e-11,1.17908e-11,    5.88676e+00,  5.34752e-05,4.48302e+05
   78859,   24628,      1.12361e-11,4.14505e-12,    8.13363e+00,  5.22626e-05,6.81747e+05
  111475,   37959,      3.98502e-12,2.05838e-12,    1.08345e+01,  4.91679e-05,9.02105e+05
  171559,   56639,      1.02501e-12,7.09925e-13,    1.55623e+01,  4.57516e-05,1.41233e+06
  255829,   86892,      3.72847e-13,3.45583e-13,    2.17354e+01,  4.27990e-05,1.92018e+06
  392257,  126585,      7.89995e-14,7.50702e-14,    3.32486e+01,  4.26512e-05,3.47927e+06
  571045,  174324,      1.98672e-14,1.04691e-14,    4.95826e+01,  4.36296e-05,5.75216e+06
  786175,  265274,      6.68642e-15,4.81733e-15,    7.07131e+01,  4.51700e-05,8.10038e+06
 1195927,  422902,      2.68592e-15,1.61451e-14,    1.00333e+02,  4.21024e-05,1.04835e+07
 1905679,  565556,      5.11277e-16,4.43812e-14,    1.64494e+02,  4.32778e-05,1.99857e+07
 2548006,  878609,      1.61276e-16,6.45838e-14,    2.42529e+02,  4.77044e-05,2.76117e+07
 3957715, 1294874,      6.68162e-17,1.13002e-13,    3.39430e+02,  4.29684e-05,3.55112e+07
 5831854, 1294874,      1.62050e-17,1.85827e-13,    5.05381e+02,  4.34026e-05,5.87993e+07
}{\ComplexityThree}%

	\begin{tikzpicture}[scale=1]
		\begin{loglogaxis}[xlabel = {cumulative time [s]\vphantom{$\sum_{| 
		\ell', k', j' | \le | \ell, \underline{k}, \underline{j} |} {\rm dim} 
		\, \mathcal{X}_{\ell'}$}}, 
			ylabel={error},
			ymin = 1e-18,
			each nth point=2,line join = round,
			ymajorgrids      = true,%
			grid style={densely dotted, semithick},
			line cap=round]
			
			\coordinate (legend) at (axis description cs:0.33,0.17);
			\addplot [pyBlue,thick,mark=*,mark options={solid, 
				fill=pyBlue!50!white, scale = 1.1}, solid] table 
				[x={time_cumulative},  
			y={combinedEstimator}]{\ComplexityOne}; \label{Figure3_right:EstimatorOne}
			\addplot [pyRed, first = 27, thick,mark=otimes*,mark 
			options={solid,fill=pyRed!50!white, scale = 1.1}, solid] table 
			[x={time_cumulative}, y={combinedEstimator}]{\ComplexityTwo}; 
			\label{Figure3_right:EstimatorTwo}
			\addplot [pyGreen, first = 28, thick,mark=oplus*,mark 
			options={solid,fill=pyGreen!50!white, scale = 1.1}, solid] table 
			[x={time_cumulative}, y={combinedEstimator}]{\ComplexityThree}; 
			\label{Figure3_right:EstimatorThree}
			
			\addplot [pyOrange,thick,mark=diamond*,mark options={solid, 
				fill=pyOrange!50!white, scale = 1.1}, solid] 
			table [x={time_cumulative},  
			y={GoalError}]{\ComplexityOne}; \label{Figure3_right:quasiErrorOne}
			\addplot [pyPurple, first = 27, thick,mark=halfdiamond*,mark 
			options={solid,fill=pyPurple!50!white, scale = 1.1}, 
			solid] 
			table 
			[x={time_cumulative}, y={GoalError}]{\ComplexityTwo}; 
			\label{Figure3_right:quasiErrorTwo}
			\addplot [pyYellow, , first = 22, thick,mark=halfsquare*,mark 
			options={solid,fill=pyYellow!50!white, scale = 1.1}, solid] 
			table [x={time_cumulative}, y={GoalError}]{\ComplexityThree}; 
			\label{Figure3_right:quasiErrorThree}
			\logLogSlope[reference]{0.45}{0.75}{0.3}{-1}{4pt}
			\logLogSlope[reference]{0.6}{0.52}{0.27}{-2}{4pt}
			\logLogSlope[reference]{0.65}{0.23}{0.25}{-3}{-5pt}
		\end{loglogaxis}
		\matrix [
		matrix of nodes,
		anchor=center,
		font=\scriptsize,
		line join = round,
		line cap=round
		] at (legend) {
			estimator product & goal error & \\
			\ref{Figure3_right:EstimatorOne} & \ref{Figure3_right:quasiErrorOne} &$p = 1$\\
			\ref{Figure3_right:EstimatorTwo} & \ref{Figure3_right:quasiErrorTwo} & $p = 2$\\
			\ref{Figure3_right:EstimatorThree} & \ref{Figure3_right:quasiErrorThree} & $p = 3$\\
		};
	\end{tikzpicture}
	}
	\hspace{-0.4cm}
	\caption{Convergence history plot of 
	estimator product 
	\(\eta_\ell(u_\ell^{\mm, \nn}) \, 
		\zeta_\ell(z^{\underline{\mu}, \underline{\nu}})\) 
		indicated by bullets and goal error 
	from~\eqref{eq:goalError:Estimate} indicated by 
	diamonds with respect to the 
	cumulative 
	computational work (left) 
	and with respect to the cumulative computational time (right) for the benchmark 
	problem~\eqref{eq:singularity_goal}.\label{fig:optimality}}
\end{figure}
\begin{figure}
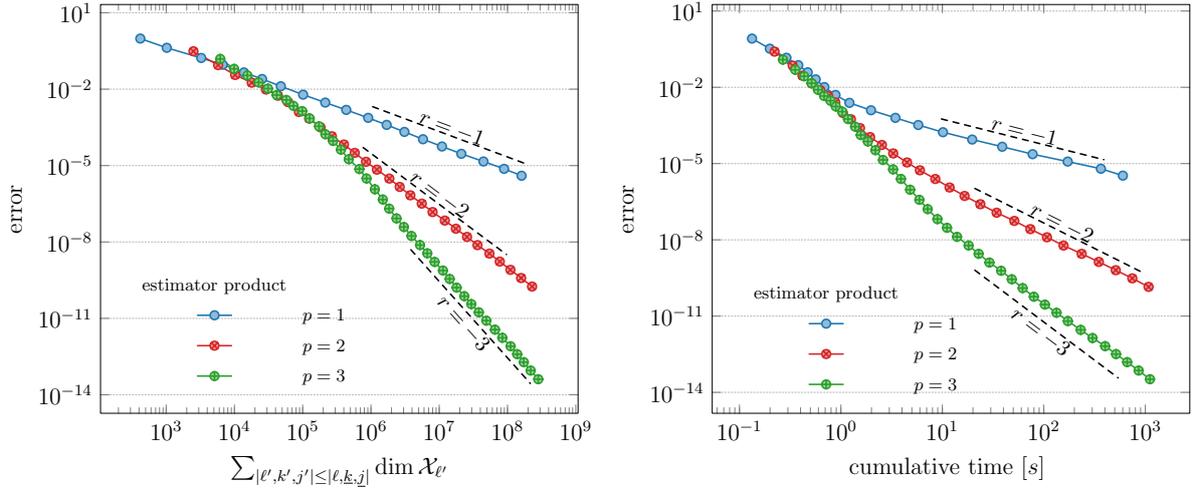

	\resizebox{\textwidth}{!}{
		\input{numerics/geometricSingularity/convergence_complexity}
		\hfil
		\input{numerics/geometricSingularity/estimator_time}
	}
	\hspace{-0.4cm}
	\caption{Convergence history plot of 
		estimator product 
		\(\eta_\ell(u_\ell^{\mm, \nn}) \, 
			\zeta_\ell(z^{\underline{\mu}, \underline{\nu}})\)  with respect 
			to the cumulative 
		computational 
		cost (left) and the cumulative 
		computational time (right) for the benchmark 
		problem~\eqref{eq:geometric_singularity}.\label{fig:Zshape_convergence}}
\end{figure}
\begin{figure}
	\centering
	\resizebox{0.55\textwidth}{!}{
		\input{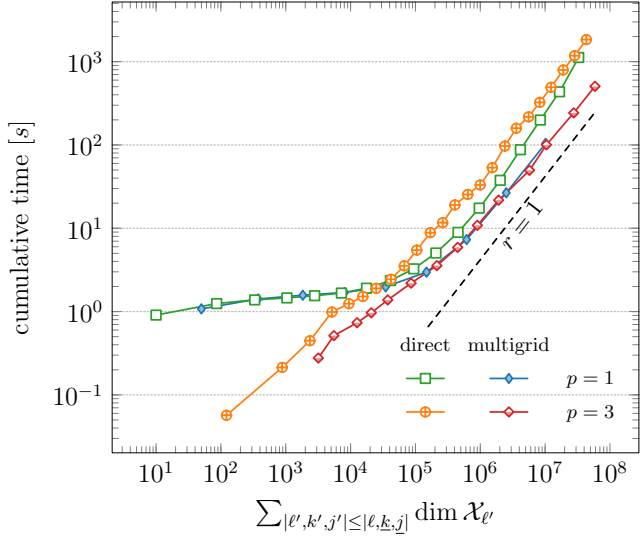}
	}
	\caption{Comparison of cumulative time of the local 
	multigrid solver with the \textsc{Matlab} built-in direct solver 
	\texttt{mldivide} with respect to the 
	cumulative computational cost for the benchmark 
	problem~\eqref{eq:geometric_singularity}.\label{fig:timing}}
\end{figure}

	\begin{figure}
	\resizebox{\textwidth}{!}{
\pgfplotstableread[col sep =
     comma]{
     ndofs, combinedEstimator, work, jPrimal, jDual
     41,0.02314,164,2,2
     64,0.017263,356,2,1
     93,0.013565,542,1,1
     136,0.0089772,814,1,1
     218,0.0060677,1250,1,1
     285,0.0044292,1820,1,1
     456,0.0028954,2732,1,1
     658,0.0019986,4048,1,1
     987,0.0013084,6022,1,1
     1509,0.00087928,9040,1,1
     2302,0.00057421,13644,1,1
     3606,0.00036821,20856,1,1
     5220,0.00025243,31296,1,1
     8070,0.00016145,47436,1,1
     13084,0.00010323,73604,1,1
     18673,7.0037e-05,1.1095e+05,1,1
     28700,4.5331e-05,1.6835e+05,1,1
     45494,2.9799e-05,2.5934e+05,1,1
     66484,1.9707e-05,3.9231e+05,1,1
     1.0217e+05,1.283e-05,5.9665e+05,1,1
     1.56e+05,8.6083e-06,9.0865e+05,1,1
     2.4004e+05,5.5648e-06,1.3887e+06,1,1
     3.4634e+05,3.7968e-06,2.0814e+06,1,1
     5.1006e+05,2.5333e-06,3.1015e+06,1,1
     8.2768e+05,1.6501e-06,4.7569e+06,1,1
     1.1762e+06,1.1145e-06,7.1094e+06,1,1
     1.7585e+06,7.3515e-07,1.0626e+07,1,1
     2.7611e+06,4.949e-07,1.6148e+07,1,1
     4.0556e+06,3.2555e-07,2.426e+07,1,1
}{\OrderOne}%
\pgfplotstableread[col sep = comma]{
     ndofs, combinedEstimator, work, jPrimal, jDual
     466,7.2536e-05,9223,1,2
     646,3.8983e-05,11161,1,2
     880,1.9053e-05,13801,1,2
     1120,8.7334e-06,17161,1,2
     1324,4.1114e-06,21133,1,2
     1606,1.6094e-06,25951,1,2
     1927,8.2623e-07,31732,1,2
     2419,3.3736e-07,36570,1,1
     2833,1.651e-07,42236,1,1
     3580,8.2397e-08,49396,1,1
     4849,3.9859e-08,59094,1,1
     6034,1.8478e-08,71162,1,1
     8050,8.8317e-09,87262,1,1
     10606,3.3767e-09,1.0847e+05,1,1
     13582,1.5105e-09,1.3564e+05,1,1
     18868,6.6882e-10,1.7337e+05,1,1
     23431,3.0837e-10,2.2024e+05,1,1
     31639,1.3532e-10,2.8351e+05,1,1
     40168,5.4534e-11,3.6385e+05,1,1
     55852,2.3484e-11,4.7555e+05,1,1
     74947,1.0547e-11,6.2545e+05,1,1
     92419,4.8443e-12,8.1029e+05,1,1
     1.2457e+05,2.1146e-12,1.184e+06,1,2
     1.5692e+05,8.8752e-13,1.4978e+06,1,1
     2.1504e+05,3.9938e-13,1.9279e+06,1,1
     2.8435e+05,1.9034e-13,2.781e+06,1,2
     3.5304e+05,8.6053e-14,3.487e+06,1,1
     4.7027e+05,3.8595e-14,4.4276e+06,1,1
     6.019e+05,1.6221e-14,5.6314e+06,1,1
     8.0253e+05,7.4316e-15,7.2365e+06,1,1
     1.0506e+06,3.6497e-15,9.3377e+06,1,1
     1.3176e+06,1.706e-15,1.329e+07,1,2
     1.7249e+06,7.821e-16,1.674e+07,1,1
     2.2187e+06,3.4361e-16,2.1178e+07,1,1
     2.5482e+06,1.749e-16,2.6274e+07,1,1
     3.2902e+06,9.615e-17,3.2854e+07,1,1
     4.1561e+06,5.2051e-17,4.1167e+07,1,1
}{\OrderThree}%

\begin{tikzpicture}[scale=1]
     \begin{semilogxaxis}[xlabel={$\sum_{| \ell', k', j' | \le | \ell, \underline{k},
                                   \underline{j} |} {\rm
                                   dim} \, \mathcal{X}_{\ell'}$},
               ylabel={total number of solver steps},
               ytick = {1, 2},
               ymajorgrids      = true,%
               grid style={densely dotted, semithick},
               each nth point=2,line join = round,
               line cap=round]

          \coordinate (legend) at (axis description cs:0.5,0.2);
          \addplot [pyBlue,thick,mark=*,mark
          options={solid,fill=pyBlue!50!white, scale = 1.1}, only marks,
          solid] table [x={work},
                    y={jPrimal}] {\OrderOne}; \label{primalOrderOne}
          \addplot [pyRed, thick, mark=otimes*,mark
          options={solid,fill=pyRed!50!white, scale = 1.1}, only marks, solid]
          table
               [x={work}, y={jPrimal}] {\OrderThree}; \label{primalOrderThree}

          \addplot [pyOrange,thick,mark=diamond*,mark
          options={solid,fill=pyOrange!50!white, scale = 1.1}, only marks, solid] table [x={work},
                    y={jDual}] {\OrderOne}; \label{dualOrderOne}
          \addplot [pyPurple, thick, mark=halfdiamond*,mark
          options={solid,fill=pyPurple!50!white, scale = 1.1}, only marks, solid]
          table
               [x={work}, y={jDual}] {\OrderThree}; \label{dualOrderThree}

     \end{semilogxaxis}

\end{tikzpicture}
	\hfil
\pgfplotstableread[col sep =
  comma]{
  ndofs, combined_estimator, complexity, jPrimal, jDual
  24,1.8008,72,2,1
  37,0.95787,516,7,5
  49,0.43293,859,3,4
  80,0.28926,1019,1,1
  140,0.16853,1299,1,1
  232,0.11754,1763,1,1
  345,0.077335,2453,1,1
  571,0.047594,3595,1,1
  948,0.028799,5491,1,1
  1553,0.017914,8597,1,1
  2514,0.01124,13625,1,1
  4088,0.007142,21801,1,1
  6631,0.0044589,35063,1,1
  10598,0.0028085,56259,1,1
  16937,0.0017677,1.0707e+05,2,1
  27024,0.0010978,1.6112e+05,1,1
  43564,0.00068909,2.4825e+05,1,1
  69850,0.00043035,5.2765e+05,2,2
  1.1016e+05,0.00027206,9.6829e+05,2,2
  1.7649e+05,0.00017238,1.4978e+06,1,2
  2.8078e+05,0.00010848,2.6209e+06,2,2
  4.3794e+05,6.8811e-05,4.3726e+06,2,2
  6.9374e+05,4.3766e-05,7.1476e+06,2,2
  1.0844e+06,2.7303e-05,1.4738e+07,3,4
  1.652e+06,1.7442e-05,2.7954e+07,4,4
  2.5069e+06,1.1778e-05,4.5503e+07,3,4
  3.8568e+06,7.5469e-06,7.6357e+07,4,4
  5.8347e+06,5.1273e-06,9.9696e+07,2,2
}{\OrderOne}%
\pgfplotstableread[col sep =
  comma]{
  ndofs, combined_estimator, complexity, jPrimal, jDual
  154,1.8713,616,2,2
  271,1.2805,2513,4,3
  412,0.09837,7045,5,6
  580,0.032616,10525,4,2
  742,0.017701,12751,2,1
  1015,0.0098268,16811,2,2
  1273,0.0053682,24449,4,2
  1483,0.0033873,28898,2,1
  1864,0.0019709,45674,6,3
  2182,0.0010859,52220,2,1
  2476,0.00075083,69552,4,3
  2926,0.0005213,81256,2,2
  3463,0.00031681,1.055e+05,4,3
  4093,0.00018883,1.2187e+05,2,2
  4783,0.00010846,1.697e+05,6,4
  5503,6.1647e-05,2.3023e+05,6,5
  6499,3.6578e-05,2.8222e+05,4,4
  7540,1.9467e-05,3.5762e+05,4,6
  8668,1.1923e-05,4.5297e+05,6,5
  10462,6.0075e-06,5.9944e+05,7,7
  12745,3.1902e-06,7.6512e+05,7,6
  15031,1.6952e-06,9.1544e+05,4,6
  18733,8.7129e-07,1.1215e+06,5,6
  23482,4.2062e-07,1.4033e+06,6,6
  28873,2.0709e-07,1.7498e+06,6,6
  35998,1.0849e-07,2.1097e+06,4,6
  45574,4.8791e-08,2.6566e+06,6,6
  56758,2.4461e-08,3.3377e+06,6,6
  70903,1.2555e-08,4.1886e+06,6,6
  88876,6.461e-09,5.2551e+06,6,6
  1.0875e+05,3.3219e-09,6.56e+06,6,6
  1.3659e+05,1.698e-09,8.1992e+06,6,6
  1.702e+05,8.6582e-10,1.0582e+07,7,7
  2.0845e+05,4.8512e-10,1.2041e+07,3,4
  2.6192e+05,2.4241e-10,1.4922e+07,5,6
  3.2712e+05,1.2305e-10,1.8848e+07,6,6
  4.0053e+05,6.4574e-11,2.4455e+07,7,7
  4.896e+05,3.6956e-11,2.7882e+07,3,4
  6.1939e+05,1.7912e-11,3.5315e+07,6,6
  7.582e+05,9.5963e-12,4.593e+07,7,7
  9.1836e+05,5.3936e-12,5.6032e+07,6,5
  1.1226e+06,3.0552e-12,6.838e+07,6,5
  1.375e+06,1.6492e-12,8.488e+07,6,6
  1.6531e+06,9.2138e-13,1.0802e+08,7,7
  1.9841e+06,5.8261e-13,1.1596e+08,2,2
  2.4378e+06,3.0232e-13,1.4278e+08,5,6
  2.9408e+06,1.6811e-13,1.7807e+08,6,6
  3.4964e+06,1.0601e-13,1.9205e+08,2,2
  4.2544e+06,5.6575e-14,2.3885e+08,5,6
}{\OrderThree}%

\begin{tikzpicture}[scale=1]
  \begin{semilogxaxis}[xlabel={$\sum_{| \ell', k', j' | \le | \ell, \underline{k},
              \underline{j} |} {\rm
              dim} \, \mathcal{X}_{\ell'}$},
      ylabel={total number of solver steps},
      ymajorgrids      = true,%
      grid style={densely dotted, semithick},
      each nth point=2,line join = round,
      line cap=round]

    \coordinate (legend) at (axis description cs:1.15, 0.55);
    \addplot [pyBlue,thick,mark=*,mark
    options={solid,fill=pyBlue!50!white, scale = 1.1}, only marks,
    solid] table [x={work},
        y={jPrimal}] {\OrderOne}; \label{primalOrderOne}
    \addplot [pyRed, thick, mark=otimes*,mark
    options={solid,fill=pyRed!50!white, scale = 1.1}, only marks, solid]
    table
      [x={work}, y={jPrimal}] {\OrderThree}; \label{primalOrderThree}

    \addplot [pyOrange,thick,mark=diamond*,mark
    options={solid,fill=pyOrange!50!white, scale = 1.1}, only marks, solid] table [x={work},
        y={jDual}] {\OrderOne}; \label{dualOrderOne}
    \addplot [pyPurple, thick, mark=halfdiamond*,mark
    options={solid,fill=pyPurple!50!white, scale = 1.1}, only marks, solid]
    table
      [x={work}, y={jDual}] {\OrderThree}; \label{dualOrderThree}

  \end{semilogxaxis}

  \matrix [
    matrix of nodes,
    anchor=south,
    font=\scriptsize,
    line join = round,
    line cap=round
  ] at (legend) {
    primal                 & \\
    \ref{primalOrderOne}   &
    $p=1$                    \\
    \ref{primalOrderThree} &
    $p=3$                    \\
    \medskip
    dual                   & \\
    \ref{dualOrderOne}     &
    $p=1$                    \\
    \ref{dualOrderThree}   &
    $p=3$                    \\
  };
\end{tikzpicture}
	}
    \caption{Number of total solver steps
        \(| \ell, \underline{m}, \underline{n}| - |\ell, 0, 0|\) resp.\ \(| \ell, \underline{\mu}, \underline{\nu}| - |\ell, 0, 0|\) on each mesh level
        for the benchmark problems~\eqref{eq:singularity_goal} (left) and 
		\eqref{eq:geometric_singularity} (right).\label{fig:iterations}}
\end{figure}
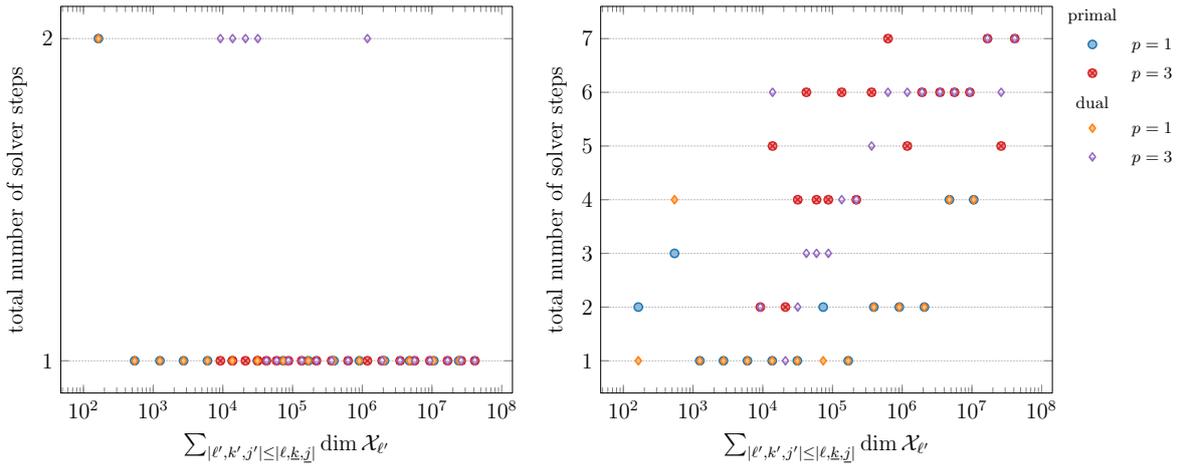

\begin{table}[htbp!]
	\resizebox{\columnwidth}{!}{
		\begin{tabular}{c ccccc ccccc ccccc}
	\toprule
	\(\cdot 10^{-7}\) & \multicolumn{5}{c}{\(\theta = 0.1\)} & \multicolumn{5}{c}{\(\theta = 0.3\)} & \multicolumn{5}{c}{\(\theta = 0.5\)} \\
	\cmidrule(lr){2-6} \cmidrule(lr){7-11} \cmidrule(lr){12-16}
	\diagbox{$\lambda_{\rm alg}$}{$\lambda_{\rm sym}$} & 0.1 & 0.3 & 0.5 & 0.7 & 0.9 & 0.1 & 0.3 & 0.5 & 0.7 & 0.9 & 0.1 & 0.3 & 0.5 & 0.7 & 0.9 \\
	\midrule
	0.1 & 38.7 & 33.4 & 29.6 & \colorbox{row!40}{22.1} & 24.4 & 10.2 & 5.12 & 4.90 & 4.83 & \colorbox{row!40}{4.74} & 6.18 & \colorbox{row!40}{4.48} & 4.66 & 4.89 & 5.25 \\
	0.3 & 36.2 & 24.7 & 24.5 & \colorbox{best!40}{21.8} & \colorbox{col!40}{23.1} & 7.28 & 4.98 & 3.53 & 3.27 & \colorbox{best!40}{3.26} & \colorbox{row!40}{4.18} & 4.54 & 4.79 & 5.01 & 5.13 \\
	0.5 & 24.3 & 24.7 & 24.7 & \colorbox{row!40}{23.4} & 23.6 & 5.84 & 3.64 & 3.39 & \colorbox{row!40}{3.27} & 3.37 & 3.41 & 2.71 & 2.52 & \colorbox{row!40}{2.49} & 2.68 \\
	0.7 & 24.1 & 24.8 & 23.8 & \colorbox{row!40}{22.2} & 24.0 & 4.95 & 3.59 & 3.30 & \colorbox{best!40}{3.25} & 3.42 & \colorbox{col!40}{2.74} & 2.35 & \colorbox{col!40}{2.41} & \colorbox{best!40}{2.24} & 2.46 \\
	0.9 & \colorbox{col!40}{23.5} & \colorbox{col!40}{24.6} & \colorbox{best!40}{22.3} & 24.4 & 23.8 & \colorbox{col!40}{4.90} & \colorbox{col!40}{3.58} & \colorbox{col!40}{3.29} & \colorbox{row!40}{3.26} & 3.41 & 2.81 & \colorbox{col!40}{2.30} & 2.43 & \colorbox{row!40}{2.27} & \colorbox{col!40}{2.41} \\
	\midrule
	& \multicolumn{5}{c}{\(\theta = 0.7\)} & \multicolumn{5}{c}{\(\theta = 0.8\)} & \multicolumn{5}{c}{\(\theta = 0.9\)} \\
	\cmidrule(lr){2-6} \cmidrule(lr){7-11} \cmidrule(lr){12-16}
	0.1 & 5.82 & \colorbox{row!40}{5.18} & 5.43 & 5.40 & 5.93 & 8.53 & \colorbox{row!40}{6.10} & 7.31 & 6.67 & 7.77 & 11.6 & \colorbox{row!40}{8.86} & 9.12 & 9.87 & 9.97 \\
	0.3 & \colorbox{row!40}{4.65} & 4.86 & 5.35 & 5.98 & 6.67 & 6.27 & \colorbox{row!40}{5.92} & 7.20 & 7.46 & 7.57 & 8.62 & \colorbox{row!40}{8.40} & 9.27 & 10.6 & 11.5 \\
	0.5 & 3.69 & 2.89 & \colorbox{row!40}{2.88} & 2.95 & 3.13 & 5.09 & \colorbox{row!40}{3.61} & 3.66 & 3.63 & 3.66 & 7.27 & 5.32 & \colorbox{row!40}{4.84} & 4.93 & 5.12 \\
	0.7 & 2.99 & \colorbox{row!40}{2.56} & \colorbox{col!40}{2.64} & \colorbox{col!40}{2.62} & \colorbox{col!40}{2.89} & 3.75 & 3.12 & 3.23 & \colorbox{best!40}{3.03} & \colorbox{col!40}{3.11} & \colorbox{col!40}{4.58} & \colorbox{best!40}{3.95} & \colorbox{col!40}{4.04} & 4.43 & 4.79 \\
	0.9 & \colorbox{col!40}{2.89} & \colorbox{best!40}{2.49} & 2.65 & 2.66 & \colorbox{col!40}{2.89} & \colorbox{col!40}{3.79} & \colorbox{best!40}{3.11} & \colorbox{col!40}{3.19} & 3.13 & 3.27 & 4.67 & \colorbox{row!40}{4.06} & 4.16 & \colorbox{col!40}{4.35} & \colorbox{col!40}{4.61} \\
	\bottomrule
\end{tabular}

	}
	\vspace{0.1cm}
	\caption{Optimal selection of parameters with 
		respect to the cumulative computational costs
        (overall computation time in seconds) for the 
		experiment~\eqref{eq:singularity_goal} with fixed polynomial 
		degree 
		\(p=2\) and \(\delta = 0.5\). For comparison,
        the table displays the value of 
		the weighted costs from~\eqref{eq:weighted_costs} (in \(10^{-7}\)) 
		with overall stopping 
		criterion $\eta_\ell(u_\ell^{\underline{m}, \underline{n}}) \, 
		\zeta_\ell(u_\ell^{\underline{\mu}, \underline{\nu}}) < 5 \cdot 
		10^{-10}$ for 
		various choices of 
		$\lambda_{\rm sym}$, $\lambda_{\rm alg}$, and $\theta$. For each 
		\(\theta\)-block, we mark the row-wise optimal values 
		in blue, the 
		column-wise optimal values in yellow, and in green if 
		both optimal values coincide.\label{tab:cost}
	}
\end{table}
\vspace{-0.4cm}

\medskip
\section{Summary}\label{section:conclusion}
In this work, we developed a cost-optimal goal-oriented adaptive finite element
    method for the efficient computation of the quantity of interest \(G(u^\star)\)
    with solution \(u^\star\) to the general second-order linear elliptic partial
    differential equation \eqref{eq:model_problem}. Since the current analysis of
    iterative algebraic solvers for nonsymmetric systems with optimal preconditioner
    only leads to contraction of the residual in a vector norm, we proposed a
    nested iterative solver for the primal and dual problem in parallel. The strategy consists
    of the Zarantonello iteration~\eqref{eq:Zarantonello-iterations} as an outer solver loop and
    an optimal multigrid solver for the arising SPD system as an innermost solver loop. In recent own work~\cite{fps2023},
    we have shown that the link between convergence rates with respect to the degrees
    of freedom and the total computational cost is full linear convergence of the quasi-error~\(\Eta_{\ell}^{k, j} \, \Zeta_{\ell}^{k, j}\).
    To this end, Theorem~\ref{lem:full_linear_convergence} shows that the proposed adaptive
    algorithm contracts (up to a multiplicative constant) the quasi-error product \(\Eta_{\ell}^{k, j} \, \Zeta_{\ell}^{k, j} \)
    in every step, independently of the algorithmic decision to employ mesh refinement, symmetrization, or the algebraic
    solver. A particular problem in the analysis is that the nested iterative solver procedure only guarantees contraction as long as \(1 \le k < \kk[\ell]\), whereas contraction for the final iterate is only guaranteed up to an estimator term (cf.~\eqref{eq2:inexact_Zarantonello_contraction}). Another difficulty arises from the nonsymmetric setting with a quasi-Pythagorean estimate~\eqref{eq:quasi_orthogonality} replacing the usual Pythagorean estimate.
    Therefore, the proof of Theorem~\ref{lem:full_linear_convergence} employs the equivalence of R-linear convergence and tail-summability of the
    quasi-error product~\(\Eta_{\ell}^{k, j} \, \Zeta_{\ell}^{k, j} \) and leads to mild restriction on the product \(\lamsym \, \lamalg\) of the involved solver stopping parameters.
        The key ingredients to cost-optimality are an
        adaptive mesh-refinement algorithm with optimal
        convergence rate with respect to the number of
        degrees of freedom (under the assumption of exact
        solution) and an algebraic solver for the
        linear system of equations that is contractive
        with respect to the underlying Sobolev norm.
        In this regard, the analysis in this paper may guide
        the generalization to conforming discretizations of
        vector-valued elliptic problems.
    Finally, the numerical experiments in
    Section~\ref{section:numerics} suggest that the proposed
    strategy allows for large stopping parameter in practice
    and that a larger choice is beneficial in terms of total
    runtime.
    Admittedly, the development of an optimal solver for the
    nonsymmetric problem~\eqref{eq:discrete_formulation}
    would allow to prove full linear convergence with an
    arbitrary selection of the stopping parameter.

\sloppy
\printbibliography
\end{document}